\documentclass[11pt,letterpaper]{amsart}

\usepackage[utf8]{inputenc}
\usepackage[usenames,dvipsnames]{color}
\usepackage{amsthm,amsfonts,amssymb,amsmath,amsxtra}
\usepackage[all]{xy}
\SelectTips{cm}{}
\usepackage{xr-hyper}
\usepackage[colorlinks=
citecolor=Black,
linkcolor=Red,
urlcolor=Blue]{hyperref}
\usepackage{verbatim}
\usepackage{caption}

\usepackage{mathrsfs}

\usepackage[symbol]{footmisc}

\usepackage{tabularx}
\usepackage{makecell}

\usepackage{booktabs}

\usepackage{tikz}

\makeatletter
\newcommand*\circled[2][1.6]{\tikz[baseline=(char.base)]{
    \node[shape=circle, draw, inner sep=1pt,
        minimum height={\f@size*#1},] (char) {\vphantom{WAH1g}#2};}}
\makeatother

\RequirePackage{xspace}
\RequirePackage{etoolbox}
\RequirePackage{varwidth}
\RequirePackage{enumitem}
\RequirePackage{tensor}
\RequirePackage{mathtools}
\RequirePackage{longtable}
\RequirePackage{multirow}

\newcommand{\supp}{\text {\rm supp}}

\newcommand{\sgn}{\text{ \rm sgn}}

\newcommand{\Id}{\textrm{id}}

\def\i{^{-1}}
\def\ge{\geqslant}
\def\le{\leqslant}
\def\<{\langle}
\def\>{\rangle}

\def\bb{\bold{b}}

\def\ad{{\rm{ad}}}

\def\PGL{{\rm PGL}}

\def\a{\alpha}
\def\b{\beta}
\def\g{\gamma}

\def\d{\delta}
\def\D{\Delta}

\def\e{\epsilon}

\def\o{\omega}

\def\s{\sigma}
\def\t{\tau}
\def\th{\theta}
\def\k{\kappa}
\def\l{{\lambda}}
\def\z{\zeta}

\def\tPhi{\tilde \Phi}

\def\ZZ{\mathbb Z}

\def\BB{\mathbb B}
\def\NN{\mathbb N}
\def\QQ{\mathbb Q}
\def\JJ{\mathbb J}
\def\FF{\mathbb F}
\def\RR{\mathbb R}

\def\SS{\mathbb S}
\def\kk{\bold{k}}

\def\ca{\mathcal A}

\def\cd{\mathcal D}
\def\ce{\mathcal E}

\def\ch{\mathcal H}

\def\co{\mathcal O}

\def\car{\mathcal R}

\def\cz{\mathcal Z}

\def\car{\mathcal R}

\def\tta{{\tilde \alpha}}

\def\tW{\tilde W}
\def\tw{{\tilde w}}

\def\fs{\mathfrak S}

\def\ad{{\rm ad}}

\def\dft{{\rm def}}
\def\Irr{{\rm Irr}}
\def\pr{{\rm pr}}

\def\rk{{\rm rk}}
\def\tp{{\rm top}}

\def\GL{{\rm GL}}
\def\GSp{{\rm GSp}}
\def\Res{{\rm Res}}

\def\dft{{\rm def}}
\def\GU{{\rm GU}}
\def\char{{\rm char}}
\def\der{{\rm der}}
\def\gen{{\rm gen}}

\def\ul{{\underline{\lambda}}}

\def\df{{\rm def}}

\def\bmu{{\mu_\bullet}}
\def\bl{{\lambda_\bullet}}

\def\bb{{b_\bullet}}

\def\bs{\sigma_\bullet}

\def\bbeta{{\eta_\bullet}}
\def\bup{{\upsilon_\bullet}}
\def\bg{\gamma_\bullet}

\def\bxi{{\xi_\bullet}}
\def\bo{{\omega_\bullet}}

\def\Gr{{\rm Gr}}
\def\MV{{\rm MV}}
\def\brF{{\breve F}}

\theoremstyle{plain}
\newtheorem{thm}{Theorem}[section]
\newtheorem{conj}{Conjecture}[section]

\newtheorem*{thm*}{Theorem}
 \newtheorem{prop}[thm]{Proposition}
 \newtheorem{lem}[thm]{Lemma}
 \newtheorem{cor}[thm]{Corollary}

\theoremstyle{definition}

\newtheorem{rmk}[thm]{Remark}

\theoremstyle{remark}

\newtheorem*{claim*}{Claim}

\begin{document}

\author{Sian Nie}
\address{Institute of Mathematics, Academy of Mathematics and Systems Science, Chinese Academy of Sciences, 100190, Beijing, China}
\email{niesian@amss.ac.cn}

%\thanks{This work is supported by the National Natural Science Foundation of China (No. 11621061 and No. 11501547) and by Key Research Program of Frontier Sciences, CAS, Grant No. QYZDB-SSW-SYS007.}

\title{Irreducible components of affine Deligne-Lusztig varieties}
\begin{abstract}
We determine the (top-dimensional) irreducible components (and their stabilizers in the Frobenius twisted centralizer group) of affine Deligne-Lusztig varieties in the affine Grassmannian of a reductive group, by constructing a natural map from the set of irreducible components to the set of Mirkovi\'{c}-Vilonen cycles. This in particular verifies a conjecture by Miaofen Chen and Xinwen Zhu.

%Nous d\'{e}terminons les composants irr\'{e}ductibles (de dimension sup\'{e}rieure) des vari\'{e}t\'{e}s Deligne-Lusztig affines dans la Grassmannienne affine d'un groupe r\'{e}ductif, par construire une carte naturelle de l'ensemble des composants irr\'{e}ductibles \`{a} l'ensemble des cycles Mirkovi\'{c}-Vilonen. Cela v\'{e}rifie en particulier une conjecture de Miaofen Chen et Xinwen Zhu.

\end{abstract}

\maketitle

\section*{Introduction}
\subsection{Background}
The notion of affine Deligne-Lusztig variety was first introduced by Rapoport in \cite{R}, which plays an important role in understanding geometric and arithmetic properties of Shimura varieties. Thanks to the uniformization theorem by Rapoport and Zink \cite{RZ}, the Newton strata of Shimura varieties can be described explicitly in terms of so-called Rapoport-Zink spaces, whose underlying spaces are special cases of affine Deligne-Lusztig varieties.

In \cite{KR} and \cite{R}, Kottwitz and Rapoport made several conjectures on basic properties of affine Deligne-Lusztig varieties. Most of them have been verified by a number of authors. We mention the works by Rapoport-Richartz \cite{RR}, Kottwitz \cite{Ko3}, Gashi \cite{Ga}, and He \cite{He12}, \cite{He} on the ``Mazur inequality" criterion of non-emptiness; the works by G\"{o}rtz-Haines-Kottwitz-Reuman \cite{GHKR}, He \cite{He2}, He-Yu \cite{HY}, Viehmann \cite{V}, Hamacher \cite{H}, and Zhu \cite{Z} on the dimension formula; the works by Hartl-Viehmann \cite{HaV1}, \cite{HaV2}, Mili\'{c}evi\'{c}-Viehmann \cite{MiVi}, and Hamacher \cite{H1} on the irreducible components; and the works by Viehmann \cite{V4}, Chen-Kisin-Viehmann \cite{CKV}, and the author \cite{N1} on the connected components in the hyperspecial case (see also \cite{HZ}, \cite{CN} for some partial results in the parahoric case). For a thorough survey we refer to the report \cite{He1}. These advances on affine Deligne-Lusztig varieties have found several interesting applications in arithmetic geometry. For example, the dimension formula leads to a proof by Hamacher \cite{H1} for the Grothendieck conjecture on the closure relations of Newton strata of Shimura varieties, and the description of connected components in \cite{CKV} plays an essential role in the proof by Kisin \cite{Ki} for the Langlands-Rapoport conjecture on mod $p$ points of Shimura varieties (see \cite{Zh}, \cite{Ho} for recent progresses).

%Understanding geometric and arithmetic properties of Shimura varieties has been a key problem in arithmetic geometry. An important approach is to study the special fiber of certain suitable integral model of the Shimura variety. For a PEL Shimura variety, the special fiber can be viewed as a moduli space of abelian varieties with additional structure. To such abelian varieties, we can attach their $p$-divisible groups which inherit corresponding additional structure. Then the Newton strata are the loci where the isogeny classes of the attached $p$-divisible groups are constant. Thanks to the uniformization theorem by Rapoport and Zink \cite{RZ}, we can describe the Newton strata explicitly in terms of so-called Rapoport-Zink spaces (see \cite{VW}, \cite{Man}, \cite{H1}), whose underlying space is called an affine Deligne-Lusztig variety \cite{R}. (see \cite{VW}, \cite{XZ}).

%Affine Deligne-Lusztig varieties are closely related to the Rapoport-Zink moduli spaces of $p$-divisible groups, and play an important role in the study of Shimura varieties. There has been an extensive study on affine Deligne-Lusztig varieties. However, many basic aspects of their geometric structure are not fully understood yet. and \cite{HV} for the current status of these topics.

\subsection{Main results}
This paper is concerned with the parametrization problem of top-dimensional irreducible components of affine Deligne-Lusztig varieties. The problem was first considered by Xiao and Zhu in \cite{XZ}, where they solved the unramified case in order to prove certain cases of the Tate conjecture for Shimura varieties. We will provide a complete parametrization in the general case.

To state the results, we introduce some notations. Let $F$ be a non-archimedean local field with residue field $\FF_q$. Let $\brF$ be the completion of the maximal unramified extension of $F$. Denote by $\co_F$ and $\co_\brF$ the valuation rings of $F$ and $\brF$ respectively. Let $\s$ be the Frobenius automorphism of $\brF / F$.

Let $G$ be a connected reductive group over $\co_F$. Fix $T \subseteq B \subseteq G$, where $T$ is a maximal torus and $B = T U$ is a Borel subgroup with unipotent radical $U$. Denote by $Y$ the cocharacter group of $T$, and by $Y^+$ the set of dominant cocharacters determined by $B$. Let $K = G(\co_\brF)$. Fix a uniformizer $t \in \co_F$ and set $t^\l = \l(t) \in G(\brF)$ for $\l \in Y$. Then we have the Cartan decomposition for the affine Grassmannian $$\Gr = \Gr_G = G(\brF) / K = \sqcup_{\mu \in Y^+} \Gr_\mu^\circ,$$ where $\Gr_\mu^\circ = K t^\mu K/K$. For $b \in G(\brF)$ and $\mu \in Y^+$, the attached affine Deligne-Lusztig set is defined by $$X_\mu(b) = X_\mu^G(b)=\{g \in G(\brF); g\i b \s(g) \in K t^\mu K\} / K,$$ which is a subscheme locally of finite type in the usual sense if $\char(F) > 0$, and in the sense of Bhatt-Scholze \cite{BS} and Zhu \cite{Z} if $\char(F) = 0$. By left multiplication it carries an action of the group $$\JJ_b=\JJ_b^G=\{g \in G(\brF); g\i b \s(g)=b\}.$$ Up to isomorphism, $X_\mu(b)$ only depends on the $\s$-conjugacy class $[b] = [b]_G$ of $b$. Thanks to Kottwitz \cite{Ko1}, $[b]$ is uniquely determined by two invariants: the Kottwitz point $\k_G(b) \in \pi_1(G)_\s = \pi_1(G) / (1-\s)(\pi_1(G))$ and the Newton point $\nu_G(b) \in Y_\RR = Y \otimes \RR$, see \cite[\S 2.1]{HV}. Then $X_\mu(b) \neq \emptyset$ if and only if $\k_G(t^\mu)=\k_G(b)$ and $\nu_G(b) \le \mu^\diamond$, where $\mu^\diamond$ denotes the $\s$-average of $\mu$, and $\le$ denotes the partial order on $Y_\RR$ such that $v \le v' \in Y_\RR$ if $v' - v$ is a non-negative linear combination of coroots in $B$. Moreover, in this case, its dimension is given by $$\dim X_\mu(b)=\<\rho_G, \mu-\nu_G(b)\> - \frac{1}{2} \dft_G(b),$$ where $\rho_G$ is the half-sum of roots of $B$ and $\dft_G(b)$ is the {\it defect} of $b$, see \cite[\S 1.9.1]{Ko2}. Let $\Irr^\tp X_\mu(b)$ denote the set of top-dimensional irreducible components of $X_\mu(b)$.

\

The first goal of this paper is to give an explicit description of the set $\JJ_b \backslash \Irr^\tp X_\mu(b)$ of $\JJ_b$-orbits of $\Irr^\tp X_\mu(b)$. We invoke a conjecture by Xinwen Zhu and Miaofen Chen which suggests a parametrization of $\JJ_b \backslash \Irr^\tp X_\mu(b)$ by certain Mirkovi\'{c}-Vilonen cycles.

Recall that Mirkovi\'{c}-Vilonen cycles are irreducible components of $S^\l \cap \Gr_\mu$ for $\mu \in Y^+$ and $\l \in Y$, where $S^\l = U(\brF) t^\l K/K$ and $\Gr_\mu = \overline{\Gr_\mu^\circ}$. We write $\MV_\mu = \sqcup_\l \MV_\mu(\l)$ with $\MV_\mu(\l) = \Irr (S^\l \cap \Gr_\mu)$ the set of irreducible components.

Let $\widehat G$ be the Langlands dual of $G$ defined over $\overline{\QQ}_l$ with $l \neq \char(\kk)$. Denote by $V_\mu = V_\mu^{\widehat G}$ the irreducible $\widehat G$-module of highest weight $\mu$. The crystal basis (or the canonical basis) $\BB_\mu = \BB_\mu^{\widehat G}$ of $V_\mu$ was first constructed by Lusztig \cite{Lu} and Kashiwara \cite{Kashiwara}. In \cite[Theorem 3.1]{BG}, Braverman and Gaitsgory proved that the set $\MV_\mu$ of Mirkovi\'{c}-Vilonen cycles admits a $\widehat G$-crystal structure and gives rise to a crystal basis of $V_\mu$ via the geometric Satake isomorphism \cite{MV}. In \cite[\S 3.3]{XZ}, Xiao and Zhu constructed a canonical isomorphism $\BB_\mu \cong \MV_\mu$ using Littelmann's path model \cite{Li}, which we denote by $\d \mapsto S_\d$. The advantage of using $\BB_\mu$ is that its $\widehat G$-crystal structure is given in a combinatorial way.

In \cite[\S 2.1]{HV}, Hamacher and Viehmann proved that, under the partial order $\le$, there is a unique maximal element $\ul_G(b)$ in the set $$\{\ul \in Y_\s = Y / (1-\s)Y; \ul = \k_G(b), \ul^\diamond \le \nu_G(b) \},$$ which is called ``the best integral approximation" of $\nu_G(b)$. Let $V_\mu(\ul_G(b))$ be the sum of $\l$-weight spaces $V_\mu(\l)$ with $\l = \ul_G(b) \in Y_\s$, whose basis in $\BB_\mu$ and $\MV_\mu$ is denoted by $\BB_\mu(\ul_G(b))$ and $\MV_\mu(\ul_G(b))$ respectively.

\begin{conj} [Chen-Zhu] \label{conj}
There exist natural bijections $$\JJ_b \backslash \Irr^\tp X_\mu(b) \cong \MV_\mu(\ul_G(b)) \cong \BB_\mu(\ul_G(b)).$$ In particular, $|\JJ_b \backslash \Irr^\tp X_\mu(b)| = \dim V_\mu(\ul_G(b))$.
\end{conj}\begin{rmk}
If $\char(F) > 0$, $X_\mu(b)$ is equi-dimensional by \cite{HaV2} and $\Irr^\tp X_\mu(b)$ coincides with the set of irreducible components of $X_\mu(b)$. If $\char(F) = 0$, the equi-dimensionality of $X_\mu(b)$ is not fully established, see \cite[Theorem 3.4]{HV}. However, $X_\mu(b)$ is always equi-dimensional if $\mu$ is minuscule.
\end{rmk}

\begin{rmk} If $\mu$ is minuscule and either $G$ is split or $b$ is superbasic, Conjecture \ref{conj} is proved by Hamacher and Viehmann \cite{HV} using the method of semi-modules, which originates in the work \cite{dJO} by de Jong and Oort. If $b$ is unramified, that is, $\dft_G(b)=0$, it is proved by Xiao and Zhu \cite{XZ} using the geometric Satake. In both cases, the authors obtained complete descriptions of $\Irr^\tp X_\mu(b)$.
\end{rmk}

\begin{rmk}
A complete description of $\Irr^\tp X_\mu(b)$ was also known for the case where $G$ is $\GL_n$ or $\GSp_{2n}$ and $\mu$ is minuscule, see \cite{V1} and \cite{V2}.
\end{rmk}

\begin{rmk}
If the pair $(G, \mu)$ is {\it fully Hodge-Newton decomposable} (see \cite{GHN}), $X_\mu(b)$ admits a nice stratification by classical Deligne-Lusztig varieties, whose index set and closure relations are encoded in the Bruhat-Tits building of $\JJ_b$. Such a stratification has important applications in arithmetic geometry, including the Kudla-Rapoport program \cite{kudla-rapoport}, \cite{kudla-rapoport-2} and Zhang's Arithmetic Fundamental Lemma \cite{Zhang}. We mention the works by Vollaard-Wedhorn \cite{VW}, Rapoport-Terstiege-Wilson \cite{Rapoport-Terstiege-Wilson}, Howard-Pappas \cite{howard-pappas}, \cite{howard-pappas2}, and G\"{o}rtz-He \cite{GH} for some of the typical examples.
\end{rmk}

Let $I \subseteq G(\brF)$ be the standard Iwahori subgroup associated to the triple $T \subseteq B \subseteq G$, see \S\ref{G-setup}. By Proposition \ref{superbasic}, for $b \in G(\brF)$, there is a unique standard Levi subgroup $T \subseteq M \subseteq G$ and a {\it superbasic} element $b_M$ of $M(\brF)$, unique up to $M(\brF)$-$\s$-conjugation, such that $[b_M] = [b]$ and $\nu_M(b_M) = \nu_G(b)$. Moreover, we may and do choose $b_M$ such that $b_M T(\brF) b_M\i = T(\brF)$ and $b_M I_M b_M\i = I_M$, where $I_M = M(\brF) \cap I$ is the standard Iwahori subgroup of $M(\brF)$. Take $b = b_M$. Let $P = M N$ be the standard parabolic subgroup with $N \subseteq U$ its unipotent radical. Using the Iwasawa decomposition $G(\brF) = N(\brF) M(\brF) K$ we have $$\Gr = N(\brF) M(\brF) K/K = \sqcup_{\l \in Y} N(\brF) I_M t^\l K/K.$$ For $\l \in Y$ let $\th^P_\l: N(\brF) I_M \to \Gr$ be the map given by $h \mapsto h t^\l K$.

Our first goal is to prove Conjecture \ref{conj}.
\begin{thm} \label{main}
Let $b$ and $M$ be as above. Then there exists a map $$\g = \g^G: \Irr^\tp X_\mu(b) \to \BB_\mu(\ul_G(b))$$ such that for $C \in \Irr^\tp X_\mu(b)$ we have $$\overline{ \{(h t^\l)\i b \s(h t^\l) K; h \in (\th^P_\l)\i(C) \} } = \e_\l^M \overline{ S_{\g(C)} },$$ where $\l$ is the unique cocharacter such that $N(\brF) I_M t^\l K/K \cap C$ is open dense in $C$ and $\e_\l^M$ is certain Weyl group element for $M$ associated to $\l$ (see \S\ref{def-epsilon}). Moreover, $\g$ factors through a bijection $$\JJ_b \backslash \Irr^\tp X_\mu(b) \cong \BB_\mu(\ul_G(b)).$$
\end{thm}

\begin{rmk} \label{numer}
The equality $|\JJ_b \backslash \Irr^\tp X_\mu(b)| = \dim V_\mu(\ul_G(b))$, which is the numerical version of Conjecture \ref{conj}, is proved by Rong Zhou and Yihang Zhu in \cite{ZZ} (even for the quasi-split case), and by the author in an earlier version of this paper, using different approaches.
\end{rmk}

\

It is a remarkable feature that the tensor product of two crystals bases is again a crystal basis. So there is a natural map $$\otimes: \BB_\bmu^{\widehat G^d} = \BB_{\mu_1}^{\widehat G} \times \cdots \times \BB_{\mu_d}^{\widehat G} \longrightarrow \BB_{\mu_1}^{\widehat G} \otimes \cdots \otimes \BB_{\mu_d}^{\widehat G}  \longrightarrow  \sqcup_\mu \BB_\mu^{\widehat G}(\ul_G(b)),$$ where the first map is given by taking the tensor product, and the second one is the canonical projection to highest weight $\widehat G$-crystals.

On the other hand, there is also a ``tensor structure" among affine Deligne-Lusztig varieties coming from the geometric Satake. Consider the product $G^d$ of $d$ copies of $G$ together with a Frobenius automorphism given by $$(g_1, g_2, \dots, g_d) \mapsto (g_2, \dots, g_d, \s(g_1)).$$ For $\bmu = (\mu_1, \dots, \mu_d) \in Y^d$ and $\bb = (1, \dots, 1, b) \in G^d(\brF)$ with $b \in G(\brF)$, we can define the affine Deligne-Lusztig variety $X_\bmu(\bb)$ in a similar way. We know (see Corollary \ref{compare}) that the projection $\Gr^d \to \Gr$ to the first factor induces a map $$\pr: \Irr^\tp X_\bmu(\bb) \to \sqcup_\mu \Irr^\tp X_\mu(b),$$ which serves as the functor of taking tensor products.

Our second main result shows that the map $\g$ (for various $G$) preserves the tensor structures on both sides.
\begin{thm} \label{conv}
There is a Cartesian square \begin{align*}\xymatrix{
  \Irr^\tp X_\bmu(\bb) \ar[d]_{\pr} \ar[r]^{\quad \ \g^{G^d}} & \BB_\bmu^{\widehat G^d} \ar[d]^{\otimes} \\
  \sqcup_\mu \Irr^\tp X_\mu(b) \ar[r]^{\quad \ \g^G} & \sqcup_\mu \BB_\mu^{\widehat G}.  } \end{align*} As a consequence, if $\BB_\mu^{\widehat G}$ appears in the tensor product $\BB_\bmu^{\widehat G} = \BB_{\mu_1}^{\widehat G} \otimes \cdots \otimes \BB_{\mu_d}^{\widehat G}$, then $\g^G$ is determined by $\g^{G^d}$ and $\Irr^\tp X_\mu(b) = \pr ( (\otimes \circ \g^{G^d})\i (\BB_\mu^{\widehat G}) )$.
\end{thm}

\begin{rmk} \label{rmk-char}
The map $\g^G$ coincides with the natural constructions of \cite{XZ} and \cite{HV} for {\it quasi-minuscule} cocharacters, see \cite[Lemme 1.1]{NP}. On the other hand, we know that each highest weight module appears in some tensor product of quasi-minuscule highest weight modules. Thus Theorem \ref{conv} gives a characterization of the map $\g^G$ by the tensor structure of $\widehat G$-crystals.
\end{rmk}

\begin{rmk}\label{rmk-constr}
As an essential application, Theorem \ref{conv}, combined with the construction of \cite{HV}, provides a representation-theoretic construction of $\Irr^\tp X_\mu(b)$ up to taking closures. Indeed, by the reduction method in \cite[\S 5]{GHKR}, it suffices to consider the case where $b$ is superbasic and $G = \Res_{E / F} \GL_n$ with $E / F$ a finite unramified extension. In this case, we can choose a minuscule cocharacter $\bmu \in Y^d$ for some $d$ such that $\BB_\mu^{\widehat G}$ appears in $\BB_\bmu^{\widehat G}$. As $\bmu$ is minuscule, both $\Irr^\tp X_\bmu(\bb)$ and $\g^{G^d}$ are explicitly constructed in \cite{HV}. Then Theorem \ref{conv} shows how to obtain $\Irr^\tp X_\mu(b)$ from $\Irr^\tp X_\bmu(\bb)$ by taking the projection $\pr$. The key is to decompose the tensor product into simple objects $$\BB_\bmu^{\widehat G} = \BB_{\mu_1}^{\widehat G} \otimes \cdots \otimes \BB_{\mu_d}^{\widehat G} = \sqcup_\mu (\BB_\mu^{\widehat G} )^{m_\bmu^\mu},$$ which can be solved combinatorially using the ``Littlewood-Richardson" rule for $\widehat G$-crystals (see \cite[\S 10]{Li}). Here $m_\bmu^\mu$ denotes the multiplicity with which $\BB_\mu^{\widehat G}$ appears in $\BB_\bmu^{\widehat G}$.
\end{rmk}

\begin{rmk} In the case mentioned above where $G = \Res_{E / F} \GL_n$ and $b$ is superbasic, Viehmann \cite{V} and Hamacher \cite{H} defined a decomposition of $X_\mu(b)$ using extended semi-modules (or extended EL-charts). In particular, $\JJ_b \backslash \Irr^\tp X_\mu(b)$ is parameterized by the set of equivalence classes of {\it top} extended semi-modules, that is, the semi-modules whose corresponding strata are top-dimensional. However, it unclear how to construct all the top extended semi-modules if $\mu$ is non-minuscule. It would be interesting to give an explicit correspondence between the top extended semi-modules and the crystal elements in $\BB_\mu(\ul_G(b))$.
\end{rmk}

\

The third goal is to give an explicit construction of an irreducible component from each $\JJ_b$-orbit of $\Irr^\tp X_\mu(b)$ and compute its stabilizer. Combined with Theorem \ref{main}, this will provide a complete parametrization of $\Irr^\tp X_\mu(b)$ in theory. If $b$ is unramified, this task has been done by Xiao-Zhu \cite{XZ}. Otherwise, using Theorem \ref{conv}, it suffices to consider the case where $G$ is adjoint, $\mu$ is minuscule, and $b$ is basic. To handle this case, we consider the decomposition $$X_\mu(b) = \sqcup_{\l \in Y} X_\mu^\l(b),$$ where each piece $X_\mu^\l(b)=I t^\l K/K \cap X_\mu(b)$ is a locally closed subset of $X_\mu(b)$.

\begin{thm} \label{small}
Keep the assumptions on $G, b, \mu$ as above.

(1) $\overline{X_\mu^\l(b)} \in \Irr^\tp X_\mu(b)$ if and only if $\l \in Y$ is {\it small};

(2) each $\JJ_b$-orbit of $\Irr^\tp X_\mu(b)$ has a representative of the form $\overline{X_\mu^\l(b)}$ with $\l$ small;

(3) if $\l \in Y$ is small, then the stabilizer of $\overline{X_\mu^\l(b)}$ in $\JJ_b$ is the standard parahoric subgroup of type $\Pi(\l)$, which is of maximal volume among all parahoric subgroups of $\JJ_b$.
\end{thm}
We refer to \S \ref{def-small} for the meanings of the smallness of $\l$ and the associated type $\Pi(\l)$. As a consequence, we obtain the following result without restrictions on $G$, $b$, and $\mu$.
\begin{thm} [He-Zhou-Zhu] \label{max-stab}
The stabilizer of each top-dimensional irreducible component of $X_\mu(b)$ in $\JJ_b$ is a parahoric subgroup of maximal volume.
\end{thm}

\begin{rmk}
Theorem \ref{max-stab} is first proved by He-Zhou-Zhu \cite{HZZ}. It is also verified in \cite{HZZ} that a parahoric subgroup has maximal volume if and only if its Weyl group has {\it maximal length} (see Theorem \ref{max}). This gives an explicit characterization of parahoric subgroups of maximal volume by their types.

The original proof in \cite{HZZ} is based on the twisted orbital integral method (see \cite{ZZ}) and the Deligne-Lusztig reduction method (see \cite{He}). Our proof is based on the combinatorial properties of small cocharacters, which shows that the stabilizers are parahoric subgroups of maximal length.

\end{rmk}

\subsection{Strategy} Now we briefly discuss the strategy. First we reduced the problem to the case where $b$ is basic and $G$ is simple and adjoint. If $G$ has no non-zero minuscule coweights, then $b$ is unramified and the problem has been solved by Xiao-Zhu \cite{XZ}. Thus, it remains to consider the case where $G$ has some non-zero minuscule cocharacter. In particular, any irreducible $\hat G$-module appears in some tensor product of irreducible $\hat G$-modules with minuscule highest weights (see Lemma \ref{appear}). Combined with the geometric Satake, this observation enables us to decompose the problem into three ingredients: (1) the construction of $\g$ in the case where $b$ is superbasic; (2) the equality $$|\JJ_b \backslash \Irr^\tp X_\mu(b)| = \dim V_\mu(\ul_G(b))$$ in the case where $\mu$ is minuscule and $b$ is basic; and (3) the construction of irreducible components and the computation of their stabilizers in the situation of (2).

The first ingredient is solved in \S\ref{sec-sup} by combining the semi-module method and Littelmann's path model. For the second ingredients, we consider in \S\ref{sec-minu} the following decomposition $$X_\mu(b) = \sqcup_{\l \in Y} X_\mu^\l(b).$$ In Proposition \ref{dim}, we show that $\JJ_b$ acts on $\Irr X_\mu^\l(b)$ transitively and $$\Irr^\tp X_\mu(b) = \sqcup_{\l \in \ca_{\mu, b}^\tp}\Irr \overline{X_\mu^\l(b)},$$ where $\ca_{\mu, b}^\tp$ is the set of coweights $\l$ such that $\dim X_\mu^\l(b) = \dim X_\mu(b)$. In particular, the action of $\JJ_b$ on $\Irr^\tp X_\mu(b)$ induces an equivalence relation on $\ca_{\mu, b}^\tp$, and the $\JJ_b$-orbits of $\Irr^\tp X_\mu(b)$ are naturally parameterized by the corresponding equivalence classes of $\ca_{\mu, b}^\tp$. Therefore, it remains to show the number of these equivalence classes is equal to $\dim V_\mu(\ul_G(b))$. To this end, we give an explicit description of $\ca_{\mu, b}^\tp$ (see Proposition \ref{polar}) and reduce the question to the superbasic case, which has been solved by Hamacher-Viehmann \cite[Theorem 1.5]{HV}. Finally, to solve the last ingredient, we introduce the notion of small cocharacters in \S\ref{sec-stab}. We prove that $X_\mu^\l(b)$ is irreducible if and only if $\l$ is small, and show its stabilizer is of maximal length in this case. Here we will use a general result \cite[Theorem 3.3.1]{ZZ} by Zhou-Zhu showing that the stabilizers are parahoric subgroups, which simplifies the original proof following \cite{XZ}.

\begin{rmk}
Even if the simple adjoint group $G$ has no non-zero minuscule cocharacters, the above approach still works but is more technically involved, by using quasi-minuscule cocharacters instead.
\end{rmk}

%The paper is organized as follows. In \S \ref{sec-pre}, we introduce basic notations and properties. In \S \ref{sec-sup}, we address the superbasic case using the method of semi-modules. In \S \ref{sec-main}, we prove Theorem \ref{main} and Theorem \ref{conv} by assuming Proposition \ref{minu}. In \S \ref{sec-minu}, we prove Proposition \ref{minu}. In \S \ref{sec-stab}, we describe a way to compute the stabilizers and prove Theorem \ref{stab-const}.

\subsection{Comparison with the work \cite{ZZ} by Zhou-Zhu} As mentioned before, this paper aims to give a complete parametrization of $\Irr^\tp X_\mu(b)$, which consists of three parts: the parametrization of $\JJ_b \backslash \Irr^\tp X_\mu(b)$; the construction of representative irreducible components; and the computation of their stabilizers. The major overlap with \cite{ZZ} lies in the first part, see Remark \ref{numer}. A key new feature of this paper is that the $\widehat G$-crystal structure plays an essential role in the construction, see Remark \ref{rmk-char} \& \ref{rmk-constr}. This in particular enables us to handle the type $A$ case, which is not covered in \cite{ZZ}. There is a minor overlap in the third part, where the difference is that this paper gives an algorithm for computing the stabilizers (see Theorem \ref{small}); while the work by Zhou-Zhu produces extra interesting information on the volumes of stabilizers, see \cite[Theorem C \& Remark 1.4.3]{ZZ}.

%\begin{rmk} As an application of \cite{ZZ}, He-Zhou-Zhu \cite{HZZ} verify a conjecture by Xinwen Zhu that all the stabilizers are parahoric subgroups of the maximal volume. In a subsequent paper, we will show that Theorem \ref{small} leads to a new proof of this result.\end{rmk}

\subsection*{Acknowledgement}
We would like to thank Miaofen Chen, Ulrich G\"{o}rtz, Xuhua He, Wen-wei Li, Xu Shen, Liang Xiao, Chia-Fu Yu, Xinwen Zhu and Yihang Zhu for helpful discussions and comments. The author is indebted to Ling Chen for verifying the type $E_7$ case of Lemma \ref{superbasic} using the computer program. We are grateful to Paul Hamacher and Eva Viehmann for detailed explanations on their joint work \cite{HV}.

\section{Preliminaries} \label{sec-pre}
We keep the notations in the introduction. Set $K_H = H(\co_\brF)$ for any subgroup $H \subseteq G$ over $\co_\brF$.

\subsection{Root system} \label{G-setup} Let $\car=(Y, \Phi^\vee, X, \Phi, \SS_0)$ be the based root datum of $G$ associated to the triple $T \subseteq B \subseteq G$, where $X$ and $Y$ denote the (absolute) character and cocharacter groups of $T$ respectively together with a perfect pairing $\< , \>: X \times Y \to \ZZ$; $\Phi$ (resp. $\Phi^\vee$) is the roots system (resp. coroot system); $\SS_0$ is the set of simple reflections. Denote by $\Phi^+$ the set of (positive) roots appearing in $B$. Then $\Phi = \Phi^+ \sqcup \Phi^-$ with $\Phi^- = -\Phi^+$. For $\a \in \Phi$, we denote by $s_\a$ the reflection which sends $\l \in Y$ to $\l - \<\a, \l\> \a^\vee$ with $\a^\vee \in \Phi^\vee$ the corresponding coroot of $\a$. The Frobenius map of $G$ induces an automorphism of $\car$ of finite order, which is still denoted by $\s$. In particular, $\s$ acts on $Y_\RR$ as a linear transformation of finite order.

Let $W_0 = W_G$ be the Weyl group of $T$ in $G$, which is a reflection subgroup of $\GL(Y_\RR)$ generated by $\SS_0$. The Iwahori-Weyl group of $T$ in $G$ is given by $$\tW = \tW_G = N_T(\brF) / K_T \cong Y \rtimes W_0=\{t^\l w; \l \in Y, w \in W_0\},$$ where $N_T$ denotes the normalizer of $T$ in $G$. We can embed $\tW$ into the group of affine transformations of $Y_\RR$ so that the action of $\tw=t^\mu w$ is given by $v \mapsto \mu+w(v)$. Let $\Phi^+$ be the set of (positive) roots appearing in Borel subgroup $B \supseteq T$ and let $$\D = \D_G = \{v \in Y_\RR; 0 < \<\a, v\> < 1, \a \in \Phi^+\}$$ be the base alcove. Then we have $\tW=W^a \rtimes \Omega$, where $W^a=\ZZ \Phi^\vee \rtimes W_0$ is the affine Weyl group and $\Omega$ is the stabilizer of $\D$. Let $Y^+$ be the set of dominant cocharacters. For $\chi, \eta \in Y$ we write $\chi \le \eta$ if $\eta - \chi$ is a sum of positive roots. Write $\chi \leq \eta$ if $\bar \chi \le \bar \eta$. Here $\bar \eta, \bar \chi$ are the dominant $W_0$-conjugate of $\eta, \chi$ respectively.

For $\a \in \Phi$, let $U_\a \subseteq G$ denote the corresponding root subgroup. We set $$I = K_T \prod_{\a \in \Phi^+} U_\a(t \co_{\brF}) \prod_{\b \in \Phi^+} U_{-\b}(\co_{\brF}) \subseteq G(\brF),$$ which is called the standard Iwahori subgroup associated to $T \subseteq B \subseteq G$. We have the Iwasawa decomposition $G(\brF) = \sqcup_{\tw \in \tW} I \tw I$.

\subsection{Affine roots} \label{aff-root} Let $\tPhi = \tPhi_G = \Phi \times \ZZ$ be the set of (real) affine roots. Let $a = \a+k := (\a, k) \in \tPhi$. Denote by $U_a: \co_{\brF} \to G(\brF)$, $z \mapsto U_\a(z t^k)$ the corresponding one-parameter affine root subgroup. We can view $a$ as an affine function such that $a(v)=-\<\a, v\>+k$ for $v \in Y_\RR$, whose zero locus $H_a = \{v \in Y_\RR; a(v)=0\}$ is called an affine root hyperplane. Let $s_a=s_{H_a}=t^{k \a^\vee} s_\a \in \tW$ denote the corresponding affine reflection. Set $\tPhi^+=\{a \in \tPhi; a(\D) > 0\}$. Then $\tPhi=\tPhi^+ \sqcup \tPhi^-$ with $\tPhi^- = -\tPhi^+$. The associated length function $\ell: \tW \to \NN$ is defined by $\ell(\tw)=|\tPhi^- \cap \tw(\tPhi^+)|$. Let $\SS^a=\{s_a; a \in \tPhi, \ell(s_a)=1\}$. Then $W^a$ is generated by $\SS^a$ and $(W^a, \SS^a)$ is a Coxeter system.

Let $\a \in \Phi$. Define $\tilde \a = (\a, 0) \in \tPhi^+$ if $\a < 0$ and $\tilde \a = (\a, 1) \in \tPhi^+$ otherwise. Then the map $\a \mapsto \tilde \a$ gives an embedding of $\Phi$ into $\tilde \Phi^+$, whose image is $\{a \in \tPhi; 0 < a(\D) < 1\}$. Let $\Pi$ be the set of roots $\a \in \Phi$ such that $\tilde \a$ is a simple affine root, namely, $\Pi$ consists of minus simple roots and highest positive roots.
\begin{lem}\label{prod}
Let $\tw, \tw' \in \tW$. Then $I \tw I \tw' I \subseteq \cup_{x \leq \tw} I x \tw' I$ and $I \tw I \tw' I \subseteq \cup_{x' \leq \tw'} I \tw x' I$. Consequently, $\tw I t^\l K \subseteq \cup_{x \leq \tw} I t^{x(\l)} K$ for $\l \in Y$. Here $\leq$ is the usual Bruhat order on $\tW$ associated to $\ell$.
\end{lem}
%\begin{proof} It follows by induction on the length of $\tw$. \end{proof}

\subsection{Levi subgroup} \label{levi-subsec} Let $M \supseteq T$ be a (semistandard) Levi subgroup of $G$. By replacing the triple $T \subseteq B \subseteq G$ with $T \subseteq B \cap M \subseteq M$, we can define $\Phi_M^\pm$, $\tW_M$, $W_M^a$, $W_M$, $I_M$, $\tPhi_M^\pm$, $\D_M$, $\Omega_M$ and so on as above. For $v \in Y_\RR$, we denote by $M_v$ the Levi subgroup generated by $T$ and $U_\a$ for $\a \in \Phi$ such that $\<\a, v\> = 0$, and denote by $N_v$ the unipotent subgroup generated by $U_\b$ for $\b \in \Phi$ such that $\<\b, v\> > 0$. We say $M$ is standard if $M = M_v$ for some dominant vector $v \in Y^+$.

\subsection{Superbasic element} We say $b \in G(\brF)$ is superbasic if none of its $\s$-conjugates is contained in a proper Levi subgroup of $G$. In particular, $b$ is basic in $G(\brF)$, that is, the Newton point $\nu_G(b)$ is central for $\Phi$.
\begin{prop} \label{superbasic}
If $b \in G(\brF)$ is basic, then there exists a unique standard Levi subgroup $M \subseteq G$ such that $M(\brF) \cap [b]$ is a (single) superbasic $\s$-conjugacy class of $M(\brF)$.
\end{prop}
The existence is known. The uniqueness is proved in Appendix B, which is only used in the formulation of Theorem \ref{main}.

\subsection{The element $\e_\l^G$} \label{def-epsilon}
Let $\l \in Y$ and $\g \in \Phi$. We set $\l_\g = -\tilde\g(\l)$, that is, $\l_\g = \<\g, \l\>$ if $\g < 0$ and $\l_\g = \<\g, \l\> - 1$ otherwise. Let $U_\l$ (resp. $U_\l^-$) be the (maximal) unipotent subgroup of $G$ generated by $U_\a$ such that $\l_\a \ge 0$ (resp. $\l_\a < 0$). We define $\e_\l = \e^G_\l \in W_0$ such that $U_\l = {}^{\e_\l} U := \e_\l U \e_\l\i$. Here $U$ denotes the unipotent radical of $B$. Set $I_\l = I \cap t^{\l} K t^{-\l}$ and \begin{align*} I_\l^- &= K_T (I_\l \cap U_\l^-) = K_T (I \cap U_\l^-); \\ I_\l^+ &= K_T (I_\l \cap U_\l) = K_T t^\l K_{U_\l} t^{-\l}. \end{align*} It follows from the Iwasawa decomposition that $I_\l = I_\l^- I_\l^+ = I_\l^+ I_\l^-$.

Let $p: \tW \rtimes \<\s\> \to W_0 \rtimes \<\s\>$ denote the natural projection, where $\<\s\> $ is the finite cyclic subgroup of $\GL(Y_\RR)$ generated by $\s$.
\begin{lem} \label{eta}
Let $\l \in Y$ and $\a \in \Phi$. Then

(1) $\l_\a + \l_{-\a} = -1$;

(2) $s_{\tta}(\l) = \l - \l_\a \a^\vee$;

(3) $\l_\a = \o(\l)_{p(\o)(\a)}$ and $\e_{\o(\l)} = p(\o) \e_\l$ for $\o \in \Omega$.
\end{lem}
\begin{proof}
The first two statements follow directly by definition. We show the last one. Write $\o = t^\eta p(\o)$ for some $\eta \in Y$. Then $$\<p(\o)(\a), \o(\l)\> = \<\a, \l\> + \<p(\o)(\a), \eta\>.$$ By the statement (1) we may assume $\a > 0$. Since $\o \in \Omega$, $\<p(\o)(\a), \eta\> = 0$ if $p(\o)(\a) > 0$ and $\<p(\o)(\a), \eta\> = -1$ otherwise. It follows that $\o(\l)_{p(\o)(\a)} = \l_\a$. In particular, $U_{\o(\l)} = {}^{p(\o)} U_\l$ and hence $\e_{\o(\l)} = p(\o) \e_\l$.
\end{proof}

\begin{lem} \label{inc}
Let $\l, \eta \in Y$ such that $\eta - \l$ is minuscule. Then $I_\eta^- \subseteq I_\l$.
\end{lem}
\begin{proof}
It suffices to show $U_\a(t^{\e_\a}\co_{\brF}) \subseteq I_\l$ for $\l_\a < 0$, where $\e_\a = 0$ if $\a < 0$ and $\e_\a = 1$ otherwise. If $\l_\a < 0$, there is nothing to prove. Suppose $\l_\a > 0$. Then $\<\a, \l\> > \<\a, \eta\>$ and hence $\<\a, \l\> = \<\a, \eta\> + 1$ as $\eta - \l$ is minuscule. This means $-1 \le \<\a, \eta\> \le 0$. If $\<\a, \eta\> = 0$, then $\a > 0$ (since $\l_\a < 0$) and $U_\a(t^{\e_\a}\co_{\brF}) = U_\a(t\co_{\brF}) = U_\a(t^{\<\a, \l\>}\co_{\brF}) \subseteq I_\l^+$. If $\<\a, \eta\> = -1$, then $\a < 0$ (since $\l_\a > 0$) and $U_\a(t^{\e_\a}\co_{\brF}) = U_\a(\co_{\brF}) = U_\a(t^{\<\a, \l\>}\co_{\brF}) \subseteq I_\l^+$.
\end{proof}

\subsection{The convolution map} \label{conv-setup}
Let $d \in \ZZ_{\ge 1}$ and let $G^d$ be the product of $d$ copies of $G$. Let $\bs$ be the Frobenius-type automorphism on $G^d$ given by $(g_1, g_2, \dots, g_d) \mapsto (g_2, \dots, g_d, \s(g_1))$. We set $\bb = (1, \dots, 1, b) \in G(\brF)^d$. Let $\bmu = (\mu_1, \dots, \mu_d) \in Y^d$ be a dominant cocharacter of $G^d$. Let $X_\bmu(\bb)$ be the corresponding affine Deligne-Lusztig variety in $\Gr^d$ using automorphism $\bs$. Consider the twisted product $$\Gr_\bmu^\circ := K t^{\mu_1} K \times_K \cdots \times_K K t^{\mu_d} K/K$$ together with the convolution map $$m_\bmu: \Gr_\bmu := \overline{ \Gr_\bmu^\circ } \to \Gr_{|\bmu|} = \cup_{\mu \leq |\bmu|} \Gr_\mu^\circ$$ given by $(g_1, \dots, g_{d-1}, g_d K) \mapsto g_1 \cdots g_d K$, where $|\bmu|=\mu_1 + \cdots + \mu_d \in Y^+$.
\begin{thm} [\cite{MV}, \cite{NP}, {\cite[Theorem 1.3]{Haines}}] \label{sat}
Let notations be as above. Let $\mu \in Y^+$ with $\mu \leq |\bmu|$ and $y \in \Gr_\mu^\circ$. Then

(1) $\dim m_\bmu\i(y) \le \<\rho, |\bmu| - \mu\>$, and moreover, the number of irreducible components of $m_\bmu\i(y)$ having dimension $\<\rho, |\bmu|-\mu\>$ equals the multiplicity $m_\bmu^\mu$ with which $\BB_\mu^{\widehat G}$ occurs in $\BB_\bmu^{\widehat G} := \BB_{\mu_1}^{\widehat G} \otimes \cdots \otimes \BB_{\mu_d}^{\widehat G}$.

(2) $m_\bmu\i(y)$ is equi-dimensional of dimension $\<\rho, |\bmu|-\mu\>$ if $\bmu$ is minuscule.

Here $\rho = \rho_G$ is the half sum of roots in $\Phi^+$.
\end{thm}

Thanks to Zhu \cite[\S 3.1.3]{Z}, there is a Cartesian square \begin{align*} \xymatrix{
  X_\bmu(\bb) \ar[d]^{\pr} \ar[r] & G(\brF) \times_K \Gr_\bmu^\circ \ar[d]^{\Id \times_K m_\bmu} \\
\cup_{\mu \leq |\bmu|} X_\mu(b) \ar[r] & G(\brF) \times_K \Gr_{|\bmu|},  }\end{align*} where $\pr$ is the projection to the first factor; the lower horizontal map is given by $g_1 K \mapsto (g_1, g_1\i b \s(g_1) K)$; the upper horizontal map is given by $$(g_1 K, \dots, g_d K) \mapsto (g_1, g_1\i g_2, \dots, g_{d-1}\i g_d,  g_d\i b \s(g_1) K).$$ Moreover, via the identification $$\JJ_b \cong \JJ_\bb, \quad g \mapsto (g, \dots, g),$$ the above Cartesian square is $\JJ_b$-equivariant by left multiplication.
\begin{cor} \label{compare}
Let the notation be as above. Then $$\Irr^\tp X_\bmu(\bb) = \sqcup_{\mu \in Y^+, m_\bmu^\mu \neq 0} \sqcup_{C \in \Irr^\tp X_\mu(b)} \Irr^\tp \overline{ \pr\i(C) }.$$ In particular, $$\JJ_b \backslash \Irr^\tp X_\bmu(\bb) = \sqcup_{\mu \in Y^+, m_\bmu^\mu \neq 0} \sqcup_{C \in \JJ_b \backslash \Irr^\tp X_\mu(b)} \Irr^\tp \overline{ \pr\i(C) }$$ and hence $$|\JJ_b \backslash \Irr^\tp X_\bmu(\bb)| = \sum_{\mu \in Y^+} m_\bmu^\mu |\JJ_b \backslash \Irr^\tp X_\mu(b)|.$$ As a consequence, if Theorem \ref{main} is true, then the diagram of Theorem \ref{conv} is Cartesian if it is commutative.
\end{cor}
\begin{proof}
Using the same strategies of \cite{GHKR} and \cite{H} we have $$\dim X_\bmu(\bb) = \<\rho, |\bmu| - \nu_b\> - \frac{1}{2} \df(b) = \dim X_\mu(b) + \<\rho, |\bmu|-\mu\>.$$  Let $\mu \in Y^+$ and $C \in \Irr^\tp X_\mu(b)$. By Theorem \ref{sat} (1), $$\dim \pr\i(C) = \dim C + \<\rho, |\bmu|-\mu\> = \dim X_\mu(b) + \<\rho, |\bmu|-\mu\> \le \dim X_\bmu(\bb),$$ and moreover, the number of irreducible components of $\pr\i(C)$ having dimension $\dim X_\bmu(\bb)$ is equal to $m_\bmu^\mu$ as desired.
\end{proof}

\subsection{Tensor structure} Let $\mu \in Y^+$. Recall that $V_\mu = V_\mu^{\widehat G}$ denotes the simple $\widehat G$-module of highest weight $\mu$, and $\BB_\mu = \BB_\mu^{\widehat G}$ denotes the crystal basis of $V_\mu$, which is a highest weight $\widehat G$-crystal. We refer to \cite{Kashiwara}, \cite{Li} and \cite[\S 3.3]{XZ} for the definition of $\widehat G$-crystals and a realization of $\BB_\mu$ using Littelmann's path model.

For $\l \in Y$, let $\BB_\mu(\l)$ be the set of basis elements of weight $\l$. Then $$|\BB_\mu(\l)| = \dim V_\mu(\l),$$ where $V_\mu(\l)$ denotes the $\l$-weight space of $V_\mu$.

Recall that $\MV_\mu = \MV_\mu^{\widehat G}$ denotes the set of Mirkovi\'{c}-Vilonen cycles in $\Gr_\mu$. By \cite[Theorem 3.1]{BG}, $\MV_\mu$ admits a $\widehat G$-crystal structure, which is isomorphic to $\BB_\mu$. Let $S_1 \in \MV_\mu(\l_1)$ and $S_2 \in \MV_\mu(\l_2)$ be two Mirkovi\'{c}-Vilonen cycles. The twisted product of $S_1$ and $S_2$ is $$S_1 \tilde\times S_2 = (\th_{\l_1}^U)\i(S_1) t^{\l_1} \times_{K_U} S_2 \subseteq G(\brF) \times_K \Gr,$$ where $\th_{\l_1}^U: U(\brF) \to \Gr$ is given by $u \mapsto u t^{\l_1} K$. The convolution of $S_1$ and $S_2$ is defined by $$S_1 \star S_2 = \overline{ m (S_1 \tilde\times S_2) } \cap S^{\l_1 + \l_2},$$ where $m: G(\brF) \times_K \Gr \to \Gr$ denotes the usual convolution map. Note that $\overline{ S_1 \star S_2 } = \overline{ m (S_1 \tilde\times S_2) }$.

Following \cite[Proposition 3.3.15]{XZ}, we fix from now on a bijection $\d \mapsto S_\d$ from $\BB_\mu$ to $\MV_\mu$ for $\mu \in Y^+$, which is compatible with the tensor product for $\widehat G$-crystals, that is, $S_{\d_1 \otimes \d_2} = S_{\d_1} \star S_{\d_2}$.

\subsection{Admissible set} \label{admissible} Let $P=M N$ be a standard parabolic subgroup with standard Levi subgroup $M = \s(M) \supseteq T$ and unipotent radical $N = \s(N)$. Let $\ce$ be one of groups $I$, $M(\brF)$, $N(\brF)$ and $P(\brF)$. For $n \in \ZZ_{\ge 0}$ set $\ce_n = \ce \cap K_n$, where $K_n = \{g \in K = G(\co_{\brF}); g \equiv 1 \mod t^n\}$.  Following \cite{GHKR}, we say a subset $\cd \subseteq \ce$ is admissible if there exists some integer $r > 0$ such that $\cd \ce_r = \cd$ and $\cd / \ce_r \subseteq \ce / \ce_r$ is a (bounded) locally closed subset. In this case, define $$\dim \cd = \dim \cd / \ce_r - \dim \ce_0 / \ce_r,$$ and moreover, we can define topological notions for $\ce$, such as open/closed subsets, irreducible/connected components and so on, by passing to the quotient $\cd / \cd_r$. These definitions are independent of the choice of $r$ since the natural quotient map $\ce / \ce_n  \to \ce / \ce_{n+1}$ is an affine space fiber bundle. For instance, the irreducible components of $\cd$ is defined to be the inverse images of the irreducible components of $\cd / \ce_r \subseteq \ce / \ce_r$ under the natural projection $\ce \to \ce / \ce_r$. We denote by $\Irr \cd$ the set of irreducible components of $\cd$ in this sense.

\section{The set $X_\mu^{P, \l}(b)$}
Keep the notations in the introduction and \S \ref{sec-pre}. In this section, we introduce a decomposition $X_\mu(b) = \sqcup_{\l \in Y} X_\mu^{P, \l}(b)$ with respect to certain parabolic subgroup $P \subseteq G$, and study the irreducible components of $X_\mu^{P, \l}(b)$.

\subsection{The set $H^P(C)$} \label{Levi} Let $P=M N$ be a standard parabolic subgroup as in \S \ref{admissible}. Suppose $b \in M(\brF)$ such that $b$ is basic in $M(\brF)$ and $\nu_M(b) = \nu_G(b)$. Moreover, we always assume that $b \in N_T(\brF)$ is a lift of an element in $\Omega_M$, whose image in $\tW_M$ is still denote by $b$. Notice that $b$ normalizes $N(\brF) I_M$. Let $\phi_b^P: N(\brF) I_M \to N(\brF) I_M$ be the Lang's map given by $h \mapsto h\i b\s(h)b\i$. For $\upsilon \in Y$ let $\th_\l^P: N(\brF) I_M \to \Gr$ be the map given by $h \mapsto h t^\l K$.

Let $C \subseteq \Gr$ be locally closed and irreducible. We define $$H^P(C) = H^P(C; b) = \phi_b^P((\th_\l^P)\i(C)) \subseteq N(\brF) I_M,$$ where $\l \in Y$ such that $N(\brF) I_M t^\l K/K \cap C$ is open dense in $C$. In this case, the map $\g^G$ of Theorem \ref{main} can be formulated by $$\{(h t^\l)\i b \s(h t^\l) K; h \in (\th_\l^P)\i(C)\} = t^{-\l} H^P(C) t^{b\s(\l)} K/K,$$ where $b\s(\l) \in Y$ is defined by the affine action of $\tW \rtimes \<\s\>$ on $Y$, see \S\ref{G-setup}.

For $\mu \in Y^+$ and $\l \in Y$ we set $X_\mu^{\l, P}(b) = N(\brF) I_M t^\l K/K \cap X_\mu(b)$.
\begin{lem}\label{irr-NM}
The map $C \mapsto H^P(C)$ for $C \in \Irr X_\mu^{\l, P}(b)$ induces a bijection $$(N(\brF) I_M \cap \JJ_b) \backslash \Irr X_\mu^{\l, P}(b) \cong \Irr (t^\l K t^\mu K t^{-b\s(\l)} \cap N(\brF) I_M).$$ In particular, $H^P(C; b)$ is invariant under left/right multiplication by $K_T$.
\end{lem}
\begin{proof}
Note that $b t^\chi K = t^{b(\chi)} K$ and $K t^{-\chi} b\i = K t^{-b(\chi)}$ for $\chi \in Y$. Therefore, \begin{align*} (\th_\l^P)\i(X_\mu^{\l,P}(b)) &= (\phi_b^P)\i (t^\l K t^\mu K t^{-\s(\l)} b\i \cap N(\brF) I_M) \\ &= (\phi_b^P)\i(t^\l K t^\mu K t^{-b\s(\l)} \cap N(\brF) I_M),\end{align*}
 As $\phi_b^P$ is an etale covering of $N(\brF) I_M$ with Galois group $N(\brF) I_M \cap \JJ_b$, the map $C \mapsto H^P(C)$ for $C \in \Irr X_\mu^{\l, P}(b)$ induces a bijection $$(N(\brF) I_M \cap \JJ_b) \backslash \Irr X_\mu^{\l, P}(b) \cong \Irr (t^\l K t^\mu K t^{-b\s(\l)} \cap N(\brF) I_M).$$ The proof is finished.
\end{proof}

Now we focus on the basic case. Let $I_\der$ denote the derived subgroup of $I$.
\begin{cor} \label{derived}
Assume $b$ is basic. Then the map $C \mapsto H_\der^G(C)$ for $C \in \Irr X_\mu^{\l, G}(b)$ induces a bijection $$(I_\der \cap \JJ_b) \backslash \Irr X_\mu^{\l, G}(b) \cong \Irr (t^\l K t^\mu K t^{-b\s(\l)} \cap I_\der).$$ Here $H^G_\der(C) = \phi_b^G((\th_\l^G)\i(C) \cap I_\der)$.
\end{cor}
\begin{proof}
Note that $I t^\l K/K = I_\der t^\l K/K$ and that $\phi_b^G$ restricts to an etale covering of $I_\der$ with Galois group $I_\der \cap \JJ_b$. Then the statement follows in the same way of Lemma \ref{irr-NM}.
\end{proof}

recall that $p: \tW \rtimes \<\s\> \to W_0 \rtimes \<\s\>$ is the natural projection.
\begin{lem} \label{change} \label{der}
Assume $b$ is basic. For $\l \in Y$ and $C \in \Irr X_\mu^{\l, G}(b)$ we have $$t^{-\l'}H^G(C'; b') t^{b'\s(\l')} = p(\o) t^{-\l} H^G(C; b) t^{b\s(\l)} p(\o)\i,$$ where $\o \in \Omega$, $\dot \o \in N_T(\brF)$ is a lift of $\o$, $\l' = \o(\l)$, $C' = \dot\o C$ and $b' = \dot\o b \s(\dot\o)\i$. Note that the right hand side is independent of the choice of $\dot\o$ by Lemma \ref{irr-NM}.
\end{lem}
\begin{proof}
As $\o \in \Omega$, we have $\dot\o I \dot\o\i = I$ and hence $$C' = \dot\o C \subseteq \dot\o I t^\l K/K = I t^{\o(\l)} K/K.$$ So $(\th_{\l'}^G)\i(C') = \dot\o (\th_\l^G)\i(C) \dot\o\i$, and the statement follows by definition.
\end{proof}

\begin{cor} \label{indep}
In the superbasic case, the map $\g^G$ in Theorem \ref{main} is independent of the choice of $b$.
\end{cor}
\begin{proof}
Suppose $b$ is superbasic and let $b, b', \l, \l', \o, C, C'$ be as in Lemma \ref{der}. Let $\g_b^G$ (resp. $\g_{b'}^G$) be the map $\g^G$ in Theorem \ref{main} defined with respect to $b$ (resp. $b'$). We need to show that $\g^G_b(C) = \g^G_{b'}(C')$. By definition (for the superbasic case), it suffices to show that $$(\e_\l^G)\i t^{-\l} H^G(C; b) t^{b\s(\l)} K/K = (\e_{\l'}^G)\i t^{-\l'}H^G(C'; b') t^{b'\s(\l')} K/K,$$ which follows from Lemma \ref{eta} (3) and Lemma \ref{der}.
\end{proof}

\begin{cor} \label{irr}
Assume $b$ is basic. For $\l \in Y$ there are natural bijections $$(I \cap \JJ_b) \backslash \Irr X_\mu^{\l, G}(b) \cong \Irr (t^\l K t^\mu K t^{-b\s(\l)} \cap I_{U_\l}) \cong (I_\der \cap \JJ_b) \backslash \Irr X_\mu^{\l, G}(b),$$ where $I_{U_\l} = I \cap U_\l$ and $U_\l$ is as in \S \ref{def-epsilon}. Moreover, for $C \in  \Irr X_\mu^{\l, G}(b)$ we have $$H^G(C) = K_T H^G_\der(C) = H^G_\der(C) K_T.$$
\end{cor}
\begin{proof}
By Lemma \ref{irr-NM} the map $C \mapsto H^G(C)$ for $C \in \Irr X_\mu^{\l, G}(b)$ induces a bijection \begin{align*} \tag{a} (I \cap \JJ_b) \backslash \Irr X_\mu^{\l, G}(b) \cong \Irr(t^\l K t^\mu K t^{-b\s(\l)} \cap I) \cong \Irr (t^\l K t^\mu K t^{-b\s(\l)} \cap I_{U_\l}),\end{align*} where the second bijection follows from that $$t^\l K t^\mu K t^{-b\s(\l)} \cap I = I_\l^- (t^\l K t^\mu K t^{-b\s(\l)} \cap I_{U_\l}).$$ Similarly, by Corollary \ref{derived} we have \begin{align*} \tag{b} (I_\der \cap \JJ_b) \backslash \Irr X_\mu^{\l, G}(b) \cong \Irr(t^\l K t^\mu K t^{-b\s(\l)} \cap I_\der) \cong \Irr (t^\l K t^\mu K t^{-b\s(\l)} \cap I_{U_\l}).\end{align*} So the first statement follows.

By (a) and (b), there exist $\cz, \cz' \in \Irr (t^\l K t^\mu K t^{-b\s(\l)} \cap I_{U_\l})$ such that $$(I_\l^- \cap I_\der) \cz = H_\der^G(C) \subseteq H^G(C) = I_\l^- \cz'.$$ In particular, $\cz = \cz'$. Moreover, $K_T$ normalises $t^\l K t^\mu K t^{-b\s(\l)} \cap I_{U_\l}$, and hence normalises each of its irreducible components. So we have $$H^G(C) = I_\l^- \cz = (I_\l^- \cap I_\der) K_T \cz = H_\der^G(C) K_T = K_T H_\der^G(C).$$ The second statement is proved.
\end{proof}

\begin{lem} \label{irr-u}
For $\l, \chi \in Y$ there is a natural bijection $$\Irr (t^\l K t^\mu K t^{-\chi} \cap I_{U_\l}) \cong \Irr (K t^{-\mu} K/K \cap t^{-\chi} I_{U_\l} t^\l K/K). $$
\end{lem}
\begin{proof}
The map $g \mapsto t^{-\chi} g\i t^\l$ gives a bijection $$t^\l K t^\mu K t^{-\chi} \cap I_{U_\l} \cong K t^{-\mu} K \cap t^{- \chi} I_{U_\l} t^\l.$$ By the definition of $U_\l$ we have $K_{U_\l} \subseteq t^{-\l} I_{U_\l} t^\l $. Therefore, \begin{align*}\Irr (t^\l K t^\mu K t^{-\chi} \cap I_{U_\l}) &\cong \Irr ((K t^{-\mu} K \cap t^{-\chi} I_{U_\l} t^\l) / K_{U_\l} ) \\ &= \Irr (K t^{-\mu} K/K \cap t^{-\chi} I_{U_\l} t^\l K/K), \end{align*} where the identity follows from that $t^{-\chi} I_{U_\l} t^\l / K_{U_\l} \cong t^{-\chi} I_{U_\l} t^\l K/K$.
\end{proof}

\subsection{The minuscule and basic case} \label{def-R}
Suppose $\mu \in Y^+$ is minuscule and $b$ is basic. For $D \subseteq \tW$ we set $D \cap \JJ_b = \{\tw \in D; b \s(\tw) b\i = \tw\}$.

For $\l \in Y$ we write $X_\mu^\l(b) = X_\mu^{\l, G}(b) = I t^\l K/K \cap X_\mu(b)$. Let $\ca_{\mu, b}^G$ and $\ca_{\mu, b}^{G, \tp}$ be the sets of $\l \in Y$ such that $X_\mu^\l(b) \neq \emptyset$ and $\dim X_\mu^\l (b) = \dim X_\mu(b)$ respectively.

For $\a \in \Phi$ define $\a^i = p(b\s)^i(\a) \in \Phi$ and $\tta^i = (b\s)^i(\tta) \in \tPhi$ for $i \in \ZZ$, where $\tta$ is as in \S \ref{aff-root} and $p: \tW \rtimes \<\s\> \to W_0 \rtimes \<\s\>$ is the natural projection.

For $\l \in \ca_{\mu, b}^G$ define $\l^\natural = -\l + b\s(\l)$, and denote by $R_{\mu, b}^G(\l)$ the set of roots $\a \in \Phi$ such that $\<\a, \l^\natural\> = -1$ and $\l_\a \ge 1$. By Lemma \ref{lam} (1) below, this condition is equivalent to that $\<\a, \l^\natural\> = -1$ and $\l_{\a\i} \ge 0$.
\begin{lem} \label{lam} \label{natural}
Let $\l \in Y$. Then we have (1) $\<\a, \l^\natural\> = \l_{\a\i} - \l_\a$ for $\a \in \Phi$ and (2) $\tw(\l)^\natural = p(\tw)(\l^\natural)$ for $\tw \in \tW \cap \JJ_b$.
\end{lem}
\begin{proof}
Suppose $b \in t^\t W_0$ for some $\t \in Y$. Then $$\<\a, \l^\natural\> = -\<\a, \l\> + \<\a, \t\> + \<\a\i, \l\>.$$ As $b \in \Omega$, $b\s$ preserves the fundamental alcove $\D$ and hence preserves the set $\{\tilde \b; \b \in \Phi\}$, see \S \ref{aff-root}. Thus $\a, \a\i$ are both positive or negative if $\<\a, \t\>=0$; $\a < 0$ and $\a\i > 0$ if $\<\a, \t\> = -1$;  $\a > 0$ and $\a\i < 0$ if $\<\a, \t\> = 1$. In all cases we have $\<\a, \l^\natural\> = \l_{\a\i} - \l_\a$ as desired. For $\tw \in \tW \cap \JJ_b$, it follows that $$\tw(\l)^\natural = -\tw(\l) + b\s \tw(\l) = -\tw(\l) + \tw b\s(\l) = p(\tw) (-\l + b\s(\l)) = p(\tw)(\l^\natural).$$ The proof is finished.
\end{proof}

\begin{lem} \label{omega}
We have $R_{\mu, b}^G(\o(\l)) = p(\o) R_{\mu, b}^G(\l)$ for $\o \in \Omega \cap \JJ_b$, $\l \in \ca_{\mu, b}^G$.
\end{lem}
\begin{proof}
Notice that $p(\tw)(\g)^i = p(\tw)(\g^i)$ for $\g \in \Phi$, $\tw \in \tW \cap \JJ_b$ and $i \in \ZZ$. The statement now follows from Lemma \ref{eta} (3) and Lemma \ref{natural} (1).
%As $\o \in \Omega \cap \JJ_b$, then $\<p(\o)(\g), \o(\l)^\natural\> =\<\a, \l^\natural\>$ by Lemma \ref{natural}. Moreover, by Lemma \ref{eta} we have $\o(\l)_{p(\o)(\g)^i} = \o(\l)_{p(\o)(\g^i)} = \l_{\g^i}$ for $i \in \ZZ$. Thus $R_{\mu, b}^G(\o(\l)) = p(\o) R_{\mu, b}^G(\l)$ by definition.
\end{proof}

\begin{prop} \label{dim}
Suppose $\mu$ is minuscule. Then  $\l \in \ca_{\mu, b}^G$, that is $X_\mu^\l(b) \neq \emptyset$, if and only if $\l^\natural$ is conjugate to $\mu$ by $W_0$. Moreover, in this case,

(1) $t^\l K t^\mu K t^{-b\s(\l)} \cap I = I_\l \prod_{\d \in R_{\mu, b}^G(\l)} U_\d(t^{\<\d, \l\> - 1} \co_{\brF})$;

(2) $X_\mu^\l(b)$ is smooth and $I \cap \JJ_b$ acts on $\Irr X_\mu^\l(b)$ transitively;

(3) $\dim X_\mu^\l(b) = |R_{\mu, b}^G(\l)|$.
\end{prop}
\begin{proof}
By Corollary \ref{irr} we have $$ (I \cap \JJ_b) \backslash \Irr X_\mu^\l(b) \cong \Irr (t^\l K t^\mu K t^{-b\s(\l)} \cap I) \cong  \Irr (K t^\mu K \cap t^{-\l} I_{U_\l} t^{\l} t^{\l^\natural}).$$ As $\mu$ is minuscule, we see that $\emptyset \neq K t^\mu K \cap t^{-\l} I_{U_\l} t^{\l} t^{\l^\natural} \subseteq U_\l(\brF) t^{\l^\natural}$, that is, $\l \in \ca_{\mu, b}^G$, if and only if $\l^\natural$ is conjugate to $\mu$. Moreover, in this case, $$K t^\mu K \cap U_\l(\brF) t^{\l^\natural} = K_{U_\l} t^{\l^\natural} K_{U_\l} = (\prod_\a U_\a( \co_\brF) \prod_\b U_\b(t\i \co_\brF))  t^{\l^\natural},$$ where $\a, \b$ range over the roots of $U_\l$ such that $\<\a, \l^\natural\> \ge 0$ and $\<\b, \l^\natural\> = -1$ in any fixed orders. Here the root subgroups $U_\b$ commute with each other since $\l^\natural$ is minuscule. On the other hand, $t^{-\l} I_{U_\l} t^{\l} t^{\l^\natural} = (\prod_\g U_\g(t^{-\l_\g} \co_\brF)) t^{\l^\natural}$, where $\g$ ranges over the roots of $U_\l$ (or $\l_\g \ge 0$) in the above fixed order. Thus $$K_{U_\l} t^{\l^\natural} K_{U_\l} \cap t^{-\l} I_{U_\l} t^{\l} t^{\l^\natural} = (K_{U_\l} \prod_\d U_\d(t\i \co_\brF)) t^{\l^\natural},$$ where $\d$ ranges over $R_{\mu, b}^G(\l) = \{\g \in \Phi; \<\g, \l^\natural\> = -1, \l_\g \ge 1\}$. Therefore, \begin{align*} t^\l K t^\mu K t^{-b\s(\l)} \cap I &=  I_\l t^\l (K t^\mu K \cap t^{-\l} I_{U_\l} t^{\l} t^{\l^\natural})  t^{-b\s(\l)} \\ &= I_\l t^\l (K_{U_\l} t^{\l^\natural} K_{U_\l} \cap t^{-\l} I_{U_\l} t^{\l} t^{\l^\natural})  t^{-b\s(\l)}  \\ &= I_\l t^\l (K_{U_\l} \prod_{\d \in R_{\mu, b}^G(\l)} U_\d(t\i \co_\brF)) t^{-\l} \\ &= I_\l \prod_{\d \in R_{\mu, b}^G(\l)} U_\d(t^{\<\d, \l\> - 1} \co_{\brF}). \end{align*} So the statement (1) follows. The statement (2) follows from Corollary \ref{irr} by noticing that $t^\l K t^\mu K t^{-b\s(\l)} \cap I$ is smooth and irreducible.

As $(\th_\l^G)\i(X_\mu^\l(b)) = (\phi_b^G)\i(t^\l K t^\mu K t^{-b\s(\l)} \cap I)$, we deduce by (1) that \begin{align*} \dim X_\mu^\l(b) &= \dim ((\th_\l^G)\i(X_\mu^\l(b)) / I_\l) \\ &= \dim ((\th_\l^G)\i(X_\mu^\l(b))) - \dim I_\l \\ &= \dim (t^\l K t^\mu K t^{-b\s(\l)} \cap I) - \dim I_\l \\ &= |R_{\mu, b}^G(\l)|. \end{align*} So the statement (3) follows.
\end{proof}

\subsection{The set $H^{P^d}(C)$} \label{product}
Let $P = M N$ and $b \in M(L)$ be as in \S \ref{Levi}. Let $G^d$, $\bs$, $\bb$, $\bmu$ and $\pr$ be as in \S\ref{conv-setup}. We also denote by $\pr$ the projections $G^d(\brF) \to G(\brF)$ and $Y^d \to Y$ to the first factors.

Let $C \in \Irr^\tp X_\bmu(\bb)$ and $\bl \in Y^d$ such that $(N(\brF)I_M)^d t^\bl K^d / K^d \cap C$ is open dense in $C$. By Corollary \ref{compare}, $\overline{ \pr(C) } \in \Irr^\tp X_\mu(b)$ for some $\mu \in Y^+$. Let $\l = \pr(\bl)$. Then $N(\brF) I_M t^\l K/K \cap \pr(C)$ is open dense in $\overline{\pr (C)}$. Set $\l^\dag_\bullet = \bb \bs(\bl)$, $\phi_\bb = \phi_\bb^{P^d}$ and $\th_\bl = \th_\l^{P^d}$. By Lemma \ref{irr-NM}, $$H^{P^d}(C) = \phi_\bb (\th_\bl\i(C)) \in \Irr (t^\bl K^d t^\bmu K^d t^{-\l^\dag_\bullet} \cap (N(\brF) I_M)^d).$$ So we can write $$H^{P^d}(C) = H_1(C) \times \cdots \times H_d(C),$$ where $H_\t(C) \in \Irr (t^{\l_\t} K t^{\mu_\t} K t^{-\l^\dag_\t} \cap N(\brF) I_M)$ for $1 \le \t \le d$ with $\bmu = (\mu_1, \dots, \mu_d)$, $\bl = (\l_1, \dots, \l_d)$ and $\l^\dag_\bullet = (\l_1^\dag, \dots, \l_d^\dag)$.
\begin{lem} \label{convolution}
Let notations be as above. Then we have $$\overline{ H^P(\overline{\pr(C)}) } = \overline{ H_1(C) \cdots H_d(C)} \subseteq N(\brF) I_M$$ As a consequence, $$\overline{ t^{-\l} H^P(\overline{\pr(C)}) t^{b\s(\l)} K/K } = \overline{ t^{-\l_1} H_1(C) t^{\l^\dag_1} \cdots  t^{-\l_d} H_d(C) t^{\l^\dag_d} K/K  }.$$
\end{lem}
\begin{proof}
As $\overline{ \pr((N(\brF)I_M)^d t^\bl K^d / K^d \cap C) } = \overline{ \pr(C)}$, we see that $$\overline{ \pr(\th_\bl\i(C)) } = \overline{ (\th_\l^P)\i(\overline{\pr(C)}) } \subseteq  N(\brF) I_M.$$ On the other hand, the equality $H^{P^d}(C) = \phi_\bb(\th_{\bl}\i(C))$ means that $$H_1(C) \times \cdots \times H_d(C) = \{(h_1\i h_2, \dots, h_{d-1}\i h_d, h_d\i b \s(h_1) b\i); (h_1, \dots, h_d) \in \th_\bl\i(C)\}.$$ In particular, $$H_1(C) \cdots H_d(C) = \phi_b^P(\pr(\th_\bl\i(C))).$$ So $\overline{ H^P(\overline{\pr(C)}) } = \overline{ \phi_b^P((\th_\l^P)\i(\overline{\pr(C)})) } = \overline{ \phi_b^P(\pr(\th_\bl\i(C))) } = \overline{ H_1(C) \cdots H_d(C)}$ as desired.
\end{proof}

\section{The superbasic case} \label{sec-sup}
In this section we consider the case where $b$ is superbasic.

\subsection{The formulation} \label{formulation} For $d \in \ZZ_{\ge 1}$ let $\bs$, $\bb$, $\bmu$ and $P = G$ be as in \S \ref{conv-setup}. Let $C \in \Irr^\tp X_\bmu(\bb)$ and let $\bl \in Y^d$ such that $(I^d t^{\bl} K^d/K^d) \cap C$ is open dense in $C$. Following \S \ref{product} let $$H^{G^d}(C) = H_1(C) \times \cdots \times H_d(C) \in \Irr (t^\bl K^d t^\bmu K^d t^{-\bb \bs(\bl)} \cap I^d),$$ where $H_\t(C) \in \Irr (t^{\l_\t} K t^{\mu_\t} K t^{-\l_\t^\dag} \cap I)$ for $1 \le \t \le d$ with $\bl=(\l_1, \dots, \l_d)$ and $\l_\bullet^\dag = \bb \bs(\bl) = (\l_1^\dag, \dots, \l_d^\dag)$. We set $X_\bmu^{\bl}(\bb) = X_\bmu^{\bl, G^d}(\bb)$ for simplicity.

The main result of this section is
\begin{thm} \label{reform}
Let $C, \bl, \l_\bullet^\dag$ be as above. There is $\g^{G^d}(C) = (\g_1, \dots, \g_d) \in \BB_\bmu^{\widehat G^d}$ such that $$\overline{ t^{-\l_1} H_1(C) t^{\l_1^\dag} \times_K \cdots \times_K t^{-\l_d} H_d(C) t^{\l_d^\dag} K/K } = \e_{\l_1} \overline{ S_{\g_1} \tilde\times \cdots \tilde\times S_{\g_d} },$$ where $\e_{\l_1} = \e_{\l_1}^G$ is as in \S \ref{def-epsilon}. Moreover, the map $C \mapsto \g^{G^d}(C)$ factors through a bijection $$\JJ_\bb \backslash \Irr^\tp X_\bmu(\bb) \cong \BB_\bmu^{\widehat G^d}(\ul_{G^d}(\bb)).$$ In particular, Theorem \ref{main} is true if $b$ is superbasic by taking $d = 1$.
\end{thm}

\subsection{Reduction procedure} First we show how to pass to the case where $G = \Res_{E / F} \GL_n$ with $E / F$ an unramified extension.
\begin{lem} \label{isogeny}
Let $f: G \to G'$ be a central isogeny. Then Theorem \ref{reform} is true for $G$ if and only if it is true for $G'$.
\end{lem}
\begin{proof}
We still denote by $f$ the induced maps $G(\brF) \to G'(\brF)$, $\Gr_G \to \Gr_{G'}$ and so on. Let $\s'$ be the Frobenius automorphism of $G'$. Let $C \in \Irr^\tp X_\bmu(\bb)$ and $\bl \in Y^d$ such that $I^d t^\bl K^d / K^d \cap C$ is open dense in $C$. Let $\bo \in \pi_1(G^d)$ such that the corresponding connected component $\Gr_{G^d}^\bo$ contains $C$. Let $K' = G'(\co_\brF)$ and $I' \subseteq K'$ the Iwahori subgroup containing $f(I)$. Denote by $\mu_\bullet', \l_\bullet', C', b_\bullet', \o_\bullet'$ the images of $\bmu, \bl, C, \bb, \bo$ under $f$ respectively. Write $\bl = (\l_1, \dots, \l_d)$, $\bb = (b_1, \dots, b_d)$, $\l_\bullet' = (\l_1', \dots, \l_d')$ and $b_\bullet' = (b_1', \dots, b_d')$.

By \cite[Corollary 2.4.2]{CKV} and \cite[Proposition 3.1]{HV}, $f$ induces a homeomorphism $$X_\bmu(\bb) \cap \Gr_{G^d}^\bo =: X_\bmu(\bb)^\bo \overset \sim \longrightarrow X_{\mu_\bullet'}(b_\bullet')^{\o_\bullet'} := X_{\mu_\bullet'}(b_\bullet')^{\o_\bullet'} \cap \Gr_{{G'}^d}^{\o_\bullet'}.$$ So $C' \in \Irr^\tp X_{\mu_\bullet'}(b_\bullet')$ and ${I'}^d t^{\l_\bullet'} {K'}^d / {K'}^d \cap C'$ is open dense in $C'$. Moreover, as $f(I_\der) = I_\der'$ we have \begin{align*} f((\th_\bl^{G^d})\i(C) \cap (I_\der)^d) &= (\th_{\l_\bullet'}^{{G'}^d})\i(C') \cap (I_\der')^d \\ f(H_\der^{G^d}(C)) &= H_\der^{{G'}^d}(C'), \end{align*} where $H_\der^{G^d}(C)$ and $H_\der^{{G'}^d}(C')$ are defined as in Corollary \ref{derived}. By Corollary \ref{der}, \begin{align*} H^{G^d}(C) &= H_\der^{G^d}(C) T^d(\co_\brF) = T^d(\co_\brF) H_\der^{G^d}(C); \\ H^{{G'}^d}(C') &= H_\der^{{G'}^d}(C') {T'}^d(\co_\brF) = {T'}^d(\co_\brF) H_\der^{{G'}^d}(C'). \end{align*} Therefore, $f$ induces a surjection and hence a bijection \begin{align*} &\quad\ t^{-\l_a} H_a(C) t^{\l_a^\dag} \times_K \cdots \times_K t^{-\l_c} H_c(C) t^{\l_c^\dag} K/K \\ &= t^{-\l_a} H_{\der, a}(C) t^{\l_a^\dag} \times_K \cdots \times_K t^{-\l_c} H_{\der, c}(C) t^{\l_c^\dag} K/K \\ &\cong t^{-\l_a'} H_{\der, a}(C') t^{\l_a^{\prime\dag}} \times_{K'} \cdots \times_{K'} t^{-\l_c'} H_{\der, c}(C') t^{\l_c^{\prime\dag}} K'/K' \\ &= t^{-\l_a'} (H_a(C') t^{\l_a^{\prime\dag}} \times_{K'} \cdots \times_{K'} t^{-\l_c'} H_c(C') t^{\l_c^{\prime\dag}} K'/K',\end{align*} where, as in \S\ref{product} we write \begin{align*} H_\der^{G^d}(C) &= H_{\der, 1}(C) \times \cdots \times H_{\der, d}(C) \\ H_\der^{{G'}^d}(C') &= H_{\der, 1}(C') \times \cdots \times H_{\der, d}(C'). \end{align*}

By Corollary \ref{der}, we have the following commutative diagram
\begin{align*}\tag{a}
\xymatrix{
  ((I_\der)^d \cap \JJ_\bb)  \backslash \Irr X_\bmu^\bl(\bb) \ar[d]_f \ar[r]^{\sim\quad\ \ } & \Irr (t^\bl K^d t^\bmu K^d t^{-\l_\bullet^\dag} \cap (I^d)_{(U^d)_\bl}) \ar[d]^\wr_f \\
  ((I_\der')^d \cap \JJ_{b_\bullet'}) \backslash \Irr X_{\mu_\bullet'}^{\l_\bullet'}(b_\bullet') \ar[r]^{\sim\quad\quad\ } & \Irr (t^{\l_\bullet'} {K'}^d t^{\mu_\bullet'} {K'}^d t^{-\l_\bullet^{\prime\dag}} \cap ({I'}^d)_{(U^d)_{\l_\bullet'}}),  }
\end{align*} where the right vertical bijection follows from Lemma \ref{irr-u} and the homeomorphism $f: \Gr_{G^d}^{\bo} \cong \Gr_{{G'}^d}^{\o_\bullet'}$.

Let $\JJ_\bb^0$ and $\JJ_{b_\bullet'}^0$ be the kernels of the natural projections $\JJ_\bb \to \pi_1(G^d)$ and $\JJ_{b_\bullet'} \to \pi_1({G'}^d)$ respectively. Then we have a commutative diagram \begin{align*} \xymatrix{
\JJ_\bb^0 \backslash \Irr^\tp X_\bmu(\bb)^\bo \ar[d] \ar[r]^{\sim} & \JJ_\bb \backslash \Irr^\tp X_\bmu(\bb) \ar[d] \\
\JJ_{b_\bullet'}^0 \backslash \Irr^\tp X_{\mu_\bullet'}(b_\bullet')^{\o_\bullet'} \ar[r]^{\sim} & \JJ_{b_\bullet'} \backslash \Irr^\tp X_{\mu_\bullet'}(b_\bullet').  }\end{align*} Thus the bijection $\JJ_\bb \backslash \Irr^\tp X_\bmu(\bb) \cong \JJ_{b_\bullet'} \backslash \Irr^\tp X_{\mu_\bullet'}(b_\bullet')$ follows from the following commutative diagrams \begin{align*} \xymatrix{
((I_\der)^d \cap \JJ_\bb)  \backslash \Irr^\tp X_\bmu(\bb)^\bo \ar[d]^{\wr} \ar[r]^{\quad \sim} & (I^d \cap \JJ_b) \backslash \Irr^\tp X_\bmu(\bb)^\bo \ar[d] \ar[r]^{\quad \sim} & \JJ_\bb^0 \backslash \Irr^\tp X_\bmu(\bb)^\bo \ar[d]  \\
  ((I_\der')^d \cap \JJ_{b_\bullet'}) \backslash \Irr^\tp X_{\mu_\bullet'}(b_\bullet')^{\o_\bullet'} \ar[r]^{\quad \sim} & ({I'}^d \cap \JJ_{b_\bullet'}) \backslash \Irr^\tp X_{\mu_\bullet'}(b_\bullet')^{\o_\bullet'} \ar[r]^{\quad \sim} & \JJ_{b_\bullet'}^0 \backslash \Irr^\tp X_{\mu_\bullet'}(b_\bullet')^{\o_\bullet'} ,   }
\end{align*}
where the left horizontal bijections follow from Corollary \ref{der}; the right horizontal bijections follow from that $\JJ_\bb^0 = I^d \cap \JJ_\bb$ and $\JJ_{b_\bullet'}^0 = {I'}^d \cap \JJ_{b_\bullet'}$ as $\bb, b_\bullet'$ are superbasic; the leftmost vertical bijection follows from the natural bijection $$((I_\der)^d \cap \JJ_\bb)  \backslash \Irr X_\bmu^\bl(\bb) \cong ((I_\der')^d \cap \JJ_{b_\bullet'}) \backslash \Irr X_{\mu_\bullet'}^{\l_\bullet'}(b_\bullet')$$ in the commutative diagram (a). The proof is finished.

%\begin{align*}\xymatrix{((I_\der)^d \cap \JJ_\bb)  \backslash \Irr^\tp X_\bmu(\bb)^\bo \ar[d]_f \ar[r] & \sqcup_\bl ((I_\der)^d \cap \JJ_\bb)  \backslash \Irr X_\bmu^\bl(\bb) \ar[d]^\wr_f \\ ((I_\der')^d \cap \JJ_{b_\bullet'})  \backslash \Irr X_{\mu_\bullet'}(b_\bullet')^{\o_\bullet'} \ar[r] & \sqcup_{\l_\bullet'} ((I_\der')^d \cap \JJ_{b_\bullet'})  \backslash \Irr X_{\mu_\bullet'}^{\l_\bullet'}(b_\bullet'),   } \end{align*} where the upper horizontal map is given by $C \mapsto C \cap X_\bmu^\bl(\bb)$ for $C \in \Irr^\tp X_\bmu(\bb)$ such that $C \cap X_\bmu^\bl(\bb)$ is open dense in $C$, and the lower one is given in a similar way.
\end{proof}

Let $G_\ad$ denote the adjoint group of $G$. As $b$ is superbasic, by \cite[Lemma 3.11]{CKV}, $G_\ad \cong \prod_i \Res_{F_i / F} \PGL_{n_i}$ for some unramified extensions $F_i / F$. In view of the following natural central isogenies $$G \longrightarrow G_\ad \cong \prod_i \Res_{F_{d_i} / F} \PGL_{n_i} \longleftarrow \prod_i \Res_{F_{d_i} / F} \GL_{n_i},$$ we will assume in the rest of this section that $G = \Res_{E / F} \GL_n$ for some unramified extension $E / F$ by Lemma \ref{isogeny}.

\subsection{Reduction procedure in the minuscule case} \label{red-min} Now we consider the case where $\bmu$ is minuscule. Let $\ca_{\bmu, \bb} = \ca_{\bmu, \bb}^{G^d}$ and $\ca_{\bmu, \bb}^\tp = \ca_{\bmu, \bb}^{G^d, \tp}$ be defined in \S \ref{def-R}. For $\bl \in \ca_{\bmu, \bb}$ set $\l_\bullet^\dag = \bb \bs(\bl)$, $\l^\natural_\bullet = -\bl + \l_\bullet^\dag$ and $\l^\flat_\bullet = \e_\bl\i (\l_\bullet^\natural)$, where $\e_\bl := \e_\bl^{G^d}$ is defined in \S \ref{def-epsilon}.

Since $\bmu$ is minuscule, we identify $\BB_\bmu^{\widehat G^d}$ canonically with the set of cocharacters in $Y^d$ which are conjugate to $\bmu$. Moreover, for $\z_\bullet \in \BB_\bmu^{\widehat G^d}$, the corresponding Mirkovi\'{c}-Vilonen cycle is $S_{\z_\bullet} = (K_U)^d t^{\z_\bullet} K^d/K^d$.
\begin{thm} [{\cite[Proposition 1.6]{HV}}] \label{supb-minu}
Assume $\bmu$ is minuscule. For $\l \in \ca_{\bmu, \bb}$,

(1) $X_\bmu^\bl(\bb)$ is an affine space;

(2) $\bl \in \ca_{\bmu, \bb}^\tp$ if and only if $\l^\flat_\bullet = \e_\bl\i (\l_\bullet^\natural) \in \BB_\bmu^{\widehat G^d}(\ul_{G^d}(\bb))$. \\
Moreover, the maps $\bl \mapsto \l_\bullet^\flat$ and $\bl \mapsto \overline{X_\bmu^\bl(\bb)}$ induce a bijection $$\JJ_\bb \backslash \Irr^\tp X_\bmu(\bb) \cong \BB_\bmu^{\widehat G^d}(\ul_{G^d}(\bb)).$$
\end{thm}
As a consequence, for $C \in \Irr^\tp X_\bmu(\bb)$ there exists $\bl \in \ca_{\bmu, \bb}^\tp$ such that $\check C = \overline{X_\bmu^\bl(\bb)}$. Define $$\g^{G^d}(C) = \l^\flat_\bullet \in \BB_\bmu^{\widehat G^d}(\ul_{G^d}(\bb)).$$ Moreover, we write $$H(\bl) := H^{G^d}(C) = t^{\bl} K^d t^{\bmu} K^d t^{-\l^\dag_\bullet} \cap I^d = H_1(\bl) \times \cdots \times H_d(\bl),$$ where $H_\t(\l) := H_\t(C)$ for $1 \le \t \le d$ as in \S\ref{product}.

\begin{rmk} \label{cotype}
In \cite{HV}, the EL-charts for $X_\bmu(\bb)$ are parameterized by cocharacters $\bl$ in $\ca_{\bmu, \bb} \subseteq Y^d \cong (\ZZ^n)^{dl}$ with $l = \deg E/F$. By \cite[Corollary 4.18]{HV}, the map, sending $\bl$ to its {\it cotype}, induces a bijection $$\JJ_\bb \backslash \Irr^\tp X_\bmu(\bb) \cong \BB_\bmu^{\widehat G^d}(\ul_{G^d}(\bb)).$$ Following \cite[Definition 4.13]{HV}, the cotype of $\bl$ is equal to $\varepsilon_\bullet\i(\l_\bullet^\natural)$, where $\varepsilon_\bullet$ lies in $(W_0)^d \cong (\fs_n)^{dl}$ such that $$n\l_{i, j}(\varepsilon_{i, j}(k)) + \varepsilon_{i, j}(k) <  n\l_{i, j}(\varepsilon_{i, j}(k')) + \varepsilon_{i, j}(k')$$ for $1 \le i \le d$, $1 \le j \le l$, and $1 \le k' < k \le n$. This means that $(\l_{i, j})_{\varepsilon_{i, j}(\a)} \ge 0$ for all positive roots $\a$. So $\varepsilon_\bullet = \e_\bl$ and the cotype of $\bl$ coincides with $\l_\bullet^\flat$.
\end{rmk}

By the definition of $\g^{G^d}$, the second statement of Theorem \ref{reform} (for the minuscule case) follows from Theorem \ref{supb-minu}. It remains to show the first statement, which follows from the following result.
\begin{prop} \label{reform-minu}
Let $\bmu$ be minuscule and let $\bl \in \ca_{\bmu, \bb}^\tp$. For $1 \le a \le c \le d$, \begin{align*} &\quad\ \overline{ t^{-\l_a} (H_a(\bl) t^{\l_a^\dag} \times_K \cdots \times_K t^{-\l_c} H_c(\bl) t^{\l_c^\dag} K/K } \\  &= \e_{\l_a} \overline{ K_U t^{\l^\flat_a} \times_K \cdots \times_K K_U t^{\l^\flat_c} K / K } \\ &= \e_{\l_a} \overline{ S_{\l_a^\flat} \tilde\times \cdots \tilde\times S_{\l_c^\flat} }, \end{align*} where $\l_\bullet^\flat = (\l_1^\flat, \dots, \l_d^\flat)$ and $\l_\bullet^\dag = (\l_1^\dag, \dots, \l_d^\dag)$.
\end{prop}

As $G = \Res_{E / F} \GL_n$, we have $G^d(\brF) = \prod_{i = 1}^{dl} G_i(\brF)$, where $l = \deg E/F$ and each $G_i$ is isomorphic to $\GL_n$ over $E$. Moreover, $\bs$ sends $G_i$ to $G_{i-1}$ for $1 \le i \le dl$ with $G_{dl + 1} = G_1$. By Lemma \ref{dim} (1), we see that $H(\bl)$ only depends on $\bl \in \ca_{\bmu, \bb}^\tp$, the image of $\bb$ in $(\Omega_G)^d$ and the induced action of $\bs$ on the root system. Moreover, by Corollary \ref{indep}, we can assume, by replacing $b$ with a suitable $\Omega$-$\s$-conjugate, that $$\bb = (1, \dots, 1, b) \in \prod_{i = 1}^{dl} G_i(\brF).$$ Let $G' = \GL_n$ and let $\s_\bullet'$ be the Frobenius automorphism of $(G')^{dl}$ defined in \S\ref{conv-setup}. Via the natural identification (over $E$) $$(G')^{dl} = \prod_{i = 1}^{dl} G_i = G^d,$$ we see that the induced actions of $\s_\bullet'$ and $\bs$ on the root system coincide. Thus Proposition \ref{reform-minu} for the triple $(G = \Res_{E / F} \GL_n, d, \bb)$ is a consequence of its counterpart for the triple $(G' = \GL_n, dl, \bb)$. So we will assume that $G = \GL_n$ when $\bmu$ is minuscule.

\subsection{The minuscule case with $G = \GL_n$} Assume $G = \GL_n$. Let $T$ and $B$ be the group of diagonal matrices and the group of upper triangular matrices respectively. Let $V = \oplus_{i=1}^n \brF e_i$ be the natural representation of $G(\brF)$. Then there are natural identifications $X = \oplus_{i=1}^n \ZZ e_i$, $Y = \oplus_{i=1}^n \ZZ e_i^\vee$ and $W_0 = \fs_n$, where $(e_i^\vee)_{1 \le i \le n}$ is the dual basis to $(e_i)_{1 \le i \le n}$ and $\fs_n$ denotes the symmetric group. Then the natural action of $w \in W_0$ on $X$ is given by $w(e_i) = e_{w(i)}$. Moreover, we have $\Phi = \{\a_{i, j} = e_i - e_j; 1 \le i \neq j \le n\}$ and the simple roots are $\a_i = e_i - e_{i+1}$ for $1 \le i \le n-1$. Notice that $\Omega$ is a free abelian group of rank one. We fix a generator $\o \in \Omega$ which sends $e_i$ to $e_{i+1}$ for $i \in \ZZ$, where $e_{j + n} = t e_j$ for $j \in \ZZ$. So we can assume $b = \o^m$ for some $m \in \ZZ$. As $b$ is superbasic, $m$ is coprime to $n$. The embedding $h \mapsto (1, \dots, 1, h)$ induces an identification $Y = Y_\s \cong (Y^d)_{\bs}$, through which we have $\ul_G(b) = \ul_{G^d}(\bb)$.

Suppose $\bmu \in Y^d$ is minuscule. Let $\bl = (\l_1, \dots, \l_d) \in \ca_{\bmu, \bb}$. Following \S\ref{red-min} we can define $\l_\bullet^\dag = (\l_1^\dag, \dots, \l_d^\dag)$, $\l_\bullet^\natural = (\l_1^\natural, \dots, \l_d^\natural)$, $\l_\bullet^\flat = (\l_1^\flat, \dots, \l_d^\flat)$ and $H(\bl) = H_1(\bl) \times \cdots \times H_d(\bl)$. Notice that $\l_\t^\dag = \l_{\t + 1}$ for $1 \le \t \le d-1$.
\begin{lem} \label{red-superbasic}
Let $\bl \in \ca_{\bmu, \bb}$ and $1 \le a \le c \le d$. Then $$t^{-\l_a} H_a(\bl) t^{\l^\dag_a} \times_K \cdots \times_K  t^{-\l_c} H_c(\bl) t^{\l^\dag_c} K = K_{U_{\l_a}} t^{\l^\natural_a} \times_K \cdots \times_K K_{U_{\l_c}} t^{\l^\natural_c} K.$$
\end{lem}
\begin{proof}
Let $R_{\bmu, \bb}(\bl) = \sqcup_{\t=1}^d R_\t(\bl) \subseteq \sqcup_{\t=1}^d \Phi$ be as in \S\ref{def-R}, where $$R_\t(\bl) = \{\a \in \Phi; (\l_\t)_\a \ge 1, \<\a, \l^\natural_\t\> = \<\a, -\l_\t + \l_\t^\dag\> = -1 \}.$$ By the proof of Proposition \ref{dim}, we have $H_\t(\bl) = I_{\l_\t} \Sigma_\t$, where $$\Sigma_\t =\prod_{\a \in R_\t(\bl)} U_\a(t^{\<\a, \l_\t\> - 1} \co_{\brF}).$$  Thus $(\l_\t^\dag)_\a = (\l_\t)_\a - 1 \ge 0$ and $U_\a(t^{\<\a, \l_\t\> - 1} \co_{\brF}) = U_\a(t^{\<\a, \l^\dag_\t\>} \co_{\brF}) \subseteq I_{\l^\dag_\t}^+$ for $\a \in R_\t(\bl)$, which means $\Sigma_\t \subseteq I_{\l^\dag_\t}^+$ for $1 \le \t \le d$. As $\l^\dag_\t = \l_{\t+1}$ for $1 \le \t \le d-1$, we have \begin{align*} &\quad\ t^{-\l_a} H_a(\bl) t^{\l_a^\dag} \times_K \cdots \times_K t^{-\l_{c-1}} H_{c-1} t^{\l^\dag_{c-1}} \times_K t^{-\l_c} H_c(\bl) t^{\l^\dag_c} K \\ &= t^{-\l_a} H_a(\bl) t^{\l_{c+1}} \times_K \cdots \times_K t^{-\l_{c-1}} H_{c-1}(\bl) t^{\l_c} \times_K t^{-\l_c} H_c(\bl) t^{\l^\dag_c} K  \\ &= t^{-\l_a} I_{\l_a} \Sigma_a t^{\l_{a+1}} \times_K \cdots \times_K t^{-\l_{c-1}} I_{\l_{c-1}} \Sigma_{c-1} t^{\l_c} \times_K t^{-\l_c} I_{\l_c} \Sigma_c t^{\l^\dag_c} K  \\ &= t^{-\l_a} I_{\l_a} t^{\l_{a+1}} \times_K \cdots \times_K t^{-\l_{c-1}} I_{\l_{c-1}} t^{\l_c} \times_K t^{-\l_c} I_{\l_c} t^{\l^\dag_c} K  \\ &= t^{-\l_a} I_{\l_a}^+ I_{\l_a}^-  t^{\l_{a+1}} \times_K \cdots \times_K t^{-\l_{c-1}} I_{\l_{c-1}}^+ I_{\l_{c-1}}^-  t^{\l_c} \times_K (t^{-\l_c} I_{\l_c}^+ I_{\l_c}^-  t^{\l^\dag_c}) K  \\ &= t^{-\l_a} I_{\l_a}^+  t^{\l_{a+1}} \times_K \cdots \times_K t^{-\l_{c-1}} I_{\l_{c-1}}^+  t^{\l_c} \times_K t^{-\l_c} I_{\l_c}^+  t^{\l^\dag_c} K  \\ &= K_{U_{\l_a}} t^{\l^\natural_a} \times_K \cdots \times_K K_{U_{\l_c}} t^{\l^\natural_c} K, \end{align*} where the fifth equality follows from Lemma \ref{inc} that $I_{\l_\t}^- \subseteq I_{\l^\dag_\t}$ for $1 \le \t \le d$ since $\l^\natural_\bullet = -\bl+ \l^\dag_\bullet$ is minuscule. The proof is finished.
\end{proof}

Write $\e_{\bl} = (\e_1, \dots, \e_d) \in (\fs_n)^d$ with $\e_\t := \e_{\l_\t}^G$ for $1 \le \t \le d$. We define $a_{\t, i} = \e_\t(i) + n \l_\t(\e_\t(i))$ for $1 \le \t \le d$. By the definition of $\e_\t = \e_{\l_\t}$ (see also Remark \ref{cotype}), $a_{\t, 1} > \cdots > a_{\t, n}$ is the arrangement of the integers $i + n\l_\t(i)$ for $1 \le i \le n$ in the decreasing order. Define $w_{\bl} = (w_\t)_{1 \le \t \le d} \in (\fs_n)^d$ such that \begin{align*} \tag{$\ast$} a_{\t, i} = \begin{cases} a_{\t+1, w_\t(i)} - n\l^\flat_\t(i), & \text{ if } 1 \le \t \le d-1; \\ a_{1, w_d(i)} - n\l^\flat_d(i) + m, & \text{ if } \t=d. \end{cases} \end{align*}
\begin{lem} \label{relation}
We have $\e_\t = \e_1 w_1\i \cdots w_{\t-1}\i$ and hence $\l^\flat_\t = w_{\t-1} \cdots w_1 \e_1\i (\l^\natural_\t)$ for $1 \le \t \le d$.
\end{lem}
\begin{proof}
Suppose $2 \le \t \le d$. Then $a_{\t, i} = a_{\t-1, w_{\t-1}\i(i)} + n\l^\flat_{\t-1}(w_{\t-1}\i(i))$, that is, $$\e_\t(i) + n \l_\t(\e_\t(i)) = \e_{\t-1}(w_{\t-1}\i(i)) + n\l_{\t-1}(\e_{\t-1}(w_{\t-1}\i(i))) + n\l^\flat_{\t-1}(w_{\t-1}\i(i)),$$ which means $\e_\t(i) = \e_{\t-1}(w_{\t-1}\i(i))$. By induction we have $\e_\t = \e_1 w_1\i \cdots w_{\t-1}\i$.

% Suppose $\t = 1$. Then $a_{1, i} = a_{d, w_d\i(i)} + n\l^\flat_d(w_d\i(i)) - m$, that is, $$\e_1(i) + n \l_1(\e_1(i)) = \e_d(w_d\i(i)) + n\l_d(\e_d(w_d\i(i))) + n\l^\flat_d(w_d\i(i)) - m,$$ that is \begin{align*} \l^\flat_d(w_d\i(i)) = -\l_d(\e_d w_d\i(i)) + (\l_1(\e_1(i)) + \frac{\e_1(i) + m - \e_d(w_d\i(i))}{n}), \end{align*} which means $w_d(\l^\flat_d) = - w_d \e_d\i (\l_d) + w_d \e_d\i(\l^\dag_d) = w_d \e_d\i (\l^\natural_d)$, that is, $\l^\flat_d = \e_d\i(\l^\natural_d) = w_{d-1} \cdots w_1 \e_1\i (\l^\natural_d)$. The proof is finished.
\end{proof}

\begin{lem} \label{order}
Let $1 \le \t \le d$ and $1 \le i < j \le n$. We have

(1) $w_\t(i) > w_\t(j)$, that is, $a_{\t+1, w_\t(i)} < a_{\t+1, w_\t(j)}$ if and only if $a_{\t, i} - a_{\t, j} < n$ and $\l^\flat_\t(j) - \l^\flat_\t(i) = 1$, in which case $a_{\t+1, w_\t(j)} - a_{\t+1, w_\t(i)} < n$.

(2) $\ell(w_\t) = |\{\a \in \Phi; (\l_\t)_\a \ge 0, (\l_\t^\dag)_\a < 0\}|$.
\end{lem}
\begin{proof}
By ($\ast$) we have \begin{align*} \tag{i} a_{\t+1, w_\t(i)} - a_{\t+1, w_\t(j)} &= a_{\t, i} - a_{\t, j} + n(\l^\flat_\t(i) - \l^\flat_\t(j)) \\ &= (\e_\t(i) + n\l_\t(\e_\t(i))) - (\e_\t(j) + n\l_\t(\e_\t(j))).\end{align*} Then the first statement follows from that $\l^\flat_\t$ is minuscule. For $1 \le k \neq l \le n$ we ave $\l_{\a_{k, l}} \ge 0$ if and only if $k + n\l(k) > l + n\l(l)$. Then it follows from (i) that the map $\g \mapsto \e_\t(\g)$ gives a bijection between $\Phi^+ \cap -w_\t\i(\Phi^+)$ and the set $\{\a \in \Phi; (\l_\t)_\a \ge 0, (\l_\t^\dag)_\a < 0\}$. The second statement is proved.
\end{proof}

For $w \in \fs_n$ we denote by $\supp(w)$ the set of integers $1 \le i \le n-1$ such that the simple reflection $s_{\a_i}$ appears in some/any reduced expression of $w$.
\begin{lem} \label{minus}
Let $1 \le \t \le d$ and $1 \le i \le n-1$ such that $i \in \supp(w_\t)$. Then there are roots $\a, \b \ge \a_i$ such that $w_\t\i(\a) < 0$ and $w_\t(\b) < 0$. As a consequence, $a_{\t, i} - a_{\t, i+1} < n$ and $a_{\t+1, i} - a_{\t+1, i+1} < n$.
\end{lem}
\begin{proof}
The first statement follows from that $i \in \supp(w_\t)$. Let $\a = \a_{j, j'} \in \Phi^+$ such that $w_\t\i(\a) < 0$ and $\a_i \le \a$. In other words, $j \le i < i+1 \le j'$ and $w_\t\i(j) > w_\t\i(j')$. Then we have $a_{\t+1, i} - a_{\t+1, i+1} \le a_{\t+1, j} - a_{\t+1, j'} < n$ by Lemma \ref{order}. The inequality $a_{\t, i} - a_{\t, i+1} < n$ follows in a similar way.
\end{proof}

\begin{lem} \label{dim-1}
We have $\dim X_{\bmu}^{\bl}(\bb) = \<\rho_\bullet, \bmu - \l^\flat_\bullet\> - \ell(w_{\bl})$. Here $\rho_\bullet$ denotes the half sum of positive roots of $G^d$.
\end{lem}
\begin{proof}
We set \begin{align*} E &= \{(\t, i, j); 1 \le \t \le d, 1 \le i < j \le n, \l^\flat_\t(j) - \l^\flat_\t(i) = 1\} \\ E' &= \{(\t, i, j) \in E; a_{\t, i} - a_{\t, j} > n\} \\ E'' &=\{(\t, i, j); 1 \le \t \le d, 1 \le i < j \le n, w_\t(i) > w_\t(j)\}.\end{align*} Then $E = E' \sqcup E''$ by Lemma \ref{order}. Applying Proposition \ref{dim} (3) we have $$\dim X_{\bmu}^{\bl}(\bb) = |E'| = |E|- |E''| = \<\rho_\bullet, \overline{\l^\flat_\bullet} - \l^\flat_\bullet\> - \ell(w_{\bl}) = \<\rho_\bullet, \bmu - \l^\flat_\bullet\> - \ell(w_{\bl}),$$ where $\overline{\l^\flat_\bullet}$ denotes the dominant conjugate of $\l^\flat_\bullet$, which equals $\bmu$ by Proposition \ref{dim}.
\end{proof}

\begin{lem}\label{Coxeter}
If $\bl \in \ca_{\bmu, \bb}^\tp$, then $\ell(w_{\bl}) = \sum_{\t=1}^d \ell(w_\t) = n-1$, and $w_d \cdots w_1 \in \fs_n$ is a product of distinct simple reflections.
\end{lem}
\begin{proof}
Let $\l_{m, n} \in \ZZ^n$ such that $\l_{m, n}(i) = \lfloor \frac{im}{n} \rfloor - \lfloor \frac{(i-1)m}{n} \rfloor$ for $1 \le i \le n$. As $\bl \in \ca_{\bmu, \bb}^\tp$, It follows from \cite[\S 4.4]{HV} that $\sum_{\t=1}^d \l^\flat_\t = \l_{m, n}$ and $\dim X_{\bmu}^{\bl}(\bb) = \<\rho_\bullet, \bmu\> - \frac{n-1}{2}$. By Lemma \ref{dim-1}, \begin{align*} \dim X_{\bmu}^{\bl}(\bb) &= \<\rho_\bullet, \bmu - \l^\flat_\bullet\> - \ell(w_{\bl}) \\ &= \<\rho_\bullet, \bmu\>  - \<\rho, \l_{m, n}\> - \ell(w_{\bl}) \\ &= \<\rho_\bullet, \bmu\> + \frac{n-1}{2} - \ell(w_{\bl}) \\ &= \<\rho_\bullet, \bmu\>  - \frac{n-1}{2},\end{align*} where $\rho$ is the half sum of positive roots of $\GL_n$. Thus $\ell(w_{\bl}) = n-1$. Moreover, by ($\ast$) we see that $\e_1(i) - \e_1 w_d \cdots w_1(i) \equiv a_{1, i} - a_{1, w_d \cdots w_1(i)} \equiv  m \mod n$ for $1 \le i \le n$. So $w_d \cdots w_1 \in \fs_n$ acts on $\{1, \dots, n\}$ transitively as $m$ is coprime to $n$. This means that $n-1 \le \ell(w_d \cdots w_1) \le \ell(w_{\bl}) = n-1$ and $w_d \cdots w_1$ is a product of distinct simple reflections as desired.
\end{proof}

\begin{lem} \label{nonneg}
Let $\bl \in \ca_{\bmu, \bb}^\tp$. Let $1 \le \t \le d-1$ and $i \in \supp(w_\t)$. Then $\sum_{k = \t+1}^\iota \<\a_i, \l^\flat_k\> \in \ZZ_{\ge 0}$ for $\t+1 \le \iota \le d$.
\end{lem}
\begin{proof}
By Lemma \ref{Coxeter}, $i \notin \supp(w_l)$ for $\t+1 \le l \le d$, and moreover, there exists at most one integer $\t+1 \le \t' \le \iota$ (resp. $\t+1 \le \t'' \le \iota$) such that $i-1 \in \supp(w_{\t'})$ (resp. $i+1 \in \supp(w_{\t''})$). Without loss of generality, we assume such $\t', \t''$ exist. Therefore,

(i) for $\t+1 \le k \le \iota$ we have: (1) $w_k(i) \neq i$ if and only if $k = \t'$ and $w_{\t'}(i) < i$; (2) $w_k(i+1) \neq i+1$ if and only if $k = \t''$ and $w_{\t''}(i+1) > i+1$.

Using (i) and the equality from ($\ast$) $$\<\a_i, \l^\flat_k\> = (a_{k+1, w_k(i)} - a_{k+1, w_k(i+1)}) - (a_{k, i} - a_{k, i+1})$$ we deduce that \begin{align*} & \quad\ \sum_{k = \t+1}^\iota \<\a_i, \l^\flat_k\> \\ &= \frac{a_{\t'+1, w_{\t'}(i)} - a_{\t'+1, i}}{n} + \frac{a_{\iota+1, i} - a_{\iota+1, i+1}}{n} + \frac{a_{\t''+1, i+1} - a_{\t''+1, w_{\t''}(i+1)}}{n} - \frac{a_{\t+1, i} - a_{\t+1, i+1}}{n} \\ & \ge 0, \end{align*} where the inequality follows from that $a_{\t+1, i} - a_{\t+1, i+1} < n$ (by Lemma \ref{minus}) and that $a_{\t'+1, w_{\t'}(i)} - a_{\t'+1, i}, a_{\iota+1, i} - a_{\iota+1, i+1}, a_{\t''+1, i+1} - a_{\t''+1, w_{\t''}(i+1)} > 0$.
\end{proof}

\begin{proof}[Proof of Proposition \ref{reform-minu}]
By Lemma \ref{relation} we have $\l^\natural_k = \e_{\l_a} (w_{k-1} \cdots w_a)\i (\l^\flat_k)$ and $U_{\l_k} = {}^{\e_{\l_a} (w_{k-1} \cdots w_a)\i} U$ for $a \le k \le d$. By Lemma \ref{red-superbasic}, \begin{align*} &\quad\ t^{-\l_a} H_a(\bl) t^{\l_a^\dag} \times_K \cdots \times_K t^{-\l_c} H_c(\bl) t^{\l_c^\dag} K/K \\ &= K_{U_{\l_a}} t^{\l^\natural_a} \times_K \cdots \times_K K_{U_{\l_c}} t^{\l^\natural_c} K/K \\ &= \e_{\l_a} K_U t^{\l^\flat_a} \times_K w_a\i K_U t^{\l^\flat_{a+1}} w_a \times_K \cdots \times_K (w_{c-1} \cdots w_a)\i K_U t^{\l^\flat_c} w_{c-1} \cdots w_a K/K \\ &= \e_{\l_a} K_U t^{\l^\flat_a} \times_K w_a\i K_U t^{\l^\flat_{a+1}} \times_K \cdots \times_K w_{c-1}\i K_U t^{\l^\flat_c} K/K \end{align*} Therefore, it suffices to show that for $a \le \t \le c-1$ and $i \in \supp(w_\t)$ we have $$s_i \overline{ K_U t^{\l^\flat_{\t+1}} \times_K \cdots \times_K K_U t^{\l^\flat_c} K/K } = \overline{ K_U t^{\l^\flat_{\t+1}} \times_K \cdots \times_K K_U t^{\l^\flat_c} K/K }.$$ Set $U_i = U_{\a_i}$, $U_{-i} = U_{-\a_i}$ and $U^i = \prod_{0 < \a \neq \a_i} U_\a$ for $1 \le i \le n-1$. Then $U=U_i U^i = U^i U_i$ and $U^i$ is normalized by $U_i$ and $U_{-i}$. As we can take $s_i = U_i(-1) U_{-i}(1) U_i(-1)$, the displayed equality above is equivalent to $$U_{-i}(1) \overline { K_U t^{\l^\flat_{\t+1}} \times_K \cdots \times_K K_U t^{\l^\flat_c} K/K } = \overline{ K_U t^{\l^\flat_{\t+1}} \times_K \cdots \times_K K_U t^{\l^\flat_c} K/K }.$$

Define $f_\iota$ for $\t \le \iota \le d$ such that $f_\t = 1$ and $f_\iota = t^{e_\iota} f_{\iota-1} / (1 + z_\iota f_{\iota-1})$ with $e_\iota = \<\a_i, \l^\flat_\iota\>$ for $\t+1 \le \iota \le c$. We claim that

(i) for generic points $(z_{\t+1}, \dots, z_c) \in (\co_\brF)^{c - \t}$ we have $1 + z_\iota f_{\iota - 1} \in \co_{\brF}^\times$ and $f_\iota \in t^{\sum_{k=\t+1}^\iota e_k} \co_{\brF} \subseteq \co_{\brF}$ for $\t+1 \le \iota \le c$.

if $\iota = \t + 1$, the claim is true by taking $z_{\t+1} \in \co_\brF \setminus \{-1\}$. Suppose it is true for $\iota-1$. Then $f_{\iota - 1} \in t^{\sum_{k=\t+1}^{\iota - 1} e_k} \co_{\brF} \subseteq \co_{\brF}$ by Lemma \ref{nonneg}. So there exist generic points $z_\iota \in \co_\brF$ such that
$1 + z_\iota f_{\iota - 1} \in \co_{\brF}^\times$ and hence $f_\iota = t^{e_\iota} f_{\iota-1} / (1 + z_\iota f_{\iota-1}) \in t^{\sum_{k=\t+1}^\iota e_k} \co_{\brF}$ as desired. The claim (i) is proved.

Let $(z_{\t+1}, \dots, z_c) \in (\co_\brF)^{c - \t}$ be a generic point as in (i). Using (i) and the commutator relation $$U_{-\a}(f) U_\a(z) = U_\a(\frac{z}{1 + zf}) (1 + zf)^{-\a^\vee} U_{-\a}(\frac{f}{1+zf}) \text{ for } 1 + zf \neq 0,$$ we deduce that \begin{align*}  &\quad\ U_{-i}(1) K_{U^i} U_i(z_{\t+1}) t^{\l^\flat_{\t+1}} \times_K \cdots \times_K K_{U^i} U_i(z_c)  t^{\l^\flat_c} K/K  \\ &\subseteq K_B t^{\l^\flat_{\t+1}} \times_K U_{-i}(f_{\t+1}) K_{U^i} U_i(z_{\t+2}) t^{\l^\flat_{\t+2}} \times_K \cdots \times_K K_{U^i} U_i(z_c)  t^{\l^\flat_c} K/K \\ &\ \vdots \\ &\subseteq K_B t^{\l^\flat_{\t+1}} \times_K \cdots \times_K K_B  t^{\l^\flat_c} K/K \\ &= K_U t^{\l^\flat_{\t+1}} \times_K \cdots \times_K K_U t^{\l^\flat_c} K/K. \end{align*} Therefore, $K_U t^{\l^\flat_{\t+1}} \times_K \cdots \times_K K_U t^{\l^\flat_c} K/K$ contains an open dense subset of $U_{-i}(1) K_U t^{\l^\flat_{\t+1}} \times_K \cdots \times_K K_U t^{\l^\flat_c} K/K$ as desired.
\end{proof}

\subsection{The general case} Finally we consider the general case where $\bmu \in Y^d$ is an arbitrary dominant cocharacter. The strategy is to reduce it to the minuscule case considered in the previous subsection.

As $G = \Res_{E / F} \GL_n$, there exist $e \in \ZZ_{\ge d}$, a minuscule dominant cocharacter $\bup \in Y^e$ and a sequence $\Sigma$ of integers $1 = k_1 < \cdots < k_d < k_{d+1} = e+1$ such that $\mu_\t = \upsilon_{k_\t} + \cdots + \upsilon_{k_{\t+1} - 1}$ for $1 \le \t \le d$. Let $\pr_\Sigma: G^e \to G^d$ be the projection given by $(g_1, \dots, g_e) \mapsto (g_{k_1}, \dots, g_{k_d})$. By abuse of notation, we still denote by $\bb$ the element $(1, \dots, 1, b)$ in $G^e(\brF)$ and by $\bs$ the Frobenius of $G^e$ given by $(g_1, \dots, g_e) \mapsto (g_2, \dots, g_e, \s(g_1))$. Then there is a Cartesian square \begin{align*} \xymatrix{
  X_{\bup}(\bb) \ar[d]^{\pr_\Sigma} \ar[r] & G(\brF) \times_K \Gr_{\bup} \ar[d]^{\Id \times_K m_\bup^\Sigma} \\
\cup_{\bbeta \leq \bmu} X_{\bbeta}(\bb) \ar[r] & G(\brF) \times_K \Gr_{\bmu},  }\end{align*} where $m_\bup^\Sigma: \Gr_{\bup} \to \Gr_\bmu$ is the partial convolution map given by $$(g_1, \dots, g_{e-1}, g_e K) \mapsto (g_{k_1} \cdots g_{k_2 -1}, \dots, g_{k_{d-1}} \cdots g_{k_d - 1}, g_{k_d} \cdots g_e K);$$ the top horizontal map is given by $$(g_1 K, \dots, g_e K) \mapsto (g_1, g_1\i g_2, \dots, g_{e-1}\i g_e, g_e\i b \s(g_1) K);$$ the bottom horizontal map is given by $$(h_1 K, \dots, h_d K) \mapsto (h_1, h_1\i h_2, \dots, h_{d-1}\i h_d, h_d\i b \s(h_1) K).$$

For a dominant cocharacter $\bbeta \in Y^d$ we denote by $m_\bup^\bbeta$ the multiplicity with which $V_\bbeta^{\widehat G^d}$ appears in $V_\bup^{\widehat G^e}$. Here we view each $\widehat G^e$-crystal as a $\widehat G^d$-crystal via the embedding $\widehat G^d \hookrightarrow \widehat G^e$ given by $(h_1, \dots, h_d) \mapsto (h_1^{(k_2 - k_1)}, \dots, h_d^{(k_{d+1} - k_d)})$.
\begin{prop} \label{multp} \label{supbasic}
We have $$|\JJ_{\bb} \backslash \Irr^\tp X_{\bup}(\bb)| = \sum_{\bbeta \leq \bmu} m_\bup^\bbeta |\JJ_{\bb} \backslash \Irr^\tp X_{\bbeta}(\bb)|.$$ As a consequence, $|\JJ_{\bb} \backslash \Irr^\tp X_\bmu(\bb)| = \dim V_\bmu^{\widehat G^d}(\ul_G(b))$.
\end{prop}
\begin{proof}
The first statement follows similarly as Corollary \ref{compare}. To show the second one, we argue by induction on $|\bmu|$. If $\bmu$ minuscule, it is proved in Theorem \ref{supb-minu}. Suppose it is true for $|\bbeta| < |\bmu|$. By the choice of $\bup$ we have $m_\bup^\bmu = 1$. Therefore, \begin{align*} |\JJ_{\bb} \backslash \Irr^\tp X_{\bup}(\bb)| &= \sum_{\bbeta \leq \bmu} m_\bup^\bbeta |\JJ_{\bb} \backslash \Irr^\tp X_{\bbeta}(\bb)| \\ &= |\JJ_{\bb} \backslash \Irr^\tp X_\bmu(\bb)| + \sum_{\bbeta < \bmu} m_\bup^\bbeta \dim V_\bbeta(\ul_G(b)) \\ &= \dim V_\bup(\ul_G(b))  \\ &= \sum_{\bbeta \leq \bmu} m_\bup^\bbeta \dim V_\bbeta(\ul_G(b)), \end{align*} where the second equality follows from the induction hypothesis, and the last equality follows again from Theorem \ref{supb-minu} as $\bup$ is minuscule. Therefore, we have $|\JJ_{\bb} \backslash \Irr^\tp X_\bmu(\bb)| = \dim V_\bmu(\ul_G(b))$ as desired.
\end{proof}

Similar to the definition of $\otimes$ in Theorem \ref{conv}, let $\otimes_\Sigma: \BB_\bup^{\widehat G^e} \to  \sqcup_{\bbeta} \BB_\bbeta^{\widehat G^d}$ denote the map given by $$(\d_1, \dots, \d_e) \mapsto (\d_{k_1} \otimes \cdots \otimes \d_{k_2 - 1}, \dots, \d_{k_d} \otimes \cdots \otimes \d_{k_{d+1} - 1}).$$

\begin{proof}[Proof of Theorem \ref{reform}]
Let $C \in \Irr^\tp X_\bmu(\bb)$. By Theorem \ref{supb-minu} and Proposition \ref{multp}, there exists $\bxi \in \ca_{\bup, \bb}^\tp$ such that $\xi^\flat_\bullet \in \BB_\bup^{\widehat G^e}(\ul_G(b))$ and $C = \pr_\Sigma(\overline{X_\bup^\bxi(\bb)})$. Write $\bxi = (\xi_1, \dots, \xi_e)$, $\xi_\bullet^\dag = \bb \bs(\bxi) = (\xi_1^\dag, \dots, \xi_e^\dag)$ and $\xi_\bullet^\flat = (\xi_1^\flat, \dots, \xi_e^\flat)$. Define $$\g^{G^d}(C) = \otimes_\Sigma(\xi_\bullet^\flat) = (\g_1, \dots, \g_d) \in \sqcup_\bbeta \BB_\bbeta^{\widehat G^d}(\ul_G(b)),$$ where $\g_\t = \xi_{k_\t}^\flat \otimes \cdots \otimes \xi_{k_{\t+1}-1}^\flat \in \BB^{\widehat G} := \sqcup_\eta \BB_\eta^{\widehat G}$ for $1 \le \t \le d$.

Let $\bl = \pr_\Sigma(\bxi) \in Y^d$. Then $(I^d t^\bl K^d/K^d) \cap C$ is open dense in $C$. So $$\overline{ (\th_\bl^{G^d})\i(C)} = \overline {\pr_\Sigma( (\th_\bxi^{G^e})\i(X_\bup^\bxi(\bb)))} \subseteq I^d,$$ which means (by the proof of Lemma \ref{convolution}) that, for each $1 \le \t \le d$, $$\overline{H_\t(C)} = \overline{ H_{k_\t}(\bxi) \cdots H_{k_{\t+1}-1}(\bxi) },$$ where $H_\t(C)$ and $H(\bxi)$ are defined in \S\ref{product} and \S\ref{red-min} respectively. Thus for $1 \le a \le c \le d$, \begin{align*} &\quad\ \overline{ t^{-\l_a}H_a(C)t^{\l_a^\dag} \times_K \cdots \times_K t^{-\l_c}H_c(C)t^{\l_c^\dag} K/K  } \\ &= \overline{ m_\bup^\Sigma(t^{-\xi_{k_a}} H_{k_a}(\bxi) t^{\xi_{k_a}^\dag} \times_K \cdots \times_K t^{-\xi_{k_{c+1} - 1}} H_{k_{c+1} - 1}(\bxi) t^{\xi_{k_{c+1} - 1}^\dag} K / K) } \\ &= \overline{ m_\bup^\Sigma( \e_{\l_a} K_U t^{\xi_{k_a}^\flat} \times_K \cdots \times_K K_U t^{\xi_{k_{c+1} - 1}^\flat} K/K ) }  \\ &= \e_{\l_a} \overline{ S_{\g_a} \tilde\times \cdots \tilde\times S_{\g_c} }, \end{align*} where the first equality follows from that $\l_\t = \xi_{k_\t}$ and $\l^\dag_\t = \xi^\dag_{k_{\t+1} - 1}$ for $1 \le \t \le d$; the second equality follows from Proposition \ref{reform-minu}. In particular, we have $$\overline {t^{-\l_\t} H_\t(C) t^{\l_\t^\dag} K/K} = \e_{\l_\t} \overline{ S_{\g_\t} }$$ by taking $a = c = \t$. On the other hand, as $C \subseteq X_\bmu(\bb)$ it follows that $$t^{-\l_\t} H_\t(C) t^{\l_\t^\dag} K/K \subseteq \Gr_{\mu_\t}^\circ.$$ Thus, $\g_\t \in \BB_{\mu_\t}^{\widehat G}$ and hence $\g^{G^d}(C) \in \BB_\bmu^{\widehat G^d}(\ul_G(b))$. Now the first statement of Theorem \ref{reform} follows by taking $a=1$ and $c=d$.

As $\bb$ is superbasic, $\JJ_\bb = (\Omega^d \cap \JJ_\bb) (I^d \cap \JJ_\bb)$. By Lemma \ref{change} and Lemma \ref{irr}, the map $C \mapsto \g^{G^d}(C)$ defined in the previous paragraph induces a map $$\JJ_\bb \backslash \Irr^\tp X_\bmu(\bb) \to \BB_\bmu^{\widehat G^d}(\ul_G(b)).$$ Then we have the following commutative diagram \begin{align*}\xymatrix{
  \Irr X_\bup(\bb) \ar[d]_{\pr_\Sigma} \ar[r]^{\g^{G^e}} & \BB_\bmu^{\widehat G^e}(\ul_G(b)) \ar[d]^{\otimes_\Sigma} \\
  \sqcup_\bbeta \Irr X_\bbeta(b) \ar[r]^{\g^{G^d}} & \sqcup_\bbeta \BB_\bbeta^{\widehat G^d}(\ul_G(b)).  } \end{align*}
As $\g^{G^e}$ is bijective and $m_\bup^\bmu = 1$, the map $\JJ_\bb \backslash \Irr^\tp X_\bmu(\bb) \overset {\g^{G^d}} \longrightarrow \BB_\bmu^{\widehat G^d}(\ul_G(b))$ is surjective and hence bijective by Proposition \ref{multp}.
\end{proof}

\section{Proof of Theorem \ref{main} and \ref{conv}} \label{sec-main}

\subsection{Irreducible components of $S_{\mu, \eta}^N$} \label{part} Let $P = M N$ and $b \in M(\brF)$ be as in \S \ref{Levi}. For $\mu \in Y^+$ let $I_{\mu, M}$ be the set of $M$-dominant cocharacters $\eta$ such that $$S_{\mu, \eta}^N := N(\brF) t^\eta K/K \cap \Gr_\mu^\circ \neq \emptyset.$$ Define $I_{\mu, b, M} = \{\eta \in I_{\mu, M}; \eta = \k_M(b) \in \pi_1(M)_\s\}$.
\begin{prop} [{\cite[Proposition 5.4.2]{GHKR}}] \label{res}
Let $\eta \in I_{\mu, M}$, then $$\dim S_{\mu, \eta}^N \le \<\rho, \mu + \eta\> - 2\<\rho_M, \eta\>.$$ Moreover, let $\Sigma_{\mu, \eta}^N$ be the set of irreducible components of $S_{\mu, \eta}^N$ with the maximal possible dimension $\<\rho, \mu + \eta\> - 2\<\rho_M, \eta\>$. Then $|\Sigma_{\mu, \eta}^N|$ equals the multiplicity with which $\BB_\eta^{\widehat M}$ appears in $\BB_\mu^{\widehat G}$.
\end{prop}

For $\eta \in I_{\mu, M}$ recall that $\th_\eta^N: N(\brF) \to \Gr_P$ is the map given by $n \mapsto n t^\eta K_P$. Let $Z^N \in \Irr S_{\mu, \eta}^N$ and $g = h t^\eta h' \in K_M t^\eta K_M$ with $h, h' \in K_M$. We define \begin{align*}(\th_\eta^N)\i(Z^N) \ast g &= h (\th_\eta^N)\i(Z^N) h\i g  \subseteq P(\brF); \\ Z^N \ast (gK_M) &= (\th_\eta\i(Z^N) \ast g) K_P/K_P \subseteq \Gr_P, \end{align*} which do not depend on the choices of $h, h' \in K_M$ since the connected group $K_M \cap t^\eta K_M t^{-\eta}$ fixes $S_{\mu, \eta}^N$ and hence fixes each of its irreducible components, by left multiplication. For $\cd^M \subseteq K_M t^\eta K_M$ we set \begin{align*} (\th_\eta^N)\i(Z^N) \ast \cd^M &= \cup_{g \in \cd^M} (\th_\eta^N)\i(Z^N) \ast g; \\ Z^N \ast (\cd^M K_M/K_M) &= \cup_{g \in \cd^M} Z^N \ast (g K_M). \end{align*} Notice that $Z^N \ast (\cd^M K_M/K_M) = ((\th_\eta^N)\i(Z^N) \ast \cd^M) K_P/K_P$.

\begin{lem} \label{star}
Let $Z^N \in \Irr S_{\mu, \eta}^N$ and $g \in K_M t^\eta K_M$. Then we have

(1) $h ((\th_\eta^N)\i(Z^N) \ast g) = (\th_\eta^N)\i(Z^N) \ast (h g)$ for $h \in K_M$;

(2) $u ((\th_\eta^N)\i(Z^N) \ast g) = (\th_\eta^N)\i(Z^N) \ast g$ for $u \in K_N$.
\end{lem}
\begin{proof}
The first statement follows by definition. The second one follows from that $K_N Z^N = Z^N$ since $K_N$ fixes $S_{\mu, \eta}^N$ and hence fixes each of its irreducible components by left multiplication.
\end{proof}

\subsection{Iwasawa decomposition of $X_\mu(b)$} Notice that the natural projection $P = M N \to M$ induces a map $$\b: X_\mu(b) \hookrightarrow \Gr_G = \Gr_P \to  \Gr_M.$$ Let $\eta \in I_{\mu, b, M}$ and let $X_\eta^M(b)$ be the affine Deligne-Lusztig variety defined for $M$.  For $Z^N \in \Irr S_{\mu, \eta}^N$ and $C^M \subseteq X_\eta^M(b)$ we define $$X_\mu^{Z^N, C^M}(b) = \{g K_P \in \b\i(C^M); g\i b \s(g) K_P \in Z^N \ast \Gr_{\eta, M}^\circ\} \subseteq \Gr_P,$$ where $\Gr_{\eta, M}^\circ = K_M t^\eta K_M / K_M$. Notice that the natural projection $$Z^N \ast \Gr_{\eta, M}^\circ \to \Gr_{\eta, M}^\circ$$ is a fiber bundle with fibers isomorphic to $Z^N$.
\begin{prop} \label{relative}
Let $C^M \subseteq X_\eta^M(b)$ be locally closed and irreducible. Then

(1) $\b\i(C^M) = \cup_{Z^N \in \Irr S_{\mu, \eta}^N} X_\mu^{Z^N, C^M}(b)$;

(2) $\dim X_\mu^{Z^N, C^M}(b) \le \dim X_\mu(b)$, where the equality holds if and only if $\dim C^M = \dim X_\eta^M(b)$ and $Z^N \in \Sigma_{\mu, \eta}^N$;

(3) $N(\brF) \cap \JJ_b$ acts transitively on $\Irr X_\mu^{Z^N, C^M}(b)$;

(4) $X_\mu^{Z^N, C^M}(b)$ is irreducible if $M$ is the centralizer of $\nu_b$;

(5) $t^{-\l} H^P(C) t^{b\s(\l)} = (\th_\eta^N)\i(Z^N) \ast (t^{-\l} H^M(C^M) t^{b\s(\l)})$ if $C \in \Irr  X_\mu^{Z^N, C^M}(b)$ and $C^M \in \Irr X^{\l, M}_\eta$ for some $\l \in Y$.

Here $X_\eta^{\l, M}(b) = I_M t^\l K_M/K_M \cap X_\eta^M(b)$, and $H^M(C^M)$ is define in \S\ref{Levi} for $G = M$.
\end{prop}
\begin{proof}
Let $m K_M \in X_\eta^M(b)$. By definition, $$\b\i(m K_M) = \{m n K_P; n\i  m\i b \s(m)\s(n) K_P \in S_{\mu, \eta}^N \ast (b_m K_M)\}.$$ So the statement (1) follows. Moreover, as $\nu_M(b) = \nu_G(b)$ is dominant, it follows from \cite[Proposition 5.3.2]{GHKR} that $$\dim (X_\mu^{Z^N, C^M}(b) \cap \b\i(m K_M)) = \dim Z^N - \<2\rho_N, \eta\>.$$ Therefore, by Proposition \ref{res} we have \begin{align*} \dim X_\mu^{Z^N, C^M}(b) &= \dim C^M + \dim Z^N - 2\<\rho, \eta\> \\ &\le \<\rho_M, \eta\> - \frac{1}{2}\df_M(b) + \<\rho, \mu + \eta\> - 2\<\rho_M, \eta\> - 2\<\rho_N, \eta\> \\ &= \<\rho, \mu\> -\<\rho_N, \eta\> - \frac{1}{2}\df_G(b) \\ &= \<\rho, \mu - \nu_M(b)\> - \frac{1}{2}\df_G(b)  \\ &= \dim X_\mu(b), \end{align*} where the equality holds if and only if $\dim C^M = \dim X_\eta^M(b)$ and $Z^N \in \Sigma_{\mu, \eta}^N$. So the statement (2) follows.

The statement (3) follows similarly as \cite[Proposition 5.6]{HV} which deals with the minuscule case. Notice that for minuscule $\mu$ the sets $S_{\mu, \eta}^N = K_N t^\eta K/K$ and $N(\brF) \cap t^{-\eta} K t^\mu K = t^{-\eta} (\th_\eta^N)\i(S_{\mu, \eta}^N) t^\eta$ are irreducible. For general case we only needs to replace $N(\brF) \cap t^{-\eta} K t^\mu K$ in \cite[Claim 1 on page 1630]{HV} with the irreducible set $t^{-\eta} (\th_\eta^N)\i (Z^N) t^\eta$ for $Z^N \in \Irr S_{\mu, \eta}^N$.

The statement (4) follows from the statement (3) and that $N(\brF) \cap \JJ_b = \{1\}$.

It follows by definition that $$t^{-\l} H^P(X^{Z^N, C^M}(b)) t^{b\s(\l)} = (\th_\eta^N)\i(Z^N) \ast (t^{-\l} H^M(C^M) t^{b\s(\l)}).$$ By the statement (3), all the irreducible components of $X^{Z^N, C^M}(b)$ are conjugate under $N(\brF) \cap \JJ_b$. Thus $H^P(C) = H^P(X^{Z^N, C^M}(b))$ for $C \in \Irr X^{Z^N, C^M}(b)$, and the statement (5) follows.
\end{proof}

\begin{cor} \label{inequality}
The map $$(Z^N, C^M) \mapsto \JJ_b^M \Irr \overline{X_\mu^{Z^N, C^M}(b)} = (P(\brF) \cap \JJ_b) \Irr \overline{X_\mu^{Z^N, C^M}(b)}$$ induces a bijection $$(P(\brF) \cap \JJ_b) \backslash \Irr^\tp X_\mu(b) \cong \sqcup_{\eta \in I_{\mu, b, M}} \Sigma_{\mu, \eta}^N \times (\JJ_b^M \backslash \Irr^\tp X_\eta^M(b)).$$ As a consequence, $|\JJ_b \backslash \Irr^\tp X_\mu(b)| \le \dim V_\mu(\ul_G(b))$.
\end{cor}
\begin{proof}
The bijection follows from Proposition \ref{relative} (1), (2), (3). Choose $P = M N$ such that $b$ is superbasic in $M(\brF)$. Then \begin{align*} |\JJ_b \backslash \Irr^\tp X_\mu(b)| &\le |(P(\brF) \cap \JJ_b) \backslash \Irr^\tp X_\mu(b)| \\ &=\sum_{\eta \in I_{\mu, b, M}} |\Sigma_{\mu, \eta}^N| |\JJ_b^M \backslash \Irr^\tp X_\eta^M(b)| \\ &= \sum_{\eta \in I_{\mu, b, M}} |\Sigma_{\mu, \eta}^N| \dim V_\eta^{\widehat M} (\ul_M(b)) \\ &= \dim V_\mu(\ul_M(b)) \\ &= \dim V_\mu(\ul_G(b)), \end{align*} where the second equality follows from Proposition \ref{multp} dealing with the subperbasic case, and the last one follows from that $\ul_M(b) = \ul_G(b)$.
\end{proof}

\subsection{The numerical identity} In this subsection we prove the numerical version of Theorem \ref{main}.
\begin{prop}\label{minu}
We have $|\JJ_b \backslash \Irr^\tp X_\mu(b)| = \dim V_\mu(\ul_G(b))$ if $\mu$ is minuscule and $b$ is basic.
\end{prop}
The proof is given in \S \ref{sec-minu}.

\begin{lem} \label{appear}
If $G$ is simple, adjoint and has some nonzero minuscule cocharacter, then each irreducible $\widehat G$-module appears in some tensor product of irreducible $\widehat G$-modules with minuscule highest weights.
\end{lem}
\begin{proof}
Let $\mu \in Y^+$. By the assumption on $G$, there exists a dominant and minuscule cocharacter $\bmu \in Y^d$ for some $d \in \ZZ_{\ge 1}$ such that $\mu \leq |\bmu|$ and hence $\Gr_\mu \subseteq m_\bmu(\Gr_\bmu)$. By Theorem \ref{sat} (2), $V_\mu^{\widehat G}$ appears in $V_\bmu^{\widehat G}$ as desired.
\end{proof}

\begin{rmk}
The condition in Lemma \ref{appear} is equivalent to that $G$ is simple, adjoint, and any/some of its absolute factors is of classical type or $E_6$ type or $E_7$ type.
\end{rmk}

\begin{prop} \label{basic}
We have $|\JJ_b \backslash \Irr^\tp X_\mu(b)| = \dim V_\mu(\ul_G(b))$ if $b$ is basic.
\end{prop}

First we reduce Proposition \ref{basic} to the adjoint case.
\begin{lem}\label{adj}
Proposition \ref{basic} is true for $G$ if it is true for $G = G_\ad$.
\end{lem}
\begin{proof}
Choose $\o \in \pi_1(G)$ such that $X_\mu(b)^\o := X_\mu(b) \cap \Gr^\o \neq \emptyset$, where $\Gr^\o$ is the corresponding connected component of $\Gr$. By \cite[Corollary 2.4.2]{CKV} and \cite[Proposition 3.1]{HV}, the natural projection $G \to G_\ad$ induces a universal homeomorphism $X_\mu(b)^\o \overset \sim \to X_{\mu_\ad}(b_\ad)^{\o_\ad}$, where $\mu_\ad$, $b_\ad$ and $\o_\ad$ denote the images of $\mu$, $b$ and $\o$ respectively under the natural projection $G \to G_\ad$. Let $\JJ_b^0$, $\JJ_{b_\ad}^0$ be the kernels of the natural projections $\JJ_b \to \pi_1(G)$, $\JJ_{b_\ad} \to \pi_1(G_\ad)$ respectively. Notice that $X_\mu(b) = (\Omega \cap \JJ_b) X_\mu(b)^\o$ as $b \in \Omega$ is basic. By Corollary \ref{inequality}, \begin{align*} \dim V_\mu(\ul_G(b)) &\ge |\JJ_b \backslash \Irr^\tp X_\mu(b)| = |\JJ_b^0 \backslash \Irr^\tp X_\mu(b)^\o| \\ &\ge |\JJ_{b_\ad}^0 \backslash \Irr^\tp X_{\mu_\ad}(b_\ad)^{\o_\ad}| = |\JJ_{b_\ad} \backslash \Irr^\tp X_{\mu_\ad}(b_\ad)| \\ &= \dim V_{\mu_\ad}(\ul_{G_\ad}(b_\ad)) = \dim V_\mu(\ul_G(b)), \end{align*} where the second last equality follows by assumption.
\end{proof}

\begin{proof}[Proof of Proposition \ref{basic} by assuming Proposition \ref{minu}]
By Lemma \ref{adj}, we can assume $G$ is adjoint and simple. If the coweight lattice equals the coroot lattice, then $b$ is unramified and the statement is proved in \cite[Theorem 4.4.14]{XZ}. So we will assume $G$ has a nonzero minuscule coweight. By Lemma \ref{appear}, there exists a minuscule and dominant cocharacter $\bmu \in Y^d$ for some $d \in \ZZ_{\ge 1}$ such that $\BB_\mu^{\widehat G}$ appears in $\BB_\bmu^{\widehat G}$, that is, $m_\bmu^\mu \neq 0$. By Proposition \ref{minu}, \begin{align*} \dim V_\bmu^{\widehat G}(\ul_G(b)) &= |\JJ_\bb \backslash \Irr^\tp X_\bmu(\bb)| \\ &= \sum_{\upsilon \leq |\bmu|}  m_\bmu^\upsilon |\JJ_b \backslash \Irr X_\upsilon(b)| \\ & \le \sum_{\upsilon \leq |\bmu|} m_\bmu^\upsilon \dim V_\upsilon^{\widehat G}(\ul_G(b)) \\ &= \dim V_\bmu^{\widehat G}(\ul_G(b)). \end{align*} where the second equality follows from Corollary \ref{compare}, and the inequality follows from Corollary \ref{inequality}. Thus $|\JJ_b \backslash \Irr X_\upsilon(b)| = \dim V_\upsilon(\ul(b))$ if $m_\bmu^\upsilon \neq 0$.
\end{proof}

\begin{thm}\label{conj'}
We have $|\JJ_b \backslash \Irr^\tp X_\mu(b)| = \dim V_\mu(\ul_G(b))$. In particular, $$(P(\brF) \cap \JJ_b) \backslash \Irr^\tp X_\mu(b) \cong \JJ_b \backslash \Irr^\tp X_\mu(b).$$
\end{thm}
\begin{proof}
By Corollary \ref{inequality} we have \begin{align*} |(P(\brF) \cap \JJ_b) \backslash \Irr^\tp X_\mu(b)| &= \sum_{\eta \in I_{\mu, b, M}} |\Sigma_{\mu, \eta}^N|  |\JJ_b \backslash \Irr^\tp X_\eta^M(\l)| \\ &= \sum_{\eta \in I_{\mu, b, M}} |\Sigma_{\mu, \eta}^N| \dim V_\eta^{\widehat M}(\ul_M(b)) \\ &= \dim V_\mu(\ul_G(b)), \end{align*} where the second equality follows from Proposition \ref{basic} since $b$ is basic in $M(\brF)$. Now the first statement follows by taking $M$ to be the centralizer of $\nu_G(b)$, in which case $P(\brF) \cap \JJ_b = \JJ_b^M = \JJ_b$. The second statement follows from the equality $|(P(\brF) \cap \JJ_b) \backslash \Irr^\tp X_\mu(b)| = \dim V_\mu(\ul_G(b)) = |\JJ_b \backslash \Irr^\tp X_\mu(b)|$.
\end{proof}

\subsection{Decomposition of MV-cycles}
Notice that each ${\widehat G}$-crystal restricts to an $\widehat M$-crystal. For $\d \in \BB_\mu^{\widehat G}$ we denote by $S_\d^M$ the corresponding Mirkovi\'{c}-Vilonen cycle in $\Gr_M$.
\begin{lem} \label{S-NM}
Let $\d \in \BB_\mu^{\widehat G}$ and let $\eta \in I_{\mu, M}$ such that $\d$ lies in a highest weight $\widehat M$-crystal isomorphic to $\BB_\eta^{\widehat M}$. Then there exists a unique irreducible component $Z_\d^N \in \Sigma_{\mu, \eta}^N$ such that $$\overline{ S_\d } = \overline{Z_\d^N \ast S_\d^M }.$$ Here we view $S_\d^M$ as its open dense subset lying in $\Gr_{\eta, M}^\circ$.
\end{lem}
\begin{proof}
Let $\chi \in I_{\mu, M}$. Let $Z^N \in \Sigma_{\mu, \chi}^N$ and $\xi \in \BB_\chi^{\widehat M}(\l)$ for some $\l \in Y$. Then $Z^N \ast S_\xi^M \subseteq S^\l$ is irreducible as the natural projection $Z^N \ast S_\xi^M \to S_\xi^M$ is a fiber bundle with fibers isomorphic to $Z^N$. Moreover, \begin{align*} \dim Z^N \ast S_\xi^M &= \dim Z^N + \dim S_\xi^M \\ &= \<\rho, \mu + \chi\> - 2\<\rho_M, \chi\> + \<\rho_M, \chi + \l\>  \\ &= \<\rho, \mu + \l\> + \<\rho_N, \chi - \l\> \\ &= \<\rho, \mu + \l\> \\ &= \dim (S^\l \cap \Gr_\mu), \end{align*} where the last equality follows from that $\chi - \l \in \ZZ \Phi_M^\vee$. Therefore, $$\overline{ Z \ast S_\xi^M } \in \Irr (\overline{ S^\l \cap \Gr_\mu}) \cong \MV_\mu(\l) = \BB_\mu^{\widehat G}(\l).$$ Hence the map $(Z^N, \xi) \mapsto \overline{ Z^N \ast S_\xi^M }$ gives an embedding $$\sqcup_\chi (\Sigma_{\mu, \chi}^N \times \BB_\chi^{\widehat M}) \hookrightarrow \MV_\mu \cong \BB_\mu^{\widehat G},$$ which is bijective since $\sum_\chi |\Sigma_{\mu, \chi}^N| |\BB_\chi^{\widehat M}| = |\BB_\mu^{\widehat G}|$ by Proposition \ref{res}. Thus there exist unique $\k \in I_{\mu, M}$, $\z \in \BB_\k^{\widehat M}$ and $Z_\d^N \in \Sigma_{\mu, \k}^N$ such that $\overline{ S_\d } = \overline{ Z_\d^N \ast S_\z^M }$. It remains to show $\z = \d \in \BB_\eta^{\widehat M}$, that is, $\overline{ \pi_P(S_\d) } = \overline{ S_\d^M } $ with $\pi_P: \Gr_P \to \Gr_M$ the natural projection. In view of the construction of MV cycles using Littelmann's path model \cite[Proposition 3.3.12 \& 3.3.15]{XZ}, it suffices to consider the case where $\mu$ is a quasi-minuscule cocharacter of $G$. Then the statement follows from the explicit construction in \cite[\S 3.2.5 \& Definition 3.3.6]{XZ}.
\end{proof}

\subsection{Proof of Theorem \ref{main}}
Take $P = M N$ such that $b$ is superbasic in $M(\brF)$. Let $C \in \Irr^\tp X_\mu(b)$. By Corollary \ref{inequality}, there exist $\eta \in I_{\mu, b, M}$ and $\l \in Y$ such that $C \subseteq \overline{ X_\mu^{Z^N, C^{\l, M}}(b) }$ for some $(Z^N, C^{\l, M}) \in \Sigma_{\mu, \eta}^N \times \Irr X_\eta^{\l, M}$ such that $\overline{ C^{\l, M} } \in \Irr X_\eta^M(b)$. In particular, $(N(\brF)I_M t^\l K/K) \cap C$ is open dense in $C$. Let $\g^M(\overline{ C^{\l, M} }) \in \BB_\eta^{\widehat M}(\ul_M(b))$ be as in Theorem \ref{reform} such that $$\overline{ t^{-\l} H^M(C^{\l, M}) t^{b\s(\l)} K_M/K_M } = \e_\l^M \overline{ S^M_{\g^M(\overline{C^{\l, M}})}} \subseteq \Gr_M.$$ By Lemma \ref{S-NM}, there exists $\g(C) \in \BB_\mu(\ul(b))$ such that $$\overline{ Z^N \ast S_{\g^M(\overline{C^{\l, M}})}^M } = \overline{ S_{\g(C)} }.$$ By Proposition \ref{relative} (5), $$ \overline{ t^{-\l} H^P(C) t^{b\s(\l)} K/K } = \overline{ Z^N \ast \e_\l^M S_{ \g^M(\overline{C^{\l, M}})} } = \e_\l^M \overline{ Z^N \ast S_{ \g^M(\overline{C^{\l, M}})} } = \e_\l^M \overline{ S_{\g(C)} }.$$ So the first statement follows.

Let $C' \in \Irr^\tp X_\mu(b)$ be a conjugate of $C$ under $\JJ_b$. By Theorem \ref{conj'}, $C'$ and $C$ are conjugate under $P(\brF) \cap \JJ_b$, which, combined with Corollary \ref{inequality}, implies that $C' \subseteq \overline{ X_\mu^{Z^N, C^{\l', M}}(b) }$ for some $\l' \in Y$ and $C^{\l', M} \in \Irr X_\eta^{\l', M}(b)$ such that $\overline{C^{\l, M}}$ and $\overline{C^{\l', M}}$ are conjugate by $\JJ_b^M$. By Theorem \ref{reform}, we have $$\g^M(\overline{C^{\l', M}}) =  \g^M(\overline{C^{\l, M}})$$ and hence $\g(C') = \g(C)$. So $\g$ is invariant on the $\JJ_b$-orbits of $\Irr^\tp X_\mu(b)$.

It remains to show $\g$ induces a bijection $\JJ_b \backslash \Irr^\tp X_\mu(b) \cong \BB_\mu(\ul_G(b))$. By Theorem \ref{conj'} it suffices to show it is surjective. Let $\d \in \BB_\mu(\ul_G(b))$. Suppose $\d \in \BB_\eta^{\widehat M}(\ul_G(b))$ for some $\eta \in I_{\mu, b, M}$. It follows from Theorem \ref{reform} that there exists $C^M \in \Irr^\tp X_\eta^M(b)$ such that $$\g^M(C^M) = \d \in \BB_\eta^{\widehat M}(\ul_G(b)).$$ Let $\phi \in Y$ and $C^{\phi, M} \in \Irr X_\eta^{\phi, M}(b)$ such that $\overline{C^{\phi, M}} = C^M$. Let $Z_\d^N \in \Sigma_{\mu, \eta}^N$ be as in Lemma \ref{S-NM} such that $\overline{ S_\d } = \overline{Z_\d^N \ast S_\d^M}$. By the construction in the previous paragraph, we have $\g(C) = \d$ for any $C \in \Irr \overline{X_\mu^{Z_\d^N, C^{\phi, M}}(b)}$. So $\g$ is surjective as desired.

\subsection{Proof of Theorem \ref{conv}}
Let $C \in \Irr^\tp X_\bmu(\bb)$ and $C' \in \Irr^\tp X_\mu(b)$ for some $\mu \in Y^+$ such that $\overline{C'} = \overline{\pr(C)}$. One should not confuse with the notation in the previous subsection. Assume $\g^{G^d}(C) = \bg = (\g_1, \dots, \g_d) \in \BB_\bmu^{\widehat G^d}$. By Corollary \ref{compare}, it suffices to show that $$\g(C') = \g_1 \otimes \cdots \otimes \g_d.$$ It follows from Corollary \ref{inequality} that there exist $Z_\bullet^{N^d} = (Z_1^N, \cdots, Z_d^N) \in \Sigma_{\bmu, \bbeta}^{N^d}$ and $C^{\bl, M^d} \in \Irr X_\bbeta^{\bl, M^d}(\bb)$ for some $\bbeta \in I_{\bmu, M^d}$ and $\bl \in Y^d$ such that $\overline{C^{\bl, M^d}} \in \Irr^\tp X_\bbeta^{M^d}(\bb)$ and \begin{align*} \tag{i} C \subseteq \overline{ X_\bmu^{Z_\bullet^{N^d}, C^{\bl, M^d}}(\bb) }. \end{align*}

Let $\cz_\bullet^{N^d} = (\th_{\bbeta}^{N^d})\i( Z_\bullet^{N^d} ) = (\cz^N_1, \cdots, \cz_d^N)$. By (i) and Proposition \ref{relative} (5), \begin{align*} \tag{ii} \overline{ t^{-\bl} H^{P^d}(C) t^{\bb \bs(\bl)} } = \overline{ \cz_\bullet^{N^d} \ast (t^{-\bl} H^{M^d}(C^{\bl, M^d}) t^{\bb \bs(\bl)} ) }.\end{align*} Set $\bl= (\l_1, \dots, \l_d)$, $\l^\dag_\bullet = \bb \bs(\bl) = (\l_1^\dag, \dots, \l_d^\dag)$ and \begin{align*} H^{P^d}(C) &= H_1(C) \times \cdots \times H_d(C); \\ H^{M^d}(C^{\bl, M^d}) &= H_1(C^{\bl, M^d}) \times \cdots \times H_d(C^{\bl, M^d}). \end{align*} Applying Theorem \ref{main} (for $C$ and $\overline{C^{\bl, M^d}}$ respectively) and Lemma \ref{S-NM} we have $$\e_\bl^{M^d} \overline{ S_{\bg}^{G^d} } = \overline{ t^{-\bl} H^{P^d}(C) t^{\l^\dag_\bullet} K^d/K^d } = \overline{ \cz_\bullet^{N^d} \ast (t^{-\bl} H^{M^d}(C^{\bl, M^d}) t^{\l^\dag_\bullet} ) K^d /K^d } = \e_\bl^{M^d} \overline{Z_\bullet^{N^d} \ast S_{\bg}^{M^d} }.$$ In particular, for $1 \le \t \le d$ we have \begin{align*} \tag{iii} \overline{ S_{\g_\t} } = (\e_{\l_\t}^M)\i \overline{\cz^N_\t \ast (t^{-\l_\t} H_\t(C^{\bl, M^d}) t^{\l_\t^\dag}) K /K} = \overline{Z^N_\t \ast S_{\g_\t}^M}. \end{align*} Let $\l = \pr(\bl) = \l_1$. As $\overline{C'} = \overline{\pr(C)} \subseteq \Gr$, we see that $N(\brF) I_M t^\l K/K \cap C'$ is open dense in $C'$. By Theorem \ref{main}, \begin{align*} &\quad\ \e_\l^M \overline{ S_{\g(C')} } \\ &= \overline{ t^{-\l} H^P(C') t^{b\s(\l)} K/K } \\ &= \overline{ t^{-\l_1} H_1(C) t^{\l^\dag_1} \cdots  t^{-\l_d} H_d(C) t^{\l^\dag_d} K/K  } \\ &= \overline{ (\cz_1^N \ast (t^{-\l_1} H_1(C^{\bl, M^d}) t^{\l^\dag_1})) \cdots (\cz_d^N \ast (t^{-\l_d} H_d(C^{\bl, M^d}) t^{\l^\dag_d})) K/K  }  \\  &= \overline{ m (\cz_1^N \ast (t^{-\l_1} H_1(C^{\bl, M^d}) t^{\l^\dag_1}) \times_K \cdots \times_K (\cz_d^N \ast (t^{-\l_d} H_d(C^{\bl, M^d}) t^{\l^\dag_d})) K/K)  } \\ &= \overline{ m( \e_{\l_1}^M (Z_1^N \ast S_{\g_1}^M) \tilde\times \cdots \tilde\times (Z_d^N \ast S_{\g_d}^M) ) } \\ &= \e_{\l_1}^M \overline{ m( S_{\g_1} \tilde\times \cdots \tilde\times S_{\g_d} ) } \\ &= \e_{\l_1}^M \overline{ S_{\g_1} \star \cdots \star S_{\g_d} } \\  &= \e_\l^M \overline{ S_{\g_1 \otimes \cdots \otimes \g_d} }, \end{align*} where $m: G(\brF) \times_K \cdots \times_K G(\brF) \times_K \Gr \to \Gr$ is the usual convolution map; the second equality follows from Lemma \ref{convolution}; the third one follows from (ii) and that $\l^\dag_\bullet = (\l_1^\dag, \dots, \l_d^\dag) = (\l_2, \dots, \l_d, b\s(\l_1))$; the fifth one follows from Lemma \ref{star}, Theorem \ref{reform} and (iii). So $\g(C') = \g_1 \otimes \cdots \otimes \g_d$ as desired.

\section{Proof of Proposition \ref{minu}} \label{sec-minu}
We keep the notations in \S \ref{sec-pre}. Let $\mu \in Y^+$ be minuscule, and let $b \in G(\brF)$ be basic which is a lift of an element in $\Omega$. To prove Proposition \ref{minu}, we assume by Lemma \ref{adj} that $G$ is simple and adjoint. Then $\s$ acts transitively on the connected components of (the Dynkin diagram of) $\SS_0$. Let $d$ be the number of connected components of $\SS_0$.

By abuse of notation, we also denote by $\tw \in \tW \cap \JJ_b = \{x \in \tW; b \s(x) b\i = x\}$ some lift of $\tw$ in $N_T(\brF)$ that lies in $\JJ_b$.
\subsection{Orthogonal subset of roots} \label{def-orthogonal}
We say a subset $D \subseteq \Phi$ is strongly orthogonal if $\b' \pm \b \notin \Phi$ for any $\b', \b \in D$. In particular, if $D$ is strongly orthogonal, then it is orthogonal, that is, $\<\b', \b^\vee\> = 0$ for any $\b \neq \b' \in D$.

Let $\a \in \Phi$. Set $\co_\a = \{\a^i; i \in \ZZ\}$ and $\co_\tta = \{\tta^i; i \in \ZZ\}$, where $\a^i$ and $\tta^i$ are as in \S \ref{def-R}. Let $W_{\co_\tta}$ be the parabolic subgroup of $\tW$ generated by $s_{\tilde \b}$ for $\b \in \co_\a$. Recall that $\Pi$ is the set of minus simple roots and highest roots of $\Phi$.
\begin{lem} \label{orbit}
Let $\a \in \Pi$ such that $W_{\co_\tta}$ is finite. Let $\tw$ be the longest element of $W_{\co_\tta}$. Then $W_{\co_\tta} \cap \JJ_b = \{1, \tw\}$ and one of the following cases occurs:

(1) $\<\a^d, \a^\vee\> = -1$, $|\co_\a| = 2d$, $\co_{\a + \a^d}$ is strongly orthogonal (as $|\co_{\a + \a^d}| = d$) and $\tw = \prod_{\xi \in \co_{\a + \a^d}} s_{\tilde \xi}$;

(2) $\co_\a$ is strongly orthogonal and hence $\tw = \prod_{\b \in \co_\a} s_{\tilde \b}$.

In particular, any affine reflection of $W^a \cap \JJ_b$ is equal to $\prod_{c \in \co_a} s_c$ for some $a = (\g, k) \in \Phi \times \ZZ = \tPhi$ such that $\co_\g$ is strongly orthogonal.
\end{lem}
\begin{proof}
The first statement follows from a case-by-case analysis. The ``In particular'' part follows by noticing that each reflection of $W^a \cap \JJ_b$ is conjugate to some $\tw$ as in the first statement.
\end{proof}

\subsection{Characterization of $\ca_{\mu, b}^\tp$} \label{generic} Let $\l \in Y$. Let $X_\mu^\l(b) = I t^\l K/K \cap X_\mu(b)$, $\ca_{\mu, b} = \ca_{\mu, b}^G$, $\ca_{\mu, b}^\tp = \ca_{\mu, b}^{\tp, G}$ and $R(\l) = R_{\mu, b}^\tp(\l)$ be as in \S\ref{def-R}. By Proposition \ref{dim}, $\l \in \ca_{\mu, b}$ if and only if $\l^\natural = -\l + b\s(\l)$ is conjugate to $\mu$ by $W_0$.

Let $V = Y \otimes_\ZZ \RR$ and $V^{p(b\s)} = \{v \in V; p(b\s)(v) = v\}$. Define $$V^{p(b\s)}_\gen = \{v \in V^{p(b\s)}; \<\a, v\> = 0 \Leftrightarrow \<\a, V^{p(b\s)}\> = 0, \forall \a \in \Phi\},$$ which is open dense in $V^{p(b\s)}$. Notice that $V^{p(b\s)}_\gen \cap Y \neq \emptyset$. Let $M_b \supseteq T$ be the Levi subgroup with root system $\{\a \in \Phi; \<\a, V^{p(b\s)}\> = 0\}$. By definition, for any $v \in V^{p(b\s)}_\gen$ the centralizer $M_v$ (see \S\ref{levi-subsec}) of $v$ in $G$ coincides with $M_b$.

Fix $v \in V^{p(b\s)}_\gen \cap Y$. Denote by $\bar v$ the unique dominant $W_0$-conjugate of $v$. Let $z$ be the minimal element of $W_0$ such that $z(v) = \bar v$. Let $N_v = \prod_{\a \in \Phi; \<\a, v\> > 0} U_\a$. Set $M = M_{\bar v} = {}^z M_v = {}^z M_b$ and $b_M = z b \s(z)\i$. By \cite[Lemma 3.1]{HN}, $b_M$ is a lift of some element in $\Omega_M$, and is superbasic in $M(\brF)$.

\begin{lem} \label{bound}
Let $\l \in \ca_{\mu, b}$ and $\a \in \Phi - \Phi_{M_b}$. Then $\co_\a \cap \co_{-\a} = \emptyset$ and $$|R(\l) \cap (\co_\a \cup \co_{-\a})| \le \frac{1}{2} \sum_{\b \in \co_\a} |\<\b, \l^\natural\>|,$$ where the equality holds if and only if either $\l_\b \ge 0$ for $\b \in \co_\a$ or $\l_\b \le -1$ for $\b \in \co_\a$.
\end{lem}
\begin{proof}
As $\mu$ is minuscule, it follows from Proposition \ref{dim} and Lemma \ref{natural} (1) that $\l^\natural$ is minuscule and that \begin{align*} \tag{i} \l_{\g\i} - \l_\g = \<\g, \l^\natural\> \in \{0, \pm1\} \text{ for } \g \in \Phi. \end{align*} By assumption, we have $\<\a, v\> \neq 0$ and hence $\co_\a \cap \co_{-\a} = \emptyset$. By symmetry, we may assume $\l_\a \ge 0$ and there exist integers $$0=b_0 \le c_1 < b_1  \le  \cdots \le c_r < b_r = |\co_\a|$$ such that for $1 \le k \le r$ we have \begin{align*} \l_{\a^i} < 0 \text{ for } b_{k-1} + 1 \le i \le c_k \text{ and } \l_{\a^j} \ge 0 \text{ for } c_k + 1 \le j \le b_k. \end{align*} It follows from (i) that \begin{align*} \tag{ii} \text{ if } b_{k-1} < c_k, \text{ then } \l_{\a^{b_{k-1} + 1}} = \l_{\a^{c_k}} = -1 \text{ and } \l_{\a^{c_k + 1}} = \l_{\a^{b_{k-1}}} = 0. \end{align*}

If $b_{k-1} < c_k$ for some $1 \le k \le r$, we have \begin{align*} &\quad\ | R(\l) \cap \{\pm\a^i; b_{k-1}+1 \le i \le c_k\} | \\ &= | R(\l) \cap \{-\a^i; b_{k-1}+1 \le i \le c_k\}|  \\ &= | \{b_{k-1}+1 \le i \le c_k; \l_{-\a^{i-1}} \ge 0, \l_{-\a^i} - \l_{-\a^{i-1}} = 1 \} | \\ &= | \{b_{k-1}+1 \le i \le c_k; \l_{\a^{i-1}} \le -1, \l_{\a^{i-1}} - \l_{\a^i} = 1 \} | \\ &= | \{b_{k-1}+1 \le i \le c_k; \l_{\a^{i-1}} - \l_{\a^i} = 1 \} | - 1 \\ &= -\frac{1}{2} + \frac{1}{2} \sum_{i = b_{k-1} + 1}^{c_k} |\l_{\a^{i-1}} - \l_{\a^i}| \\ &=  -\frac{1}{2} + \frac{1}{2} \sum_{i = b_{k-1} + 1}^{c_k} |\<\a^i, \l^\natural\>|, \end{align*} where the first equality follows from that $\l_{\a^i} < 0$ and hence $\a^i \notin R(\l)$ for $b_{k-1} + 1 \le i \le c_k$; the third one follows from that $\l_{-\g} = -1 - \l_\g$ for $\g \in \Phi$; the fourth one follows from that $\l_{\a^i} \le -1$ for $b_{k-1} + 1 \le i \le c_k$ but $1 + \l_{\a^{b_{k-1}+1}} = \l_{\a^{b_{k-1}}} = 0$ by (ii); the fifth one follows from (i) and the equality $\sum_{i = b_{k-1} + 1}^{c_k} \l_{\a^{i-1}} - \l_{\a^i} = \l_{\a^{b_{k-1}}} - \l_{\a^{c_k}} = 1$ by (ii); the last one follows from (i).

Similarly, for $1 \le k \le r$, \begin{align*} &\quad\ | R(\l) \cap \{\pm\a^i; c_k + 1 \le i \le b_k\} | \\ &= |\{c_k + 1 \le i \le b_k; \l_{\a^i} \ge 1, \l_{\a^i} - \l_{\a^{i-1}} = 1\}| \\ &=\begin{cases} |\{c_k + 1 \le i \le b_k; \l_{\a^i} - \l_{\a^{i-1}} = 1\}|, & \text{ if } b_{k-1} = c_k; \\ |\{c_k + 1 \le i \le b_k; \l_{\a^i} - \l_{\a^{i-1}} = 1\}| - 1, & \text{ otherwise} \end{cases} \\ &= \begin{cases} \frac{1}{2} \sum_{i = c_k + 1}^{b_k} |\<\a^i, \l^\natural\>|, & \text{ if } b_{k-1} = c_k \\ -\frac{1}{2} + \frac{1}{2} \sum_{i = c_k + 1}^{b_k} |\<\a^i, \l^\natural\>|, & \text{ otherwise.} \end{cases} \end{align*} where the second equality follows from that $\l_{\a^i} \ge 0$ for $c_k + 1 \le i \le b_k$ and that $\l_{\a^{c_k}} \ge 0$ if and only if $b_{k-1} = c_k$.

Therefore, $$| R(\l) \cap (\co_\a \cup \co_{-\a}) | \le \frac{1}{2} \sum_{\b \in \co_\a} |\<\b, \l^\natural\>|,$$ where the equality holds if and only if $b_{k-1} = c_k$ for $1 \le k \le r$, that is, $\l_\b \ge 0$ for $\b \in \co_\a$. The proof is finished.
\end{proof}

\begin{lem} \label{finite}
Let $\a \in \Pi$ (see \S\ref{aff-root}). Then $W_{\co_\tta}$ is infinite if and only if $\co_\a = \Pi$. Moreover, in this case, $M_b = G$.
\end{lem}
\begin{proof}
It follows from a case-by-case analysis on the Dynkin diagram of $\SS_0$.
\end{proof}

\begin{lem} \label{red}
For $\l \in \ca_{\mu, b}$ the map $\a \mapsto z(\a)$ gives a bijection $R(\l) \cap \Phi_{M_b} \cong R_{z(\l^\natural), b_M}^M(z(\l))$. As a consequence, $|R(\l) \cap \Phi_{M_b}| \le \dim X_{z(\l^\natural)}^M(b_M)$. Here the subset $R_{z(\l^\natural), b_M}^M(z(\l)) \subseteq \Phi_M$ is defined in \S\ref{def-R} for $G = M$.
\end{lem}
\begin{proof}
Since $z(\Phi_{M_b}^+) = \Phi_M^+$, we have $\l_\a = z(\l)_{z(\a)}$ for $\a \in \Phi_{M_v}$. Hence the first statement follows. The second statement follows from Proposition \ref{dim} that $|R_{z(\l^\natural), b_M}^M(z(\l))| = \dim X_{z(\l^\natural)}^{z(\l),M}(b_M)$.
\end{proof}

\begin{cor} \label{criterion}
Let $\l \in \ca_{\mu, b}$. Then $\l \in \ca_{\mu, b}^\tp$ if and only if (1) $z(\l) \in \ca_{z(\l^\natural), b_M}^{M, \tp}$ and (2) for each $\a \in \Phi - \Phi_{M_b}$, either $\l_\b \ge 0$ for $\b \in \co_\a$ or $\l_\b \le -1$ for $\b \in \co_\a$.
\end{cor}
\begin{proof}
As $\l^\natural$ is conjugate to $\mu$, we have $$|\<\rho, \mu\>| =  \frac{1}{2} \sum_{\a \in \Phi_{M_v}^+ \cup \Phi_{N_v}} |\<\a, \l^\natural\>|.$$ Therefore, \begin{align*} \dim X_\mu^\l(b) &= |R(\l)|  \\ &= |R(\l) \cap \Phi_{M_b}| + |R(\l) \cap (\Phi - \Phi_{M_b})| \\ &= |R(\l) \cap \Phi_{M_b}| + \sum_{\co} |R(\l) \cap \pm \co| \\ &\le  |R(\l) \cap \Phi_{M_b}| + \frac{1}{2} \sum_{\a \in \Phi_{N_v}} |\<\a, \l^\natural\>| \\ &= |R(\l) \cap \Phi_{M_b}| + \<\rho, \overline{\l^\natural}\> - \<\rho_M, \overline{z(\l^\natural)}^M\>  \\ &\le \dim X_{z(\l^\natural)}^M(b_M) + \<\rho, \mu\> - \<\rho_M, \overline{z(\l^\natural)}^M\> \\ &= \dim X_\mu(b), \end{align*} where $\co$ ranges over $p(b\s)$-orbits of $\Phi_{N_v}$, and moreover, by Lemma \ref{bound} and Lemma \ref{red} the equality holds if and only if the conditions (1) and (2) hold. The proof is finished.
\end{proof}

\subsection{The action of $W_{\co_\tta} \cap \JJ_b$ on $\Irr X_\mu(b)$} \label{finite-orbit} Notice that $\JJ_b$ is generated by $I \cap \JJ_b$, $\Omega \cap \JJ_b$ and $W_{\co_\tta} \cap \JJ_b$ for $\a \in \Pi$. In this subsection we study the action of $W_{\co_\tta} \cap \JJ_b$ on $\Irr X_\mu(b)$. Assume that $W_{\co_\tta}$ is finite and let $\tw$ be the longest element of $W_{\co_\tta}$.
\begin{lem} \label{symm}
Let $\a, \tw$ be as in \S \ref{finite-orbit}. Then $\{\l_\b; \b \in \co_\a\} = \{-\tw(\l)_\b; \b \in \co_\a\}$ for $\l \in Y$.
\end{lem}
\begin{proof}
Recall that $\l_\b = - \tilde\b(\l)$ for $\b \in \Phi$. The statement follows by noticing that $\tw$ sends $\co_{\tta}$ to $-\co_{\tta}$.
\end{proof}

\begin{lem} \label{ge0}
Let $\a, \tw$ be as in \S \ref{finite-orbit}. Let $\l \in Y$ such that either $\l_\b \ge 1$ for $\b \in \co_\a$ or $\l_\b \le -1$ for $\b \in \co_\a$. For $\g \in \Phi$ with $\l_\g \ge 0$ we have $\tw(\l)_{p(\tw)(\g)} \ge 0$. If, moreover, $\l \in \ca_{\mu, b}$, then $p(\tw) R(\l) = R(\tw(\l))$.
\end{lem}
\begin{proof}
We argue by contradiction. Set $\l' = \tw(\l)$ and $\g' = p(\tw)(\g)$. Suppose $\l_\g \ge 0$ but $\l'_{\g'} < 0$, that is,

(i) $\<\g, \l\> \ge 0$, and $\g < 0$ if $\<\g, \l\> = 0$;

(ii) $\<\g', \l'\> \le 0$, and $\g' > 0$ if $\<\g', \l'\>  = 0$.

By assumption and Lemma \ref{orbit}, we have

(iii) $\l_\b \ge 1$ or $\l_\b \le -1$ if $\b \in \Phi$ is a sum of roots in $\co_\a$.

Case(1): $\<\a^d, \a^\vee\> \neq -1$. Then $\co_\a$ is an orthogonal set and $\tw = \prod_{\b \in \co_\a} s_{\tilde \b}$. So $\l' = \tw(\l) =  p(\tw) (\l - \sum_{\b \in \co_\a \cap \Phi^+} \b^\vee)$ and hence $\<\g', \l'\> = \<\g, \l - \sum_{\b \in E \cap \Phi^+} \b^\vee\>$, where $E =\{\b \in \co_\a; \<\g, \b^\vee\> \neq 0\}$. If $E \subseteq \Phi^-$, then $\<\g', \l'\> = \<\g, \l\>$. By (i) and (ii) this implies that $\g' > 0$, $\g < 0$ and $\<\g, \l\> = 0$. As $E$ consists of minus simple roots and $\g' = p(\tw)(\g) = (\prod_{\b \in E} s_\b) (\g)$, we deduce that $\g$ is a sum of roots in $E$, contradicting (iii) since $\l_\g = 0$. Thus $E$ contains a unique highest root $\th$ of $\Phi^+$ and $\<\g', \l'\> = \<\g, \l - \th^\vee\>$. By (i), (ii) and that $\<\g, \th^\vee\> \neq 0$, we have $\<\g, \th^\vee\> \ge 1$. If $\g = \th \in \co_\a$, then $\g' = -\th < 0$ (since $\co_\a$ is orthogonal). As $\l_\g \ge 0$ and $\g \in \co_\a$, by (iii) we have $\<\l, \g\> = \l_\g + 1 \ge 2$. So $\l'_{\g'} = \<\g, \l\> - 2 \ge 2 - 2 = 0$, which is a contradiction. So $\g \neq \pm \th$ and hence $\<\g, \th^\vee\> = 1$ (since $\th$ is a long root). By (i) and (ii) we have $0 \le \<\g, \l\> \le 1$. If $\<\g, \l\> = 1$, then $\g' = (\prod_{\b \in E - \{\th\}} s_\b) (\g - \th) > 0$ by (ii). As $\g - \th \in \Phi^-$, $\g - \th$ is a sum of roots in $E - \{\th\}$, contradicting that $\co$ is strongly orthogonal by Lemma \ref{orbit} (2). So $\<\g, \l\> = 0$ and hence $\g < 0$ by (i). In particular, $\<\g, \th^\vee\> \le 0$ as $\th^\vee$ is dominant, which contradicts that $\<\g, \th^\vee\> = 1$.

Case(2): $\<\a^d, \a^\vee\> = -1$. Let $\xi=\a + \a^d$. Then $|\co_\xi| = d$ and $\tw = \prod_{\b \in \co_\xi} s_{\tilde \b}$ by Lemma \ref{orbit}. So $\l' = \tw(\l) =  p(\tw) (\l - \sum_{\b \in \co_\xi \cap \Phi^+} \b^\vee)$ and hence $\<\g', \l'\> = \<\g, \l - \sum_{\b \in E \cap \Phi^+} \b^\vee\>$, where $E =\{\b \in \co_\xi; \<\g, \b^\vee\> \neq 0\}$. Notice that $E$ consists of at most one element. If $E = \emptyset$, then $\g = \g'$ and $\<\g', \l'\> = \<\g, \l\>$, contradicting (i) and (ii). So $E = \{\xi^{i_0}\}$ for some $1 \le i_0 \le d$. If $\xi^{i_0} < 0$, then $\a^{i_0}, \a^{i_0 + d}$ are both minus simple roots and $\<\g', \l'\> = \<\g, \l\>$. By (i) and (ii) we have $\g' > 0$, $\g < 0$ and $\<\g, \l\> = 0$. As $\g' = s_{\xi^{i_0}} (\g) = s_{\a^{i_0}} s_{\a^{i_0 + d}} s_{\a^{i_0}} (\g)$, we deduce that $\g$ is a sum of roots in $\{\a^{i_0}, \a^{i_0 + d}\}$, contradicting (iii). So $\xi^{i_0} > 0$ and $\<\g', \l'\> = \<\g, \l - (\xi^{i_0})^\vee\>$. Moreover, as $\xi^{i_0} = \a^{i_0} + \a^{i_0 + d}$, exactly one of $\{\a^{i_0}, \a^{i_0 + d}\}$ is a positive highest root. By symmetry, we can assume $\a^{i_0} < 0$ and $\a^{i_0 + d} > 0$. By (ii) and that $\<\g, (\xi^{i_0})^\vee\> \neq 0$ we have $\<\g, (\xi^{i_0})^\vee\> \ge 1$. If $\g = \xi^{i_0}$, then $\g' = - \xi^{i_0} < 0$. By (iii) we have $\<\l, \g\> = \l_\g + 1 \ge 2$ and hence $\l'_{\g'} = \<\g, \l\> - 2 \ge 2 - 2 = 0$, which is a contradiction. So $\g \neq \pm \xi^{i_0}$ and $\<\g, (\xi^{i_0})^\vee\> = 1$ (since $\xi^{i_0}$ is a long root), which means $\g' = s_{\xi^{i_0}}(\g) = \g - \xi^{i_0}$. By (i) and (ii)  we have $0 \le \<\g, \l\> \le 1$. If $\<\g, \l\> = 1$, then $\<\g', \l'\> = 0$ and hence $0 < \g' = \g - \xi^{i_0} = (\g - \a^{i_0 + d}) - \a^{i_0} \le - \a^{i_0}$ by (ii), where the last inequality follows from that $\a^{i_0 + d}$ is a positive highest root. As $-\a_{i_0}$ is a simple root, we deduce that $\g' = -\a^{i_0}$ and hence $\g = \a^{i_0 + d} \in \co_\a$, contradicting (iii) since $\l_\g = 0$. So $\<\g, \l\> = 0$ and hence $\g < 0$ by (i), which together with the equality $\g' = \g - \xi^{i_0} \in \Phi$ implies that $0 \le \g' + \a^{i_0 + d} = \g - \a^{i_0} < -\a^{i_0}$. So $\g - \a^{i_0} = 0$, that is, $\g = \a^{i_0} \in \co_\a$, which contradicts (iii) since  $\l_\g = 0$. The first statement is proved.

Let $\g \in R(\l)$, that is, $\<\g, \l^\natural\> = -1$ and $\l_{\g\i} \ge 0$. By Lemma \ref{natural} and the first statement of the lemma we have $$\<p(\tw)(\g), \tw(\l)^\natural\> = \<p(\tw)(\g), p(\tw)(\l^\natural)\> = \<\g, \l^\natural\> = -1$$ and $\tw(\l)_{p(\tw)(\g)\i} = \tw(\l)_{p(\tw)(\g\i)} \ge 0$, that is, $p(\tw)(\g) \in R(\tw(\l))$ and hence $p(\tw)R(\l) \subseteq R(\tw(\l))$. By symmetry (see Lemma \ref{symm}), we have $p(\tw)R(\l) = R(\tw(\l))$. The second statement follows.
\end{proof}

\begin{lem} \label{dom}
Let $\a, \tw$ be as in \S \ref{finite-orbit}. For $\l \in \ca_{\mu, b}^\tp$ we have

(1) either $\l_\b \ge 0$ for $\b \in \co_\a$ or $\l_\b \le -1$ for $\b \in \co_\a$;

(2) if $\l \neq \tw(\l) \in \ca_{\mu, b}^\tp$, then either $\l_\b \ge 1$ for $\b \in \co_\a$, or $\l_\b \le -1$ for $\b \in \co_\a$;

(3) if $\l' \in W_{\co_\tta}(\l) \cap \ca_{\mu, b}^\tp$, then $\l' = \l$ or $\l' = \tw(\l)$.
\end{lem}
\begin{proof}
The first statement follows from Lemma \ref{finite} and Corollary \ref{criterion} (2).

Suppose $\l \neq \tw(\l) \in \ca_{\mu, b}^\tp$. By Lemma \ref{symm} we have $$\{\l_\b; \b \in \co_\a\} = \{-\tw(\l)_\b; \b \in \co_\a\},$$ which, together with (1), implies that either $\l_\b = 0$ for $\b \in \co_\a$ or $\l_\b \ge 1$ for $\b \in \co_\a$ or $\l_\b \le -1$ for $\b \in \co_\a$. By Lemma \ref{eta} (2), the first case implies that $\l = \tw(\l)$, contradicting our assumption. So the statement (2) follows.

Suppose $\l' \in W_{\co_\tta}(\l) \cap \ca_{\mu, b}^\tp$. By (1) and the equality $\chi_\b = -\tilde\b(\chi)$ for $\b \in \Phi$ and $\chi \in Y$, we see that $\l$ and $\l'$ are contained in the union of $$\{y \in Y_\RR; \tilde \b (y) \le 0, \b \in \co_\a\} \text{ and } \{y \in Y_\RR; \tilde \b (y) > 0, \b \in \co_\a\},$$ which are the closed anti-dominant Weyl chamber and the open dominant Weyl chamber for $W_{\co_\tta}$ respectively. Therefore, as $\l' \in W_{\co_\tta}(\l)$, we see that $\l = \l'$ if they are both dominant or both anti-dominant for $W_{\co_\tta}$, and $\l' = \tw(\l)$ otherwise. The statement (3) is proved.
\end{proof}

\begin{cor} \label{adj-red}
Let $\a, \tw$ be as in \S \ref{finite-orbit}. For $\l \in \ca_{\mu, b}^\tp$ we have

(1) $\tw \overline{X_\mu^\l(b)} = \overline{X_\mu^\l(b)}$ if $\l_\b= 0$ for some $\b \in \co_\a$;

(2) $\tw X_\mu^\l(b) \subseteq X_\mu^{\tw(\l)}(b)$ if $\l_\b \le -1$ for $\b \in \co_\a$.
\end{cor}
\begin{proof}
Let $\l' \in \ca_{\mu, b}^\tp$ such that $\Irr \overline{ X_\mu^{\l'} }$ intersects $\Irr (\tw \overline{ X_\mu^\l(b) })$. By Lemma \ref{prod}, $\l' \in W_{\co_\tta}(\l)$. Thus $\l' = \l$ or $\l' = \tw(\l)$ by Lemma \ref{dom} (3). If $\l_\b = 0$ for some $\b \in \co_\a$, then $\l' = \l$ by Lemma \ref{dom} (2). So the statement (1) follows.

Suppose $\l_\b \le -1$ for $\b \in \co_\a$. Then $s_{\tilde \b} t^\l > t^\l$ for $\b \in \co_\a$. Thus $\ell(\tw t^\l) = \ell(\tw) + \ell(t^\l)$ and $$\tw I t^\l K \subseteq I \tw t^\l K = I t^{\tw(\l)} K.$$ So $\tw X_\mu^\l(b) \subseteq X_\mu^{\tw(\l)}(b)$ and the statement (2) follows.
\end{proof}

\subsection{Equivalence relation on $\ca_{\mu, b}^\tp$ and $\ca_{\mu, b}^\tp(v)$} \label{equiv-relation} For $\l, \l' \in \ca_{\mu, b}^\tp$, we write $\l \sim \l'$ if $\JJ_b \Irr \overline{X_\mu^\l(b)} = \JJ_b\Irr \overline{X_\mu^{\l'}(b)}$. Notice that $\JJ_b \Irr X_\mu^\l(b)$ is a single $\JJ_b$-orbit of $\Irr X_\mu(b)$ by Proposition \ref{dim}.

Let $v \in V^{p(b\s)}_\gen \cap Y$. Let $\ca_{\mu, b}^\tp(v)$ (resp. $\ca_{\mu, b}(v)$) denote the set of $\l \in \ca_{\mu, b}^\tp$ (resp. $\l \in \ca_{\mu, b}$) such that $\l_\a \ge 0$ for $\a \in \Phi_{N_v}$. Here $\Phi_{N_v} = \{\a \in \Phi; \<\a, v\> > 0\}$ is the set of roots in $N_v$.
\begin{lem} \label{large}
Let $\l \in \ca_{\mu, b}^\tp$. Then $\l \sim \chi$ for some $\chi \in \ca_{\mu, b}^\tp(v)$.
\end{lem}
\begin{proof}
Let $n \in \ZZ$. By Lemma \ref{prod}, $t^{n \d} I t^\l K \subseteq \cup_{x \leq t^\l} I t^{n v} x K$. Let $\chi_{x, n} \in Y$ such that $I t^{n v} x K = I t^{\chi_{x, n}} K$ for $x \leq t^\l$. Then $t^{n v} \Irr X_\mu^\l(b) \subseteq \cup_{x \leq t^\l} \Irr \overline{ X_\mu^{\chi_{x, n}}(b) }$. Thus for sufficiently large $n$ we have $(\chi_{x, n})_\a \ge 0$ for $x \leq t^\l$ and $\a \in \Phi_{N_v}$. So the statement follows.
\end{proof}

\begin{lem} \label{inv}
Let $\l, \l' \in \ca_{\mu, b}^\tp$ such that $\l \sim \l'$. Then there exists a sequence $\l = \l_0, \l_1, \dots, \l_r = \l' \in \ca_{\mu, b}^\tp$ such that $\l_i = \tw_i (\l_{i-1})$ and $R(\l_i) = p(\tw_i) R(\l_{i-1})$ for $1 \le i \le r$, where $\tw_i \in \Omega \cap \JJ_b$ or $\tw_i$ is the longest element of $W_{\co_{\tta_i}}$ for some $\a_i \in \Pi$.
\end{lem}
\begin{proof}
By assumption, there exist $C \in \Irr X_\mu^\l(b)$ and $g \in \JJ_b$ such that $g C \subseteq \overline{X_\mu^{\l'}(b)}$. Since $b \in \Omega$, $\JJ_b$ is generated by $I \cap \JJ_b$, $\Omega \cap \JJ_b$ and $W_{\co_\tta} \cap \JJ_b$ for $\a \in \Pi$ such that $W_{\co_\a}$ is finite. We may assume $g$ lies in one of the sets $I \cap \JJ_b$, $\Omega \cap \JJ_b$ and $W_{\co_\tta} \cap \JJ_b$. If $g \in I \cap \JJ_b$, then $\l =\l'$ and there is nothing to prove. If $g =\o$ for some $\o \in \Omega \cap \JJ_b$, then $\o I t^\l K/K = I t^{\o(\l)} K/K$. Hence $\l'=\o(\l)$ and $R(\l') = p(\o) R(\l)$ by Lemma \ref{omega}. Suppose $g \in W_{\co_\tta} \cap \JJ_b = \{1, \tw\}$, where $\tw$ is the unique longest element of $W_{\co_\tta}$. Then $\l'$ equals $\l$ or $\tw(\l)$ by Lemma \ref{dom} (3). So we can assume that $\l \neq \l' = \tw(\l)$ and it remains to show $R(\l') = p(\tw) R(\l)$. The statement follows from Lemma \ref{dom} (2) and Lemma \ref{ge0}.
\end{proof}

\begin{prop} \label{polar}
We have $\ca_{\mu, b}^\tp = \cup_{v' \in p(\tW \cap \JJ_b)(v)} \ca_{\mu, b}^\tp(v')$.
\end{prop}
\begin{proof}
Let $\l \in \ca_{\mu, b}^\tp$. By Lemma \ref{large}, there exist $v' \in p(\tW \cap \JJ_b)(v)$ and $\chi \in \ca_{\mu, b}^\tp(v')$ such that $\l \sim \chi$. By Lemma \ref{inv}, we can assume that $\chi \neq \l = \tw(\chi) \in \ca_{\mu, b}^\tp$, where (1) $\tw \in \Omega \cap \JJ_b$ or (2) $\tw$ is as in \S\ref{finite-orbit}. It suffices to show $\l \in \ca_{\mu, b}^\tp(p(\tw)(v'))$, that is, $\l_{p(\tw)(\b)} = \tw(\chi)_{p(\tw)(\b)} \ge 0$ for $\b \in \Phi_{N_{v'}}$. Notice that $\chi_\b \ge 0$ for $\b \in \Phi_{N_{v'}}$. Then the case (1) follows from Lemma \ref{eta} (3), and the case (2) follows from Lemma \ref{dom} (2) and Lemma \ref{ge0} as desired.
\end{proof}

\subsection{The action of $\tW \cap \JJ_b$ on $V^{b\s}$} \label{special} Notice that $W^a \cap \JJ_b$ preserves the affine space $V^{b\s} = \{v \in V; b\s(v) = v\}$. Via the restriction to $V^{b\s}$ we can identify $W^a \cap \JJ_b$ with an affine reflection group of $V^{b\s}$, whose affine root hyperplanes are $H_a \cap V^{b\s}$ for $a \in \tPhi^+$ with $V^{b\s} \neq H_a \cap V^{b\s} \neq \emptyset$. Moreover, $\D \cap V^{b\s}$ is an alcove for $W^a \cap \JJ_b$, with respect to which the simple affine reflections are the longest elements of $W_{\co_\tta}$ for $\a \in \Pi$ with $W_{\co_\tta}$ finite. We fix a special point $e'$ in the closure of  $\D \cap V^{b\s}$ for $W^a \cap \JJ_b$.

We recall a lemma on root systems.
\begin{lem} \label{separate}
Let $E$ be some euclidean space and let $\Sigma \subseteq E$ be a root system. Let $v_1, v_2 \in E$ be two regular points for $\Sigma$ (that is, not contained in any root hyperplane of $\Sigma$). Then there exists root hyperplanes $H_1, \dots, H_r$ separating $v_1$ from $v_2$ such that $s_{H_1} \cdots s_{H_r} (v_1)$ and $v_2$ are in the same Weyl chamber.
\end{lem}
%\begin{proof} Let $W(\Sigma)$ be the reflection group associated to $\Sigma$. We can assume that $v_2$ is dominant. Let $w \in W(\Sigma)$ such that $w(v_1)$ is dominant. Let $w = s_{\a_r} \cdots s_{\a_1}$ be a reduced expression, where each $\a_i$ is a simple root of $\Sigma$. Let $\b_i = s_{\a_1} \cdots s_{\a_{i-1}}(\a_i) > 0$ for $1 \le i \le r$. Then $w(v_1)= s_{\b_1} \cdots s_{\b_r} (v_1)$. It remains to show the root hyperplane $H_i$ of $\b_i$ separates $v_1$ from $v_2$ for $1 \le i \le r$. Indeed, as $v_2$ is regular and dominant, we have $\<\b_i, v_2\> > 0$. On the other hand, $s_{\a_r} \cdots s_{\a_{i+1}} (\a_i) > 0$ and hence $$\<\b_i, v_1\> = \<w(\b_i), w(v_1)\> = -\<s_{\a_r} \cdots s_{\a_{i+1}} (\a_i), w(v_1)\> < 0,$$ where the inequality follows from that $w(v_1)$ is regular and dominant. The proof is finished.
%\end{proof}

\begin{lem} \label{plane}
Let $e'$ be as in \S\ref{special}. Let $v \in V^{p(b\s)}_\gen$. Then for $\tw \in \tW \cap \JJ_b$ there exist affine root hyperplanes $H_1, \dots, H_r$ of $V^{b\s}$ passing through $e'$ such that

(1) $H_i$ separates $e' + v$ from $e' + p(\tw)\i(v)$ for $1 \le i \le r$;

(2) $e' + v = s_{H_1} \cdots s_{H_r} (e' + p(\tw)\i(v))$.

Moreover, $s_{H_i} = \prod_{\b \in \co_{\a_i}} s_{\tilde \b}$ for some $\a_i \in \Phi$ such that $\co_{\a_i}$ is strongly orthogonal.
\end{lem}
\begin{proof}
First note that $V^{p(b\s)}$ is the underlining vector space of the affine space $V^{b\s}$. As $p(\tw)$ preserves $V^{p(b\s)}$, we see that $v, p(\tw)\i(v) \in V^{p(b\s)}_\gen$. Hence $e'+v, e'+p(\tw)(v)$ are regular points for the root system associated to $W^a \cap \JJ_b$ with origin $e'$. By Lemma \ref{separate}, there are affine root hyperplanes $H_1, \dots, H_r$ of $V^{b\s}$ (passing through $e'$) separating $e' + v$ from $e' + p(\tw)\i(v)$ such that $s_{H_1} \cdots s_{H_r} (e' + p(\tw)\i(v))$ and $e' + v$ are contained in the same Weyl chamber of $V^{b\s}$ with origin $e'$. Suppose $e' + v \neq s_{H_1} \cdots s_{H_r} (e' + p(\tw)\i(v))$, that is, $v \neq p(s_{H_1} \cdots s_{H_r} \tw\i) (v) \in W_0(v)$. Then there exists $\a \in \Phi$ whose root hyperplane $H_\a$ separates $v$ from $p(s_{H_1} \cdots s_{H_r} \tw\i) (v)$. In particular, $H_\a \cap V^{b\s}$ is an affine root hyperplane for $W^a \cap \JJ_b$. As $e'$ is a special point for $W^a \cap \JJ_b$, there exists some affine root hyperplane $H$ of $V^{b\s}$ passing through $e'$ which is parallel to $H_\a \cap V^{b\s}$. Then $H$ separates $e' + v$ from $s_{H_1} \cdots s_{H_r} (e' + p(\tw)\i(v))$, contradicting our assumption. So we have $e' + v = s_{H_1} \cdots s_{H_r} (e' + p(\tw)\i(v))$ as desired.

Now we show the ``Moreover'' part. By Lemma \ref{orbit}, there exists $a_i = (\a_i, k_i) \in \tPhi^+$ such that $\co_{\a_i}$ is strongly orthogonal and $s_{H_i} = \prod_{a \in \co_{a_i}} s_a$ by viewing $s_{H_i}$ as an element of $\tW \cap \JJ_b$. Notice that $e' \in H_i = H_{a_i} \cap V^{b\s}$, that is, $a_i(e') = -\<\a_i, e'\> + k_i = 0$. As $e'$ lies in the closure of $\D$, we have $|\<\a_i, e'\>| \le 1$, which together with the inclusion $a_i \in \tPhi^+$ implies that either $\a_i > 0$ and $k_i = 1$ or $\a_i < 0$ and $k_i = 0$. In either case, $a_i = \tta_i$ and the proof is finished.
\end{proof}

\subsection{Characterization of the equivalence relation} Let $v$, $z$, $M$, $b_M$ be as in \S\ref{generic}. We give an explicit description of the equivalence relation $\sim$ on $\ca_{\mu, b}^\tp(v)$.
\begin{lem} \label{Omega}
Let $\l, \l' \in \ca_{\mu, b}^\tp(v)$ such that $\l \sim \l'$. Then $\l'= y(\l)$ for some $y \in \Omega_{M_b} \cap \JJ_b$.
\end{lem}
\begin{proof}
First note that it suffices to find an element $y \in \tW \cap \JJ_b$ such that $p(y)(v) = v$ and $\l'=y(\l)$. Indeed, the conditions $p(y)(v) = v$ and $y \in \JJ_b$ imply that $y \in \tW_{M_b}$ and $y(V^{b\s}) = V^{b\s}$. Noticing that $V^{b\s} \subseteq \D_{M_b}$ (since $V^{b\s} \cap \D \neq \emptyset$), we have $y(\D_{M_b}) = \D_{M_b}$ and hence $y \in \Omega_{M_b}$.

By Lemma \ref{inv}, there is $\tw \in \tW \cap \JJ_b$ such that $\l' = \tw(\l)$ and $R(\l') = p(\tw) R(\l)$. If $p(\tw)(v) = v$, the statement follows as in the above paragraph. Suppose $p(\tw)(v) \neq v$. Let $e'$, $H_i$ and $\a_i$ for $1 \le i \le r$ be as in Lemma \ref{plane}. We construct $x_i \in W^a \cap \JJ_b$ such that $p(x_i)=p(s_{H_i})$ and $x_i(\l)=\l$ as follows.

As $H_i$ separates $e' + v$ from $e' + p(\tw)\i(v)$, without loss of generality we may assume that $$\<\a_i, v\> < 0 < \<\a_i, p(\tw)\i(v)\> = \<p(\tw)(\a_i), v\>.$$ Let $\a \in \co_{-\a_i}$. As $v \in V^{p(b\s)}$ and that $p(\tw)$ commutes with $p(b\s)$, we have $$\<p(\tw)(\a), v\> = \<p(\tw)(-\a_i), v\>  < 0.$$ So $-p(\tw)(\a) \in \Phi_{N_v}$. Moreover, as $\l' \in \ca_{\mu, b}(v)$, we have $\l_{-p(\tw)(\a)}' \ge 0$ and $$\l_{p(\tw)(\a)}' = -\l_{-p(\tw)(\a)}' - 1 \le -1.$$ By definition, $p(\tw)(\a) \notin R(\l')$ and hence $R(\l') \cap p(\tw) \co_{-\a_i} = \emptyset$. Therefore, $R(\l) \cap \co_{-\a_i} = \emptyset$ as $R(\l') = p(\tw) R(\l)$.

We claim that $\l_\b$ is invariant for $\b \in \co_{-\a_i} \subseteq \Phi_{N_v}$. Otherwise, there exists $\xi \in \co_{-\a_i}$ such that $\<\xi, \l^\natural\> = \l_{\xi\i} - \l_\xi \neq 0$ (see Lemma \ref{lam}). Since $\l^\natural$ is minuscule and $$\sum_{\b \in \co_{-\a_i}} \<\b, \l^\natural\> = \sum_{\b \in \co_{-\a_i}} \l_{\b\i} - \l_\b = 0,$$ there exists $\g \in \co_{-\a_i}$ such that $\<\g, \l^\natural\> = \l_{\g\i} - \l_\g = -1$. On the other hand, we have $\l_{\g\i} \ge 0$ since $\l \in \ca_{\mu, b}^\tp(v)$ and $\g\i \in \co_{-\a_i} \subseteq \Phi_{N_v}$. So $\g \in R(\l)$, which contradicts that $R(\l) \cap \co_{-\a_i} = \emptyset$. The claim is proved.

Let $c_i = \l_\b = -\l_{-\b}-1 \in \ZZ$ for $\b \in \co_{-\a_i}$, which is a constant by the above claim. Let $\psi_i = \sum_{\d \in \co_{\a_i}} \l_\d \d^\vee = (-c_i-1) \sum_{\d \in \co_{\a_i}} \d^\vee$ and $x_i=t^{\psi_i} s_{H_i}$. Then $p(x_i) = p(s_{H_i})$ and $t^{\psi_i} \in W^a \cap \JJ_b$. Moreover, by Lemma \ref{eta} (2) and that $\co_{\a_i}$ is orthogonal (see Lemma \ref{plane}) we have $$x_i(\l) = \psi_i + (\prod_{\d \in \co_{\a_i}} s_{\tilde \d}) (\l) = \psi_i + \l - \sum_{\d \in \co_{\a_i}} \l_\d \d^\vee = \l.$$ Thus $x_i$ satisfies our requirements.

Let $y = \tw x_r \cdots x_1 \in \tW \cap \JJ_b$. Then it follows that $y(\l)=\tw(\l)=\l'$ and $p(y)(v) = p(\tw) p(s_{H_r}) \cdots p(s_{H_1})(v)=v$ as desired.
\end{proof}

\begin{lem} \label{equiv}
Let $\l \in \ca_{\mu, b}^\tp(v)$ and $\tw \in \Omega_{M_b} \cap \JJ_b$. Then $\tw(\l) \in \ca_{\mu, b}$. Moreover, if $\tw(\l) \in \ca_{\mu, b}(v)$, then $\tw(\l) \in \ca_{\mu, b}^\tp(v)$ and $\l \sim \tw(\l)$.
\end{lem}
\begin{proof}
Let $\chi = \tw(\l)$. By Lemma \ref{natural}, $\chi^\natural \in W_{M_b}(\l^\natural)$ and hence $\chi \in \ca_{\mu, b}$ by Proposition \ref{dim}. As $\tw \in \Omega_{M_b} \cap \JJ_b$, it follows the same way as Lemma \ref{omega} that $p(\tw) (R(\l) \cap \Phi_{M_b}) = R(\chi) \cap \Phi_{M_b}$. Combining Proposition \ref{dim} with Lemma \ref{red} we have \begin{align*} \tag{a} \dim X_{z(\chi^\natural)}^{z(\chi), M}(b_M) = |R(\chi) \cap \Phi_{M_b}| = |R(\l) \cap \Phi_{M_b}| = \dim X_{z(\l^\natural)}^M(b_M) = \dim X_{z(\chi^\natural)}^M(b_M),\end{align*} where the third equality follows from Corollary \ref{criterion} (1) by noticing that $\l \in \ca_{\mu, b}^\tp$, and the last one follows from that $z(\chi^\natural), z(\l^\natural)$ are conjugate by $W_M$.

Suppose $\chi \in \ca_{\mu, b}(v)$. Then $\chi_\b \ge 0$ and $\chi_{-\b} = -\chi_\b -1 \le -1$ for $\b \in \Phi_{N_v}$, which together with (a) implies that $\chi \in \ca_{\mu, b}^\tp$ by Corollary \ref{criterion}.

For $\eta \in \ca_{\mu, b}(v)$ we have $- \<\b, \eta\> \ge \eta_{-\b} \ge 0$ for $\b \in \Phi_{N_{-v}} = - \Phi_{N_v}$. Then $$I t^\eta K / K = I_{M_b} I_{N_v} I_{N_{-v}} t^\eta K / K = I_{M_b} I_{N_v} t^\eta K / K,$$ which implies that $t^{n v} I t^\eta K / K \subseteq I t^{n v + \eta} K / K$ for $n \in \ZZ_{\ge 0}$. In particular, $t^{nv} X_\mu^{\chi}(b) \subseteq X_\mu^{n v +\chi}(b)$ and hence $\chi \sim n v + \chi$ for $n \in \ZZ_{\ge 0}$. Choose $n$ sufficiently large so that $\tw t^{n v} I_{N_v} t^{-n v} \subseteq I_{N_v} \tw$. Then \begin{align*} \tw t^{nv} X_\mu^\l(b) &\subseteq \tw t^{n v} I t^\l K / K \\ &= \tw t^{n v}  I_{M_b} I_{N_v} t^\l K / K \\ &= I_{M_b} \tw (t^{n v} I_{N_v} t^{-n v}) t^{n v + \l} K /K \\ &\subseteq I_{M_b} I_{N_v} \tw t^{n v + \l} K /K \\ &\subseteq I t^{n v + \chi} K / K.\end{align*} Therefore, $\tw t^{nv} X_\mu^\l(b) \subseteq X_\mu^{n v + \chi}(b)$ and hence $\l \sim n v + \chi \sim \chi$ as desired.
\end{proof}

Combining Lemma \ref{Omega} with Lemma \ref{equiv} we have
\begin{cor} \label{descrp}
Let $\l, \l' \in \ca_{\mu, b}^\tp(v)$. Then $\l \sim \l'$ if and only if $\l' = \tw(\l)$ for some $\tw \in \Omega_{M_b} \cap \JJ_b$.
\end{cor}

\subsection{End of the proof} Let $v$, $z$, $M$, $b_M$ be as in \S\ref{generic}. Recall that $I_{\mu, M}$ is the set of $W_M$-orbits of $W_0(\mu)$, and $I_{\mu, b_M, M}=\{\eta \in I_{\mu, M}; \k_M(t^\eta) = \k_M(b_M)\}$. Let $\tilde \ca_{\mu, b}^\tp(v)$ denote the set of equivalence classes of $\ca_{\mu, b}^\tp(v)$ with respect to $\sim$ defined in \S\ref{equiv-relation}. Similarly, we can define an equivalence relation $\sim_M$ on $\ca_{\eta, b_M}^{M, \tp}$ for $\eta \in I_{\mu, b_M, M}$, and denote by $\tilde \ca_{\eta, b_M}^{M, \tp}$ the set of corresponding equivalence classes. As $b_M$ is superbasic in $M(\brF)$, we have $\chi \sim_M \chi' \in \ca_{\eta, b_M}^{M, \tp}$ if and only if $\chi'=\tw(\chi)$ for some $\tw \in \Omega_M \cap \JJ_{b_M}^M$.
\begin{proof} [Proof of Proposition \ref{minu}]
We show that there are bijections $$\JJ_b \backslash \Irr^\tp X_\mu(b) \overset {\Psi_1} \longleftarrow \tilde \ca_{\mu, b}^\tp(v)  \overset {\Psi_2} \longrightarrow \sqcup_{\eta \in I_{\mu, b_M, M}} \tilde \ca_{\eta, b_M}^{M, \tp},$$ where $\Psi_1$ and $\Psi_2$ are given by $\l \mapsto \JJ_b \overline{ \Irr X_\mu^\l(b) }$ and $\l \mapsto z(\l)$ respectively. Indeed, by Proposition \ref{dim} and Lemma \ref{large} we see that $\Psi_1$ is bijective.

Let $\l \in \ca_{\mu, b}^\tp(v)$. By Corollary \ref{criterion}, $z(\l) \in \ca_{z(\l^\natural), b_M}^{M, \tp}$ and $z(\l^\natural) \in I_{\mu, b_M, M}$. Moreover, by Lemma \ref{natural} and Corollary \ref{descrp} we deduce that $$\l \sim \l' \in \ca_{\mu, b}^\tp(v) \Leftrightarrow z(\l) \sim_M z(\l') \in \ca_{z(\l^\natural), b_M}^{M, \tp}.$$ So $\Psi_2$ is well defined. On the other hand, let $\chi \in \ca_{\eta, b_M}^{M, \tp}$ with $\eta \in I_{\mu, b_M, M}$. By Corollary \ref{criterion} and \ref{descrp}, the map $\chi \to n v + z\i(\chi)$ with $n \gg 0$ induces the inverse map of $\Psi_2$. So $\Psi_2$ is also bijective.

Therefore, \begin{align*} |\JJ_b \backslash \Irr^\tp X_\mu(b)| &= \sum_{\eta \in I_{\mu, b_M, M}} |\tilde \ca_{\eta, b_M}^{M, \tp}|; \\ &= \sum_{\eta \in I_{\mu, b_M, M}} \dim V_\eta^{\widehat M} (\ul_M(b_M)) \\ &= \sum_{\eta \in I_{\mu, M}} \dim V_\eta^{\widehat M} (\ul_M(b_M)) \\ &= \dim V_\mu(\ul_M(b_M)) \\ &= \dim V_\mu(\ul_G(b)),\end{align*} where the second equality follows from \cite[Theorem 1.5]{HV}; the fourth one follows from that $V_\mu =  \oplus_{\eta \in I_{\mu, M}} V_\eta^{\widehat M}$ as $\mu$ is minuscule. The proof is finished.
\end{proof}

\section{The stabilizer in $\JJ_b$} \label{sec-stab}
In this section, we give an algorithm to compute the stabilizer $N_{\JJ_b}(C)$ of $C \in \Irr^\tp X_\mu(b)$ in $\JJ_b$.

\subsection{Reduction to the adjoint case} Let $G_\ad$ be the adjoint quotient of $G$. By \cite[\S 2]{CKV}, the natural projection $f: G \to G_\ad$ induces a Cartesian square
\begin{align*}
\xymatrix{
  X_\mu(b) \ar[d] \ar[r]^f & X_{\mu_\ad}(b_\ad) \ar[d] \\
  \pi_1(G) \ar[r]^f & \pi_1(G_\ad),   }
\end{align*}
where the vertical maps are the natural projections; $\mu_\ad$ and $b_\ad$ are the images of $\mu$ and $b$ under $f$ respectively. In particular, the stabilizer $N_{\JJ_b}(C)$ can be computed from the stabilizer $N_{\JJ_{b_\ad}}(C_\ad)$ of $C_\ad = f(C)$ in $\JJ_{b_\ad}$ via the following natural Cartesian square
\begin{align*}
\xymatrix{
  N_{\JJ_b}(C) \ar[d] \ar[r]^f & N_{\JJ_{b_\ad}}(C_\ad) \ar[d] \\
  \JJ_b^0 \ar[r]^f & \JJ_{b_\ad},   }
\end{align*}
where the vertical maps are the natural inclusions, and $\JJ_b^0$ is the kernel of the natural projection $\JJ_b \to \pi_1(G)$. Thus we can assume $G$ is adjoint and simple.

\subsection{Reduction to the basic case}\label{basic-stab} Now we show how to pass to the case where $b$ is basic. Let $P = M N$ and $\b: X_\mu(b) \to \Gr_M$ be as in \S \ref{Levi} such that $M$ is the centralizer of $\nu_G(b)$. In particular, $\JJ_b = \JJ_b^M$. Let $\eta \in I_{\mu, b, M}$ such that $X_\eta^M(b)$ contains an open dense subset of $\b(C)$. Let $$C^M = \overline{ \b(C) \cap X_\eta^M(b) } \subseteq X_\eta^M(b).$$ By Proposition \ref{relative} (1) and (4), $C = \overline{ X_\mu^{Z^N, C^M}(b)}$ for some $Z^N \in  \Sigma_{\mu, \eta}^N$. Note that $j X_\mu^{Z^N, C^M}(b) = X_\mu^{Z^N, j C^M}(b)$ for $j \in \JJ_b^M = \JJ_b$. So we have $N_{\JJ_b}(C) = N_{\JJ_b^M}(C^M)$.

\subsection{Reduction to the minuscule case} \label{minu-subsec} Assume $b$ is basic. If $G$ has no nonzero minuscule cocharacters, then $b$ is unramified and $N_{\JJ_b}(C)$ is determined in \cite[Theorem 4.4.14]{XZ}. Otherwise, by Lemma \ref{appear}, there exists a dominant minuscule cocharacter $\bmu \in Y^d$ for some $d \in \ZZ_{\ge 1}$ such that $\BB_\mu^{\widehat G}$ occurs in $$\BB_\bmu^{\widehat G} = \BB_{\mu_1}^{\widehat G} \otimes \cdots \otimes \BB_{\mu_d}^{\widehat G}.$$ Let $X_\bmu(\bb)$ be as in \S \ref{conv-setup}. By Theorem \ref{conv}, there exists $C' \in \Irr^\tp X_\bmu(\bb)$ such that $$\overline{\pr(C')} = \overline C \subseteq \Gr,$$ and moreover, the map $g \mapsto (g, \dots, g)$ gives an isomorphism $N_{\JJ_b}(C) \cong N_{\JJ_\bb}(C')$.

\subsection{Small cocharacters}\label{def-small} In the rest of the section we assume that $G = G_\ad$ is simple, $\bmu$ is minuscule and $b$ is basic. By abuse of notation, we write $X_\mu(b)$ for $X_\bmu(\bb)$ by assuming that $\mu$ is minuscule in the rest of this section. Then we can adopt the notation in \S \ref{sec-minu}.

Let $v \in V^{p(b\s)}_\gen \cap Y$. For $D \subseteq \Phi$ we set $D(v, +) = \{\a \in \Phi; \<\a, v\> > 0\}$. We say $\l \in \ca_{\mu, b}^\tp$ is $v$-small if $\l \in \ca_{\mu, b}^\tp(v)$ (see \S\ref{equiv-relation}) and for each $\a \in \Pi(v, +)$ (see \S\ref{aff-root}) there exists $\b \in \co_\a$ such that $\l_\b = 0$. We say $v$ is permissible if $v$-small cocharacters exist.

We say $\l \in \ca_{\mu, b}^\tp$ is small if it is $v$-small for some $v \in V^{p(b\s)}_\gen \cap Y$, and we define $\Pi(\l)$ to be the set of roots $\a \in \Pi - \Phi_{M_b}$ such that $\l_\b \ge 0$ for some/any $\b \in \co_\a$ (see Corollary \ref{criterion}). By definition, $\Pi(\l) = \Pi(v, +)$ if $\l$ is $v$-small.

\begin{lem} \label{fin}
If $\l \in \ca_{\mu, b}^\tp$ is not small, then there exists $\a \in \Pi$ such that $W_{\co_\tta}$ is finite and $\l_\b \ge 1$ for $\b \in \co_\a$
\end{lem}
\begin{proof}
By Proposition \ref{polar}, there exists $v \in V^{p(b\s)}_\gen \cap Y$ such that $\l \in \ca_{\mu, b}^\tp(v)$. As $\l$ is not $v$-small, there exists $\a \in \Pi(v, +) - \Phi_{M_b}$ such that $\l_\b \ge 1$ for $\b \in \co_\a$. Moreover, $W_{\co_\tta}$ is finite by Lemma \ref{finite}.
\end{proof}

\begin{prop} \label{red-small}
For each $C \in \Irr^\tp X_\mu(b)$ there exists small $\l \in \ca_{\mu, b}^\tp$ such that $C \in \JJ_b \Irr \overline{ X_\mu^\l(b) }$.
\end{prop}
\begin{proof}
Recall the dominance order $\leq$ on $Y$ defined in \S \ref{G-setup}. For $\eta , \chi \in Y$ write $\eta \unlhd \chi$ if either $\eta \lneq \chi$ (see \S\ref{G-setup} for $\leq$) or $\eta \in W_0(\chi)$ and $\chi \le \eta$. Let $\l$ be a minimal cocharacter in the set $$\{\chi \in \ca_{\mu, b}^\tp; C \in \JJ_b \Irr \overline{ X_\mu^\chi(b) } \}$$ under the partial order $\unlhd$. We show that $\l$ is small.

Suppose $\l$ is not small. Let $\a \in \Pi$ as in Lemma \ref{fin}, and let $\tw \in \JJ_b$ be the maximal element of $W_{\co_\tta}$. By Lemma \ref{orbit}, $$\tw = \prod_{\b \in \co_\g} s_{\tilde \b},$$ where $\co_\g$ is orthogonal with $\g = \a$ if $\<\a^d, \a^\vee\> \neq -1$, and $\g = \a^d + \a$ otherwise. In particular, $\l_\b \ge 1$ for $\b \in \co_\g$. Let $\l' = \tw(\l)$. By Corollary \ref{adj-red}, $\l' \in \ca_{\mu, b}^\tp$ and $C \in \JJ_b \Irr \overline{ X_\mu^\l(b) } = \JJ_b \Irr \overline{ X_\mu^{\l'}(b) }$. Moreover, as $\co_\g$ is orthogonal, $$\l' = \tw(\l) = p(\tw)(\l - \sum_{\b \in \Phi^+ \cap \co_\g} \b^\vee).$$ If $\Phi^+ \cap \co_\g \neq \emptyset$, we have $\l' \lneq \l$ since $\<\b, \l\> = \l_\b + 1 \ge 2$ for $\b \in \Phi^+ \cap \co_\g$. Otherwise, $$\l' = p(\tw)(\l) = \l - \sum_{\b \in \co_\g} \l_\b \b^\vee > \l.$$ Thus, in either case we have $\l' \lhd \l$, contradicting the choice of $\l$. So $\l$ is small as desired.
\end{proof}

We say a root $\a \in \Phi(v, +)$ is indecomposable (in $\Phi(v, +)$) if it is not a sum of roots in $\Phi(v, +) \setminus \{\a\}$.
\begin{lem} \label{cond}
Let $v \in V^{p(b\s)}_\gen \cap Y$ be permissible. Then each root of $\Pi(v, +)$ is indecomposable in $\Phi(v, +)$.
\end{lem}
\begin{proof}
For $\a \in \Pi(v, +)$ set $Y'(v, \a) = \{\l \in Y; \l_\a = 0, \l_\b \ge 0 \text{ for } \b \in \Phi(v, +)\}$. We claim that

(a) $Y'(v, \a) \neq \emptyset$.

By assumption, there is a $v$-small cocharacter $\chi$. By definition, $\chi \in Y'(v, \g)$ for some $\g \in \co_\a$. By Lemma \ref{eta} (3), $Y'(v, \a)$ and $Y'(v, \g)$ are conjugate by $\<b\s\>$. So (a) is proved.

Suppose $\a = \sum_{i \in D} \a_i$ for some $\a \neq \a_i \in \Phi(v, +)$. Let $\l \in Y'(v, \a)$ by (a). Then $\l_\a = 0$ and $\<\a_i, \l\> \ge \l_{\a_i} \ge 0$ for $i \in D$. If $\a < 0$ is a minus simple root, then there exists $i_0 \in D$ such that $\a_{i_0} > 0$. Hence $\l_\a = \<\a, \l\> \ge \<\a_{i_0}, \l\> = \l_{\a_{i_0}} + 1 \ge 1$, which is a contradiction. If $\a > 0$ is the highest root, then there exist $i_1 \neq i_2 \in D$ such that $\a_{i_1}, \a_{i_2} > 0$. Again we have $\<\a, \l\> \ge \<\a_{i_1}, \l\> + \<\a_{i_2}, \l\> \ge 2$ and hence $\l_\a \ge 1$, which is also a contradiction. So $\a$ is indecomposable as desired.
\end{proof}

Let $J = p(b\s)(J) \subseteq \Pi$ such that the corresponding parabolic subgroup $W_J$ (generated by $s_{\tta}$ for $\a \in J$) is finite. By a standard parahoric subgroup of type $J$ we mean a subgroup of $\JJ_b$ generated by $I \cap \JJ_b$ and $W_J \cap \JJ_b$. We say a standard parahoric subgroup of type $J$ is of maximal length if the length, of the maximal element of $W_J$, is maximal among all standard parahoric subgroups of $\JJ_b$.
The following result will be proved in Appendix \ref{max-sec}.
\begin{prop} \label{max}
If $v \in V^{p(b\s)}_\gen \cap Y$ is permissible, then the parahoric subgroup of type $\Pi(v, +)$ is of maximal length.
\end{prop}

\subsection{Irreducibility implies smallness} \label{irrsma-subsec} Suppose $\l \in \ca_{\mu, b}^\tp$ is not small. Let $\a \in \Pi$ be as in Lemma \ref{fin}, and let $\tw \in W_{\co_\tta}$ be the longest element. Suppose $X_\mu^\l(b)$ is irreducible. By Lemma \ref{symm}, Lemma \ref{dom} and Corollary \ref{adj-red} (2), $\l \neq \tw(\l) \in \ca_{\mu, b}^\tp$ and $\tw X_\mu^{\tw(\l)}(b) \subseteq X_\mu^\l(b) $. Hence $\tw \overline{ X_\mu^{\tw(\l)}(b) } = \overline{ X_\mu^\l(b) }$ is also irreducible. In particular, $$\tw N_{\JJ_b}(\overline{ X_\mu^{\tw(\l)}(b) })) \tw\i = N_{\JJ_b}(\overline{ X_\mu^\l(b) }).$$ Notice that $\tw \in W^a \cap \JJ_b$ and $N_{\JJ_b}(\overline{ X_\mu^{\tw(\l)}(b) })), N_{\JJ_b}(\overline{ X_\mu^\l(b) })$ are both standard parahoric subgroups containing $I \cap \JJ_b$. Thus $\tw \in N_{\JJ_b}(\overline{ X_\mu^{\tw(\l)}(b) })) = N_{\JJ_b}(\overline{ X_\mu^\l(b) })$, which is a contradiction. So $X_\mu^\l(b)$ is not irreducible as desired.

\subsection{Smallness implies irreducibility} Now we show $X_\mu^\l(b)$ is irreducible if $\l$ is small. We need some results on permissible vectors introduced in \S \ref{def-small}.

Let $( , ): V \times V \to \RR$ be the Killing form such that $\<\b, \g^\vee\> = \frac{2(\b^\vee, \g^\vee)}{(\b^\vee, \b^\vee)}$ for $\b, \g \in \Phi$. For any $p(b\s)$-orbit $\co$ of $\Pi$ we set $r_\co = \sum_{\xi \in \co} \xi$ and $r_\co^\vee = \sum_{\xi \in \co} \xi^\vee$. Then we have the identification $r_\co = \frac{2}{(\xi^\vee, \xi^\vee)} r_\co^\vee$ via the bilinear form $( , )$, where $\xi$ is any/some root in $\co$.
\begin{lem} \label{basis}
If $v \in V^{p(b\s)}_\gen \cap Y$ is permissible, then $\{r_\co^\vee; \co \in \Pi(v, +) / \<p(b\s)\>\}$ is a basis of $V^{p(b\s)}$.
\end{lem}
\begin{proof}
By Proposition \ref{max}, $\Pi(v, +)$ is a maximal proper $p(b\s)$-stable subset of $\Pi$. Hence $\{r_\co^\vee; \co \in \Pi(v, +) / \<p(b\s)\>\}$ is linearly independent, and moreover, $$|\Pi(v, +) / \<p(b\s)\>| = |\Pi / \<p(b\s)\>| - 1 = \dim V^{p(b\s)}.$$ So the statement follows.
\end{proof}

\begin{lem} \label{indecomp}
Let $v \in V^{p(b\s)}_\gen \cap Y$ be permissible. Let $\g$ be an indecomposable root in $\Phi(v, +)$. Then there exists $\a \in \Pi(v, +)$ such that $\g - \a \in \Phi_{M_b} \sqcup \{0\}$.
\end{lem}
\begin{proof}
Suppose $\<\g, \b^\vee\> \le 0$ for $\b \in \Pi(v, +)$. Then $(r_\co^\vee, r_{\co'}^\vee) \le 0$ for $\co, \co' \in (\Pi(v, +) / \<p(b\s)\>) \cup \{\co_\g\}$. Thus the set $$\{r_\co^\vee; \co \in \Pi(v, +) / \<p(b\s)\>\} \sqcup \{r_{\co_\g}^\vee\} \subseteq \Phi(v, +)$$ is linearly independent, which contradicts Lemma \ref{basis}.

Thus $\<\g, \a^\vee\> > 0$ for some $\a \in \Pi(v, +)$. Notice that $\a$ is indecomposable in $\Phi(v, +)$ by Lemma \ref{cond}. If $\g = \a$, the statement follows. Otherwise, $\d := \a - \g$ is also a root. Suppose $\<\d, v\> \neq 0$. Then  $\a = \g + \d$ (resp. $\g = \a + (-\d)$) is indecomposable if $\<\d, v\> > 0$ (resp. $\<\d, v\> < 0$), which contradicts that $\a$ and $\g$ are indecomposable in $\Phi(v, +)$. So we have $\<\d, v\> = 0$, that is, $\d \in \Phi_{M_b}$ as desired.
\end{proof}

\begin{cor} \label{conj}
Let $v, v' \in V^{p(b\s)}_\gen \cap Y$ be permissible. Then there exists $\varepsilon \in \Omega \cap \JJ_b$ such that $\Pi(p(\varepsilon)(v), +) = \Pi(v', +)$ and hence $\Phi(p(\varepsilon)(v), +) = \Phi(v', +)$.
\end{cor}
\begin{proof}
By Proposition \ref{max}, one checks (using that $G$ is adjoint) that there exists $\varepsilon \in \Omega \cap \JJ_b$ such that $$\Pi(p(\varepsilon)(v), +) = p(\varepsilon)(\Pi(v, +)) = \Pi(v', +).$$ By replacing $v$ with $p(\varepsilon)(v)$, we may assume further $\Pi(v, +) = \Pi(v', +)$, and it remains to show $\Phi(v, +) = \Phi(v', +)$. Otherwise, there exists an indecomposable root $\g$ in $\Phi(v, +)$ such that $\<\g, v'\> < 0$. By Lemma \ref{indecomp}, there exists $\a \in \Pi(v, +) = \Pi(v', +)$ such that $\g - \a \in \Phi_{M_b} \sqcup \{0\}$. Hence $\<\g, v'\> = \<\a, v'\> > 0$, which contradicts our assumption. So $\Phi(v, +) = \Phi(v', +)$ as desired.
\end{proof}

\begin{prop} \label{small-conj}
Let $\l, \l' \in \ca_{\mu, b}^\tp$ be small cocharacters such that $\l \sim \l'$. Then $\l, \l'$ are conjugate under $\Omega \cap \JJ_b$. In particular, $X_\mu^\l(b)$ and $X_\mu^{\l'}(b)$ are conjugate by $\Omega \cap \JJ_b$.
\end{prop}
\begin{proof}
By Proposition \ref{polar}, there exist permissible vectors $v, v' \in V^{p(b\s)}_\gen \cap Y$ such that $\l, \l'$ are $v$-small and $v'$-small respectively. In particular, $v, v'$ are both permissible. By Corollary \ref{conj}, there exists $\varepsilon \in \Omega \cap \JJ_b$ such that $$p(\varepsilon)(\Phi(v', +)) = \Phi(p(\varepsilon)(v'), +) = \Phi(v, +).$$ Thus $\varepsilon(\l')$ is also $v$-small. By replacing $\l'$ with $\varepsilon(\l')$, we may assume $\l, \l'$ are both $v$-small. By Lemma \ref{Omega}, there exists $x \in \Omega_{M_b} \cap \JJ_b$ such that $x(\l) = \l'$. It suffices to show $x \in \Omega$.

First we claim that

(a) $x(\tta)$ is a simple affine root for each $\a \in \Pi(v, +)$.

Indeed, let $\g = p(x)(\a) \in \Phi(v, +)$. As $\l$ is $v$-small, we may assume $\l_\a = 0$ (by replacing $\a$ by a suitable $\<p(b\s)\>$-conjugate). Then we have $$U_{x(\tta)} = x U_\tta x\i = x t^\l U_\a(\co_\brF) t^{-\l} x\i = t^{x(\l)} U_{\g}(\co_\brF) t^{-x(\l)} \subseteq I,$$ where the last inclusion follows from that $x(\l) = \l'$  is $v$-small. So $x(\tta) \in \tPhi^+$. By Lemma \ref{cond}, $\a$ is indecomposable in $\Phi(v, +)$. Hence $\g$ is also indecomposable in $\Phi(v, +)$. Applying Lemma \ref{indecomp} we deduce that there exists $\b \in \Pi(v, +)$ such that either $\g = \b$ or $\g = \b + \d$ for some $\d \in \Phi_{M_b}$. By symmetry, $x\i(\tilde \b) \in \tPhi^+$. In the former case, $x(\tta) = \tilde \b + m$ for some $m \in \ZZ_{\ge 0}$. Then $\tta = x\i(\tilde \b) + m$, which means $m = 0$ since $\tta$ is simple and  $x\i(\tilde \b) \in \tPhi^+$. So we have $x(\tta) = \tilde \b$ as desired. In the latter case, we have $x(\tta) = \tilde \b + \tilde \d + m$ for some $m \in \ZZ_{\ge 0}$. As $\d \in \Phi_{M_b}$ and $x \in \Omega_{M_b}$, we have $x\i(\tilde \d) \in \tPhi_{M_b}^+$. Then $$\tta = x\i(\tilde \b) + x\i(\tilde \d) + m,$$ which is a contradiction since $\tta$ is simple but $x\i(\tilde \b), x\i(\tilde \d) \in \tPhi^+$. Thus (a) is proved.

By (a) we see that $x$ permutes the hyperplanes $H_\co$ for $\co \in \Pi(v, +)/ \<p(b\s)\>$, where $$H_\co = \{h \in V^{b\s}; \tta(h)= 0 \text{ for any/some } \a \in \co\}$$ whose underlining vector space is \begin{align*} V_\co &= \{r \in V^{p(b\s)}; \<\a, r\> = 0 \text{ for any/some } \a \in \co\} \\ &= \{r \in V^{p(b\s)}; \<r_{\co}, r\> = 0 = (r_\co^\vee, r)\}.\end{align*} By Lemma \ref{basis}, the subspaces $V_\co$ are distinct and their intersection $\cap_\co V_\co$ is trivial. On the other hand, as $x \in \Omega_{M_b}$, $p(x)$ is product of reflections $s_\a$ such that $ \<\a, v\> = 0 =\<\a, V^{p(b\s)}\>$. In particular, $p(x)$ acts on $V^{p(b\s)}$ trivially, which means $x$ acts on $V^{b\s}$ as a translation by some vector $\iota \in V^{p(b\s)}$. Thus $x$ fixes each $H_\co$ and hence $\iota \in \cap_\co V_\co = \{0\}$, that is, $x$ acts trivially on $V^{b\s}$. In particular, $x$ fixes the nonempty subset $\D \cap V^{b\s}$, which means $x \in \Omega \cap \JJ_b$ as desired.
\end{proof}

We recall the following result in \cite[Theorem 3.1.1]{ZZ}.
\begin{thm} \label{para}
Let $Z$ be an irreducible component of $X_\mu(b)$. Then the stabilizer of $Z$ in $\JJ_b$ is a parahoric subgroup of $\JJ_b$.
\end{thm}

\begin{cor}
Let $\l \in \ca_{\mu, b}^\tp$. Then $X_\mu^\l(b)$ is irreducible if and only if $\l$ is small.
\end{cor}
\begin{proof}
In view of \S \ref{irrsma-subsec}, it remains to show the ``if'' part. Let $\l \in \ca_{\mu, b}^\tp$ be small. Thanks to Theorem \ref{para}, there exists $C' \in \JJ_b \Irr \overline{X_\mu^\l(b)}$ whose stabilizer in $\JJ_b$ contains $I \cap \JJ_b$. Let $\l' \in \ca_{\mu, b}^\tp$ such that $C' \in \Irr \overline{X_\mu^{\l'}(b)}$. As  $I \cap \JJ_b$ fixes $C'$, and acts transitively on $\Irr X_\mu^{\l'}(b)$ (by Lemma \ref{dim}), we see that $\overline{X_\mu^{\l'}(b)} = C'$ is irreducible, and hence $\l'$ is small. Noticing that $\l \sim \l'$, we deduce by Proposition \ref{small-conj} that $X_\mu^\l(b)$ is also irreducible as desired.
\end{proof}

\subsection{Computation of the stabilizer} Suppose $C = \overline{ X_\mu^\l(b)} \in \Irr^\tp X_\mu(b)$ with $\l$ small. Notice that $\JJ_b$ is generated by $I \cap \JJ_b$, $\Omega \cap \JJ_b$, and the longest element $\tw_\a$ of $W_{\co_\tta}$ for $\a \in \Pi$ such that $W_{\co_\tta}$ is finite. by definition, $I \cap \JJ_b \subseteq N_{\JJ_b}(C)$. So $N_{\JJ_b}(C)$ is a standard parahoric subgroup of $\JJ_b$, and it remains to determine which $\tw_\a$ fixes $C$. By Lemma \ref{dom}, either $\l_\b \le -1$ for $\b \in \co_\a$ or  $\l_\b \ge 0$ for $\b \in \co_\a$. In the former case, we have $\a \notin \Pi(\l)$ and $\tw_\a C \neq C$ by Corollary \ref{adj-red} (2) and Lemma \ref{symm}. Suppose the latter case occurs. Then $\a \in \Pi(\l)$ and $\l_\b = 0$ for some $\b \in \co_\a$ since $\l$ small. So $\tw C = C$ by Corollary \ref{adj-red} (1). Therefore, $N_{\JJ_b}(C)$ is the parahoric subgroup of $\JJ_b$ generated by $I \cap \JJ_b$ and the longest element $\tw_\a$ of $W_{\co_\tta}$ for $\a \in \Pi(\l)$. Moreover, $N_{\JJ_b}(C)$ is of maximal length by Proposition \ref{max}.

\begin{appendix}

\section{Proof of Proposition \ref{max}} \label{max-sec}
We assume that $b$ is basic and $G$ is simple and adjoint.

\subsection{} For practical computation, we need to pass to the case where $G$ is absolutely simple, that is, the root system $\Phi$ of $G$ is irreducible. By assumption, $$G_{\co_{\brF}} = G_1 \times \cdots \times G_h,$$ where each $G_i$ is an absolutely simple factor of $G$ and the Frobenius automorphism $\s$ sends $G_i$ to $G_{i-1}$ for $i \in \ZZ / h\ZZ$. Let $$\pi: G_{\co_{\brF}} \to G_1$$ be the projection to the first factor, which induces an identification $$\JJ_b = \JJ_b^G \cong \JJ_{\underline b_1}^{G_1} = \JJ_{\underline b_1},$$ where $\underline b_1 = \pi(b \s(b) \cdots \s^{h-1}(b)) \in G_1(\brF)$ and the Frobenius automorphism of $G_1$ is given by $\s^h$.

The following lemma follows similarly as Corollary \ref{compare}.
\begin{lem} \label{abs-simple}
The projection $\pi$ induces a $\JJ_b$-equivariant map $$\pi: \Irr^\tp X_\mu(b) \to \sqcup_{\mu_1} \Irr^\tp X_{\mu_1}^{G_1}(\underline b_1).$$ Moreover, $N_{\JJ_b}(C) = N_{\JJ_{\underline b_1}}(\pi(C))$ for $C \in \Irr X_\mu(b)$.
\end{lem}

\subsection{} \label{subsec-max}
Now we assume $G$ is absolutely simple by Lemma \ref{abs-simple}. Moreover, we adopt the notation in \S \ref{minu-subsec}. Notice that $V^{p(b\s)}_\gen$ is an open dense subset of $V^{p(b\s)}$. Notice that the diagonal map gives an isomorphism $V^{p(b\s)} \cong (V^d)^{p(\bb\bs)}$.

Fix $v \in V^{p(b\s)}_\gen \cap Y$ and let $z, M_b, M, b_M$ be as in \S \ref{generic}. Notice that $b_M$ is a superbasic element of $M(\brF)$. We define $Y(v) = \{\l \in Y; \l_\a \ge 0, \forall \a \in \Phi(v, +)\}$.

The following lemma is a reformulation of Corollary \ref{criterion} and small cocharacters in \S \ref{def-small}.
 \begin{lem} \label{top}
Let $\bl = (\l_1, \dots, \l_d) \in Y^d$. Then we have

(1) $\bl \in \ca_{\bmu, \bb}^{G^d, \tp}$ if and only if (1) $\l_i \in Y(v)$ for $1 \le i \le d$ and (2) $z(\bl) := (z(\l_1), \dots, z(\l_d)) \in \ca_{z(\l_\bullet), {b_M}_\bullet}^{M^d, \tp}$, where ${b_M}_\bullet = (1, \dots, 1, b_M) \in M^d(\brF)$.

(2) $\bl$ is $v$-small if $\bl \in \ca_{\bmu, \bb}^{G^d, \tp}$ and for each $p(b\s)$-orbit $\co$ in $\Pi(v, +)$ we have $(\l_i)_\a = 0$ for some $\a \in \co$ and $1 \le i \le d$.
\end{lem}

\subsection{} \label{subsec-per} To prove Proposition \ref{max}, we need some properties of permissible vectors introduced in \S \ref{def-small}. Let $M' \supseteq T$ be a Levi subgroup and let $\l, \eta \in Y$. Define $$\ch_{M'}(\l, \eta) = \{\g \in \Phi_{M'}; \l_\g \ge 0,  \eta_\g \le -1 \} = -\ch_{M'}(\eta, \l).$$
\begin{lem} \label{ch}
For $\l, \eta \in Y$ we have

(0) $b\s(\l)_{p(b\s)(\g)} = \l_\g$ for $\g \in \Phi$;

(1) $\ch_{M_b}(\l, \eta) \subseteq \ch_{M_b}(\l, \chi) \cup \ch_{M_b}(\chi, \eta)$ for $\chi \in Y$;

(2) $z(\ch_{M_b}(\l, \eta)) = \ch_M(z(\l), z(\eta)))$;

(3) $p(b\s)(\ch_{M_b}(\l, \eta)) = \ch_{M_b}(b\s(\l), b\s(\eta))$.
\end{lem}
\begin{proof}
Note that (1) follows by definition, and (3) follows from (0) which is proved in Lemma \ref{eta} (3). As $z(\Phi_{M_b}^+) = \Phi_M^+$, we have $z(\l)_{z(\a)} = \l_\a$ for $\a \in \Phi_{M_b}$, from which (2) follows.
\end{proof}

\begin{cor} \label{ineq}
Let $\bl = (\l_1, \dots, \l_d) \in \ca_{\bmu, \bb}^{G^d, \tp}$. For $1 \le i \le d$ we have $$|\ch_{M_b}(\l_i, b\s(\l_i))| \le \df(b),$$ where $\df(b)$ denotes the defect of $b$.
\end{cor}
\begin{proof}
By Lemma \ref{top} (2), $z(\bl) \in \ca_{z(\l_\bullet), {b_M}_\bullet}^{M^d, \tp}$. Moreover, $b_M$ is superbasic in $M(\brF)$. By Lemma \ref{ch} we have \begin{align*} \df(b) &= \rk_F(M) \\ &= |\ch_M(z(\l_1), z(\l_2))| + \cdots + |\ch_M(z(\l_{d-1}), z(\l_d))| + |\ch_M(z(\l_d), z (b\s(\l_1)))| \\ &=|\ch_{M_b}(\l_1, \l_2)| + \cdots + |\ch_{M_b}(\l_{d-1}, \l_d)| + |\ch_{M_b}(\l_d, b\s(\l_1))| \\ &\ge |\ch_{M_b}(\l_1, \l_i)| + |\ch_{M_b}(\l_i, b\s(\l_1))| \\ &= |\ch_{M_b}(b\s(\l_1), b\s(\l_i))| + |\ch_{M_b}(\l_i, b\s(\l_1))| \\ &\ge |\ch_{M_b}(\l_i, b\s(\l_i))|, \end{align*} where $\rk_F(M)$ denote the $F$-semisimple rank of $M$, and the second equality follows from Lemma \ref{order} and Lemma \ref{dim-1}.
\end{proof}

For $\a \in \Pi(v, +)$ we define $$Y(v, \a) = \{\l \in Y(v); \l_\a = 0, |\ch_{M_b}(\l, b\s(\l))| \le \df(b)\}.$$
\begin{lem} \label{non-empty}
If $v \in V^{p(b\s)}_\gen \cap Y$ is permissible, then we have $Y(v, \a) \neq \emptyset$ for $\a \in \Pi(v, +)$.
\end{lem}
\begin{proof}
By assumption, there is a $v$-small cocharacter $\bl = (\l_1, \dots, \l_d)$. By definition, there exists a $\<p(b\s)\>$-conjugate $\g$ of $\a$ such that $(\l_i)_\g = 0$ for some $1 \le i \le d$. Moreover, we have $|\ch_{M_b}(\l_i, b\s(\l_i))| \le \df(b)$ by Lemma \ref{ineq}. So $\l_i \in Y(v, \g)$. By Lemma \ref{ch}, $Y(v, \g)$ and $Y(v, \a)$ are conjugate under $\<b\s\>$. So the statement follows.
\end{proof}

For $\a, \b \in \Phi$ we write $\a \to \b$ if there exists a sequence $\a = \g_0, \g_1, \dots, \g_r = \b$ of roots in $\Phi - \Phi_{M_b}$ such that $\g_i - \g_{i-1} \in \Phi^+$ is a simple root for $1 \le i \le r$.
\begin{lem} \label{adm}
Assume $v \in V^{p(b\s)}_\gen \cap Y$ is permissible. If $\a \to \b$ with $\a \in \Phi(v, +)$, then $\b \in \Phi(v, +)$.
\end{lem}
\begin{proof}
We can assume $\b - \a$ is a simple root. Notice that $\<\b, v\> \neq 0$ since $\b \notin \Phi_{M_b} = \Phi_{M_v}$. If $\b \notin \Phi(v, +)$, then $-\b \in \Phi(v,+)$. Thus $-\b + \a \in \Pi(v, +)$ is decomposable in $\Phi(v, +)$, contradicting Lemma \ref{cond}.
\end{proof}

\subsection{The classification}
Now we apply Lemma \ref{cond} and Lemma \ref{non-empty} to prove Proposition \ref{max} when $b$ is ramified, that is, $b \in \Omega$ and the identity $1$ are not $\s$-conjugate under $\Omega$ (noticing that $G$ is adjoint).

We argue by a case-by-case analysis on the (connected) Dynkin diagram of $\SS_0$. The simple roots $\a_i$ of $\Phi^+$ are labeled as in \cite[\S 11.4]{Hum}. If the fundamental coweight $\varpi^\vee_i$ of $\a_i$ is minuscule, we denote by $\o_i \in \Omega \cap t^{\varpi^\vee_i} W_0$ the unique length zero element. Let $\th > 0$ denote the highest root.

For classical types we fix an ambient vector space $V_0 = \oplus_{i=1}^n \RR e_i^\vee$ (of $\Phi^\vee$) and its dual $V_0^\ast = \oplus_{i=1}^n \RR e_i$ together with a pairing $\<,\>$ between $V_0$ and $V_0^\ast$ such that $\<e_i, e_j^\vee\> = \d_{i, j}$.

\subsubsection{Type $D_n$} The simple roots are $\a_i = e_i - e_{i+1}$ for $1 \le i \le n-1$ and $\a_n = e_{n-1} + e_n$.

\

Case(4.1.1): $\s = \Id$ and $b = \o_1$. Then $V^{p(b\s)}  = \oplus_{i=2}^{n-1} \RR e_i^\vee$ and $\Phi_{M_b}^+ = \{e_1 \pm e_n\}$. Suppose $\a_i \in \Phi(v, +)$ for some $2 \le i \le n-2$. Then \begin{gather*} \a_i = e_i - e_{i+1} \to e_i - e_{n-1} \to e_1 - e_{n-1}; \\ \a_i = e_i - e_{i+1} \to e_i - e_n \to e_i + e_{n-1} \to e_2 + e_{n-1}. \end{gather*} So $e_2 + e_{n-1}, e_1 - e_{n-1} \in \Phi(v, +)$ by Lemma \ref{adm}. Hence $$\th = (e_2 + e_{n-1}) + (e_1 - e_{n-1}) \in \Pi(v, +)$$ is decomposable  in $\Phi(v, +)$, contradicting Lemma \ref{cond}. Suppose $\a_1, \a_n, \a_{n-1} \in \Phi(v, +)$. Then we have $\a_n \to e_2 + e_n$ and $\a_{n-1} \to e_2 - e_n$. Hence $e_2 \pm e_n \in \Phi(v, +)$ and $\th = \a_1 + (e_2 + e_n) + (e_2 - e_n)$ is decomposable in $\Phi(v, +)$, a contradiction. Thus $\Pi(v, +)$ equals $\Pi \setminus \{-\a_1, \th\}$ or $\Pi \setminus \{-\a_{n-1}, -\a_n\}$ as desired.

\

Case(4.1.2): $n$ is odd and $b\s$ is of order 4. Let $m = (n-1)/2 \ge 2$. Then we have $V^{p(b\s)}  = \oplus_{i=2}^m \RR (e_i^\vee - e_{n+1-i}^\vee)$ and $$\Phi_{M_b}^+ = \{e_i + e_{n+1-i}; 2 \le i \le m\} \cup \{e_1 \pm e_{m+1}, e_1 \pm e_n, e_{m+1} \pm e_n\}.$$ Denote by $T \subseteq M^1$ (resp. $T \subseteq M^i$ for $2 \le i \le m$) the Levi subgroup of $M_b$ whose set of positive roots is $\{e_1 \pm e_{m+1}, e_1 \pm e_n, e_{m+1} \pm e_n\}$ (resp. $\{e_i + e_{n+1-i}\}$).

Suppose $\a_i, \a_{n-i} \in \Phi(v, +)$ for some $2 \le i \le m-1$. Then $m \ge 3$ and \begin{gather*} \a_{n-i} = e_{n-i} - e_{n-i+1} \to e_{n-i} - e_n \to e_{n-i} + e_{n-1} \to e_{m+1} + e_{n-1} \to e_{m+1} + e_{m+2}; \\ \a_i = e_i -e_{i+1} \to e_i - e_{m+1} \to e_2 - e_{m+1}; \\ \a_i = e_i -e_{i+1} \to e_i - e_{m+2} \to e_1 - e_{m+2}. \end{gather*} So $e_{m+1} + e_{m+2}, e_2 - e_{m+1}, e_1 - e_{m+2} \in \Phi(v, +)$ and $\th = (e_{m+1} + e_{m+2}) + (e_2 - e_{m+1}) + (e_1 - e_{m+2})$ is decomposable, a contradiction.

Suppose $\a_n \in \Phi(v, +)$. Then $\a_n \to e_{m+2} + e_n$ and hence $e_{m+2} + e_n \in \Phi(v, +)$, that is, $v(m+2) > 0$ as $v(n) = v(m+1) = v(1) = 0$. Thus $$-\a_{m+1} = e_{m+2} - e_{m+1}, e_{m+2} \pm e_1, e_{m+2} \pm e_n \in \Pi(v, +).$$ Let $\l \in Y(v, -\a_{m+1})$. Then $\l_{-\a_m} = 0$ and $\l_{e_{m+2} \pm e_1}, \l_{e_{m+2} \pm e_n} \ge 0$, which means $\l(m+1) = \l(m+2)$ and $$ 1 - \l(m+1) \le \l(1) \le \l(m+1), \  1 - \l(m+1) \le \l(n) \le \l(m+1) - 1.$$ It follows that $|\ch_{M^1}(\l, b\s(\l))| \ge 4$ because $$e_{m+1} \pm e_1, e_{m+1} \pm e_n \in \ch_{M^1}(\l, b\s(\l)).$$ On the other hand, one checks that $|\ch_{M^i}(\l, b\s(\l))| = 1$ for $2 \le i \le m$. Thus $$|\ch_{M_b}(\l, b\s(\l))| = \sum_{j=1}^m |\ch_{M^j}(\l, b\s(\l))| \ge m+3 > m+2 = \df(b),$$ contradicting Lemma \ref{ineq}. Thus $\Pi(v, +) = \Pi \setminus \{-\a_m, -\a_{m+1}\}$ as desired.

\

Case(4.1.3): $n$ is even and $b\s$ is of order 4. \label{prev} Let $m = n/2 \ge 2$. Then we have $V^{p(b\s)}  = \oplus_{i=2}^m \RR (e_i^\vee - e_{n+1-i}^\vee)$ and $$\Phi_{M_b}^+ = \{e_i + e_{n+1-i}; 2 \le i \le m\} \cup \{e_1 \pm e_n\}.$$ Suppose $\a_i, \a_{n-i} \in \Phi(v, +)$ for some $2 \le i \le m - 1$. Then $m \ge 3$ and \begin{gather*} \a_{n-i} = e_{n-i} - e_{n-i+1} \to e_{n-i} + e_n \to e_{n-i} + e_{n-i+2} \to e_{m} + e_{m+2}; \\ \a_i = e_i -e_{i+1} \to e_2 - e_{m+2}; \\ \a_i = e_i -e_{i+1} \to e_1 - e_m. \end{gather*} So $e_m + e_{m+2}, e_2 - e_{m+2}, e_1 - e_m \in \Phi(v, +)$ and $\th = (e_m + e_{m+2}) + (e_2 - e_{m+2}) + (e_1 - e_m)$ is decomposable in $\Phi(v, +)$, a contradiction. Suppose $\a_1, \a_{n-1}, \a_n \in \Phi(v, +)$. Then $\a_{n-1} \to e_2 - e_n$ and $\a_n \to e_2 + e_n$. So $e_2 \pm e_n \in \Phi(v, +)$ and $\th = \a_1 + (e_2 + e_n) + (e_2 - e_n)$ is decomposable in $\Phi(v, +)$, a contradiction. Therefore, $\Pi(v, +) = \Pi \setminus \{-\a_m\}$ as desired.

\

Case(4.1.4): $b\s$ is of order 2 and $b \in \{\o_{n-1}, \o_n\}$. Let $m = \lfloor n/2 \rfloor \ge 2$. Then we have $V^{p(b\s)}  = \oplus_{i=1}^m \RR (e_i^\vee - e_{n+1-i}^\vee)$ and $$\Phi_{M_b}^+ = \{e_i + e_{n+1-i} \in \Phi; 1 \le i \le m\}.$$ Suppose $\a_i, \a_{n-i} \in \Phi(v, +)$ for some $2 \le i \le m$. Then \begin{gather*} \a_{n-i} = e_{n-i} - e_{n-i+1} \to e_{n-i} + e_n \to e_2 + e_n; \\ \a_i = e_i -e_{i+1} \to e_1 - e_n. \end{gather*} So $e_2 + e_n, e_1 - e_n \in \Phi(v, +)$ and $\th = (e_2 + e_n) + (e_1 - e_n)$ is decomposable in $\Phi(v, +)$, a contradiction. It is also impossible that $\a_1, \a_{n-1}, \a_n \in \Phi(v, +)$ as in Case(4.1.3). Therefore, we deduce that $\Pi(v, +)$ equals $\Pi \setminus \{-\a_{n-1}, -p(b\s)(\a_{n-1})\}$ or $\Pi \setminus \{-\a_n, -p(b\s)(\a_n)\}$ as desired.

\subsubsection{Type $B_n$} The simple roots are $\a_i = e_i - e_{i+1}$ for $1 \le i \le n-1$ and $\a_n = e_n$. We can assume $\s =\Id$ and $b = \o_1$. In this case, $V^{p(b\s)}  = \oplus_{i=2}^n \RR e_i^\vee$ and $\Phi_{M_b}^+ = \{e_1\}$. Suppose $\a_i \in \Phi(v, +)$ for some $2 \le i \le n-1$. Then $$\a = e_i - e_{i+1} \to e_2 - e_n \to e_2 \to e_2 + e_n \to e_1 + e_n.$$ So $e_2 - e_n, e_1 + e_n \in \Phi(v, +)$ and $\th = (e_2 - e_n) + (e_1 + e_n)$ is decomposable in $\Phi(v, +)$, a contradiction.

Suppose $\a_1 \in \Phi(v, +)$. Then $\a_1 \to e_1 - e_n \in \Phi(v, +)$, which means $v(n) < 0$ as $v(1) = 0$. So we have $-\a_n = -e_n, \pm e_1 - e_n \in \Phi(v, +)$. Let $\l \in Y(v, -\a_n)$. Then $\l_{-\a_n} = 0$ and $\l_{\pm e_1 - e_n} \ge 0$, which means $\l(n) = 0$, $\l(1) - \l(n) \ge 1$ and $- \l(1) - \l(n) \ge 0$, a contradiction. Thus $\Pi(v, +) = \Pi \setminus \{-\a_n\}$ as desired.

\subsubsection{Type $C_n$} The simple roots are $\a_i = e_i - e_{i+1}$ for $1 \le i \le n-1$ and $\a_n = 2e_n$. We can assume $\s =\Id$ and $b = \o_n$. Let $m = \lfloor n/2 \rfloor \ge 1$. Then $V^{p(b\s)}  = \oplus_{i=1}^m \RR (e_i^\vee - e_{n+1-i}^\vee)$ and $$\Phi_{M_b}^+ = \{e_i + e_{n+1-i}; 1 \le i \le m + 1\}.$$

Case(4.3.1): $n = 2m$. Suppose $\a_i, \a_{n-i} \in \Phi(v, +)$ for some $1 \le i \le m-1$. Then $m \ge 2$ and \begin{gather*}  \a_{n-i} = e_{n-i} - e_{n-i+1} \to e_{n-i} + e_n  \to e_m + e_n; \\ \a_i = e_i - e_{i+1} \to e_1 - e_m \to e_1 - e_n \end{gather*} So $e_m + e_n, e_1 - e_m, e_1 - e_n \in \Phi(v, +)$ and $\th = (e_m + e_n) + (e_1 - e_m) + (e_1 - e_n)$ is decomposable in $\Phi(v, +)$, a contradiction.

Suppose $\a_n, -\th \in \Phi(v, +)$. Then $\a_n = 2e_n \to 2e_{m+1}$ and $-\th \to -2e_m$, which means $2e_{m+1}, -2e_m \in \Phi(v, +)$. Let $\l \in Y(v, -\a_m)$. Then $\l_{-\a_m} = 0$ and $\l_{2e_{m+1}}, \l_{-2e_m} \ge 0$, which means $\l(m+1) = \l(m)$, $\l(m+1) > 0$, $\l(m) \le 0$, a contradiction. Thus $\Pi(v, +) = \Pi \setminus \{-\a_m\}$ as desired.

\

Case(4.3.2): $n = 2m+1$. Suppose $\a_i, \a_{n-i} \in \Phi(v, +)$ for some $1 \le i \le m$. Then \begin{gather*} \a_{n-i} = e_{n-i} - e_{n-i+1} \to e_{n-i} + e_n  \to e_{m+1} + e_{m+2}; \\ \a_i = e_i - e_{i+1} \to e_1 - e_{m+1} \to e_1 - e_{m+2}. \end{gather*} So $e_{m+1} + e_{m+2}, e_1 - e_{m+1}, e_1 - e_{m+2} \in \Phi(v, +)$ and $\th = (e_{m+1} + e_{m+2}) + (e_1 - e_{m+1}) + (e_1 - e_{m+2})$ is decomposable in $\Phi(v, +)$, a contradiction. Thus $\Pi(v, +) = \Pi \setminus \{-\a_n, \th\}$ as desired.

\subsubsection{Type $A_{n-1}$} The simple roots are $\a_i = e_i - e_{i+1}$ for $1 \le i \le n-1$. Let $\varsigma_0$ be the automorphism exchanging $\a_i$ and $\a_{n-i}$ for $1 \le i \le n-1$.

\

Case(4.4.1): $\s = \Id$. Suppose $\<b\> = \<\o_h\>$ for some $1 \le h \le n-1$ dividing $n$. Then $$\Phi_{M_b}^+ = \{e_i - e_j \in \Phi^+; i - j \in h\ZZ\}.$$ If $h = 1$, then $v = 0$ and $\Pi(v, +) = \emptyset$ as desired. Suppose $h \ge 2$ and we can assume $\a_1 \in \Pi(v, +)$. If $\a_i \in \Phi(v, +)$ for some $2 \le i \le h$, then \begin{gather*} \a_1 = e_1 - e_2 \to e_1 - e_i  \to e_1 - e_h \\ \a_i = e_i - e_{i+1} \to e_i - e_{h+1}.\end{gather*} So $e_1 - e_{h+1}, e_1 - e_h$ and their $\<p(b\s)\>$-conjugates are contained in $\Phi(v, +)$. Hence $$\th = (e_1 - e_{h+1}) + (e_{h+1} - e_{2h+1}) + \cdots + (e_{n-2h+1} - e_{n-h+1}) + (e_{n-h+1} - e_n)$$ is decomposable in $\Phi(v, +)$, a contradiction. Thus $\Pi(v, +) = \Pi \setminus \co$, where $\co$ is any $p(b\s)$-orbit of $\Pi$.

\

Case(4.4.2): $\s = \varsigma_0$, $b = \o_1$ and $n \ge 4$ is even. Let $m = n/2 \ge 2$. Then we have $$\Phi_{M_b}^+ = \{e_1 - e_{m+1}\}.$$ Suppose $\a_i, \a_{n+1-i} \in \Phi(v, +)$ for some $2 \le i \le m-1$, then $m \ge 3$ and \begin{gather*} \a_i = e_i - e_{i+1} \to e_i - e_{n+1-i}  \to e_1 - e_{n+1-i} \\ \a_{n+1-i} = e_{n+1-i} - e_{n+2-i} \to e_{n+1-i} - e_n.\end{gather*} So $e_1 - e_{n+1-i}, e_{n+1-i} - e_n \in \Phi(v, +)$ and $\th = (e_1 - e_{n+1-i}) + (e_{n+1-i} - e_n)$ is decomposable in $\Phi(v, +)$, a contradiction. Suppose $\a_1, \a_m \in \Phi(v, +)$, then $\a_1 \to e_1 - e_m$ and $\a_m \to e_m - e_n$. So $\th = (e_1 - e_m) + (e_m - e_n)$ is decomposable in $\Phi(v, +)$, a contradiction. Thus $\Pi(v, +)$ equals $\Pi \setminus \{-\a_m, -\a_{m+1}\}$ or $\Pi \setminus \{-\a_1, \th\}$ as desired.

\subsection{Type $E_6$} The simple roots $\a_i$ for $1 \le i \le 6$ are labeled as in \cite[\S 11.4]{Hum}. We can assume $\s = \Id$ and $b = \o_1$. Then we have $$V^{p(b\s)}  = \RR \a_4^\vee \oplus \RR(\a_2^\vee + \a_3^\vee + 2\a_4^\vee + \a_5^\vee).$$ Suppose $\a = \a_i \in \Phi(v, +)$ with $i=2$ or $i=4$. Let $\b = \a_2 + \a_3 + \a_4 + \a_5$ and $\g = \b + \a_1 + \a_4 + \a_6$. Then $$\a \to \b \to \b + \a_4 \to \g.$$ So $\b, \g \in \Phi(v, +)$ and $\th = \b + \g$ is decomposable in $\Phi(v, +)$, a contradiction. Thus $\Pi(v, +) = \Pi \setminus \{-\a_1, -\a_6, \th\}$ as desired.

\subsection{Type $E_7$} The simple roots $\a_i$ for $1 \le i \le 7$ are labeled as in \cite[\S 11.4]{Hum}. We can assume $\s = \Id$ and $b = \o_7$. Then we have $$\Phi_{M_b}^+ = \{\g + \a_3, \g + \a_5 + \a_6 - \a_1, \g + \a_5\},$$ where $\g = \a_1 + \a_2 + \a_3 + 2\a_4 + \a_5 + \a_6 + \a_7$. Suppose $\a = \a_i \in \Phi(v, +)$ for some $1 \le i \le 6$. Let $\xi = \g + \a_3 + \a_5 - \a_7$. Then $\g - \b, \xi - \b \notin \ZZ_{\ge 0} \Pi_0$ for $\b \in \Phi_{M_b}^+$, which implies that \begin{align*} \a \to \g \text{ and } \a \to \xi. \end{align*} So $\g, \xi \in \Phi(v, +)$ and $\th = \g + \xi$ is decomposable in $\Phi(v, +)$, a contradiction. Thus $\Pi(v, +) = \Pi \setminus \{-\a_7, \th\}$ as desired.

\section{Proof of Proposition \ref{superbasic}}
Let $b \in \Omega$ be basic. Let $J \subseteq \SS_0$ be a minimal $\s$-stable subset such that $[b] \cap M_J(\brF) \neq \emptyset$. Then $[b] \cap M_J(\brF)$ is a superbasic $\s$-conjugacy class of $M(\brF)$. Suppose there is another minimal $\s$-stable subset $J' \subseteq \SS_0$ such that $[b] \cap M_{J'}(\brF) \neq \emptyset$. To prove Lemma \ref{superbasic}, we have to show that $J = J'$. Choose $x \in \Omega_J$ and $x' \in \Omega_{J'}$ such that $x, x' \in [b]$. Let ${}^J W_0^J$ be the set of elements $w \in W_0$ which are minimal in its double coset $W_{J'} w W_J$. For $u \in W_0$ we set $\supp_\s(u) = \cup_{i \in \ZZ} \s^i(\supp(u)) \subseteq \SS_0$, where $\supp(u) \subseteq \SS_0$ is the set of simple reflections that appear in some/any reduced expression of $u$.

Following \cite{HN}, we say $\tw \in \tW$ is $\s$-straight if $$\ell(\tw \s(\tw) \cdots \s^{n-1}(\tw)) = n \ell(\tw) \text{ for } n \in \ZZ_{\ge 1}.$$ Moreover, we say a $\s$-conjugacy class of $\tW$ is straight if it contains some $\s$-straight element. By \cite[Proposition 3.2]{HN}, the $\s$-conjugacy classes of $x$ and $x'$ are straight. Moreover, as $x, x' \in [b]$, these two straight $\s$-conjugacy classes coincide by \cite[Theorem 3.3]{HN}. Thus there exists $\tw \in \tW$ such that $\tw x = x' \s(\tw)$. Write $p(\tw) = u z w\i$ with $u \in W_{J'}$, $w \in W_J$ and $z \in {}^{J'} W_0^J$. By taking the projection $p$, we have $$z w\i p(x) \s(w) = u\i p(x') \s(u) \s(z).$$ Notice that $z, \s(z) \in {}^{J'} W_0^J$. Moreover, we have $\supp_\s(w\i p(x) \s(w)) = J$ and $\supp_\s(u\i p(x') \s(u)) = J'$ by the minimality of $J$ and $J'$. This means that $z = \s(z)$ and $z J z\i = J'$. So Proposition \ref{superbasic} follows from the following lemma.
\begin{lem}
Let $J \subseteq \SS_0$ be a minimal $\s$-stable subset such that $[b] \cap M_J(\brF) \neq \emptyset$. If there exists $z = \s(z) \in W_0^J$ such that $z J z\i \subseteq \SS_0$, then $z J z\i = J$.
\end{lem}
\begin{proof}
It suffices to consider the case where $G = G_\ad$ and $\SS_0$ is connected. If $b$ is unramified, that is, $1 \in [b]$, then we can take $J = \emptyset$ and the statement is trivial. So we assume that $b$ is not unramified. By the discussion above, it suffices to show the statement for some fixed $J$, and we can take $J$ as follows. Let $v$ be a generic point of $Y_\RR^{p(b\s)}$, that is, if $\<\a, v\> = 0$ for some $\a \in \Phi$, then $\<\a, Y_\RR^{p(b\s)}\> = 0$. Then we take $J$ to be the set of simple reflections $s$ such that $s(\bar v) = \bar v$, where $\bar v$ is the unique dominant $W_0$-conjugate of $v$. Then $[b] \cap M_J(\brF)$ is a superbasic $\s$-conjugacy class of $M_J(\brF)$ by \cite[Lemma 3.1]{HN}.

Case(1): $\SS_0$ is of type $A_{n-1}$ for $n \ge 2$. Take the simple roots as $\a_i = e_i - e_{i+1}$ for $1 \le i \le n-1$. Let $\o_1$ be the generater of $\Omega \cong \ZZ / n\ZZ$ such that $\o_1 \in t^{\varpi_1^\vee} W_0$, where $\varpi_1^\vee$ is the fundamental coweight corresponding to the simple root $\a_1$. Assume $b = \o_1^m$ for some $m \in \ZZ$.

Case(1.1): $\s = \Id$. Let $h$ be the greatest common divisor of $m$ and $n$, and $f = n / h$. Then we can take $$J=\{s_{i + jf}; 1 \le i \le f-1, 0 \le j \le h-1 \}.$$ Here, and in the sequel, $s_i$ denotes the simple reflection corresponding to the simple root $\a_i$. By assumption, $z$ sends each of the subsets $$ D_j = \{1+jf, 2+jf, \dots, (j+1)f\}, \quad 0 \le j \le h-1$$ to a subset of the form $$\{k+1, k+2, \cdots k+f\} \subseteq \{1, 2, \cdots, n\}.$$ This implies that $z$ permutes the sets $D_j$ for $0 \le j \le h-1$. In particular, $z J z\i = J$ as desired.

Case(1.2): $\s$ is of order 2. In this case, $\s$ sends $\a_i$ to $\a_{n-i}$ for $1 \le i \le n-1$. As $b$ is not unramified, $n$ is even. Moreover, we can take $b = \o_1$ and $$J= \{s_{n/2}\}.$$ Noticing that $z J z\i \subseteq \SS_0$ is $\s$-stable and that $s_{n/2}$ is the unique simple reflection fixed by $\s$, we deduce that $z J z\i =J$ as desired.

Case(2): $\SS_0$ is of type $B_n$ for $n \ge 2$. Then $\s = \Id$ and $b$ is of order 2 (since it is not unramified). Take the simple roots as $\a_n = e_n$ and $\a_i = e_i - e_{i+1}$ for $1 \le i \le n-1$. Then we can take $$J = \{s_n\}.$$ Noticing that $\a_n$ is the unique short simple root, we have $z(\a_n) = \a_n$ and $z J z\i = J$ as desired.

Case(3): $\SS_0$ is of type $C_n$ for $n \ge 3$. Then $\s = \Id$ and $b$ is of order 2. Take the simple roots as $\a_n = 2e_n$ and $\a_i = e_i - e_{i+1}$ for $1 \le i \le n-1$. We can take $$J = \{s_1, s_3, \dots, s_{2 \lfloor \frac{n-1}{2} \rfloor + 1 }\}.$$ If $n$ is odd, $J$ corresponds to the unique orthogonal subset of $(n+1)/2$ simple roots, which means $z J z\i = J$ as desired. If $n$ is even, $J$ corresponds to the unique orthogonal subset of $n/2$ short simple roots, which also means $z J z\i = J$ as desired.

Case(4): $\SS_0$ is of type $D_n$ for $n \ge 4$. Take the simple roots as $\a_n = e_{n-1} + e_n$ and $\a_i = e_i - e_{i+1}$ for $1 \le i \le n-1$. As $b$ is not unramified, we have $\s^2 = 1$. The Weyl group $W_0$ is the set of permutations $w$ of $\{\pm 1, \dots, \pm n\}$ such that $z(\pm i) = \pm z(i)$ for $1 \le i \le n$ and $\sgn(w) = 0 \in \ZZ/ 2 \ZZ$, where $$\sgn (w) = | \{1 \le i \le n; i w(i) < 0\} | \mod 2.$$

Case(4.1): $\s = \Id$. If $b \in t^{\varpi_1^\vee} W_0$, we can take $$J = \{s_{n-1}, s_n\}.$$ As $z J z\i \subseteq \SS_0$, $z$ preserves the set $\{\pm(n-1), \pm n\}$ and hence $z J z\i = J$ as desired. If $b \in t^{\varpi_n^\vee} W_0$, we can take \begin{align*} J = \begin{cases} J_1 := \{s_1, s_3 \cdots, s_{n-3}, s_{n-1}\}, &\text{ if $n$ is even, $\frac{n}{2}$ is even; } \\  J_2 := \{s_1, s_3 \cdots, s_{n-3}, s_n\}, & \text{ if $n$ is even, $\frac{n}{2}$ is odd; } \\ J_0 := \{s_1, s_3 \cdots, s_{n-2}, s_{n-1}, s_n\}, & \text{ otherwise.} \end{cases} \end{align*} Suppose $n, n/2$ are even and that $J \neq z J z\i \subseteq \SS_0$. Then $z J z\i = z J_1 z\i = J_2$ because $J_1, J_2$ correspond to the only two maximal orthogonal subset of simple roots which do not contain $\{s_{n-1}, s_n\}$. By composing $z$ with a suitable element in the symmetric group of $\{1, \dots, n\}$, we can assume that $z(\a_{1+2j}) = \a_{1+2j}$ for $0 \le j \le n/2-2$ and $z(\a_{n-1}) = \a_n$. This implies that $\sgn(z) = 1$, which is a contradiction as desired. The case where $n$ is even and $n/2$ is odd follows in a similar way. Suppose $n$ is odd. Then $J = J_0$ is the unique Dynkin subdiagram of $\SS_0$ which is of type $(A_1)^{\frac{n-3}{2}} \times A_3$. So $z J z\i = J$ as desired.

Case(4.1): $\s$ is of order 2. By symmetry, we can assume $\s(\a_n) = \a_{n-1}$. As $b$ is not unramified, we can assume $b \in t^{\varpi_n^\vee} W_0$. We can take  \begin{align*} J = \begin{cases} \{s_1, s_3 \cdots, s_{n-3}, s_{n-1}, s_n\}, &\text{ if $n$ is even; } \\  \{s_1, s_3 \cdots, s_{n-2}\}, & \text{ otherwise.} \end{cases} \end{align*} If $n$ is even, then $J$ corresponds to the unique orthogonal subset of $(n+2)/2$ simple roots. So $z J z\i = J$ as desired. If $n$ is odd, then $J$ corresponds to the unique orthogonal $\s$-stable subset of $(n-1)/2$ simple roots except $\a_{n-1}$ and $\a_n$. So $z J z\i = J$ as desired.

Case(5): $\SS_0$ is of type $E_6$. As $b$ is not unramified, $\s = \Id$ and we can assume $b \in t^{\varpi_1^\vee} W_0$. Here, and in the sequel, we using the labeling of $E_6$ and $E_7$ as in \cite[\S 11]{Hum}. We can take $$J = \{s_1, s_3, s_5, s_6\}.$$ Then $z J z\i = J $ since $J \subseteq \SS_0$ is the unique Dynkin subdiagram of type $A_2 \times A_2$.

Case(6): $\SS_0$ is of type $E_7$. Then $\s = \Id$ and $b \in t^{\varpi_7^\vee} W_0$. We can take $$J = \{s_2, s_5, s_7\}.$$ Then the statement is verified by computer or by the Lusztig-Spaltenstein algorithm. The proof is finished.
\end{proof}

\end{appendix}

\end{document}